\documentclass[14pt,reqno]{amsart}
\usepackage{amsmath,amsthm,amssymb,amsfonts,amscd}
\usepackage{mathrsfs}
\usepackage{bbm}
\usepackage{bbding}
\usepackage{graphicx,latexsym}
\usepackage[backref=page]{hyperref}
\usepackage{hyperref}\hypersetup{colorlinks=true, linkcolor=black}
\usepackage{geometry}
\usepackage{color}
\usepackage{xcolor}
\usepackage{picture,epic}
\usepackage{tikz}

\numberwithin{equation}{section}

\theoremstyle{plain}
\newtheorem{theorem}{Theorem}[section]
\newtheorem{lemma}[theorem]{Lemma}

\newtheorem{proposition}[theorem]{Proposition}

\theoremstyle{definition}
\newtheorem{conjecture}{Conjecture}

\theoremstyle{remark}
\newtheorem{remark}[theorem]{Remark}

\renewcommand{\Re}{\operatorname{Re}}
\renewcommand{\Im}{\operatorname{Im}}
\newcommand{\vol}{\operatorname{vol}}
\newcommand{\reg}{\operatorname{reg}}

\newcommand{\sign}{\operatorname{sign}}
\newcommand{\supp}{\operatorname{supp}}

\newcommand{\sym}{\operatorname{sym}}

\newcommand{\GL}{\operatorname{GL}}
\newcommand{\SL}{\operatorname{SL}}
\newcommand{\PSL}{\operatorname{PSL}}
\newcommand{\X}{\mathbb{X}}
\newcommand{\W}{\mathcal{W}}
\newcommand{\E}{\Xi}
\newcommand{\Kl}{\operatorname{Kl}}
\newcommand{\even}{\operatorname{even}}
\newcommand{\rel}{\operatorname{rel}}
\renewcommand{\mod}{\operatorname{mod}}

\newcommand{\dd}{\mathrm{d}}
\newcommand{\Res}{\mathop{\operatorname{Res}}}

\makeatletter
\def\@tocline#1#2#3#4#5#6#7{\relax
  \ifnum #1>\c@tocdepth 
  \else
    \par \addpenalty\@secpenalty\addvspace{#2}%
    \begingroup \hyphenpenalty\@M
    \@ifempty{#4}{%
      \@tempdima\csname r@tocindent\number#1\endcsname\relax
    }{%
      \@tempdima#4\relax
    }%
    \parindent\z@ \leftskip#3\relax \advance\leftskip\@tempdima\relax
    \rightskip\@pnumwidth plus4em \parfillskip-\@pnumwidth
    #5\leavevmode\hskip-\@tempdima
      \ifcase #1
       \or\or \hskip 1em \or \hskip 2em \else \hskip 3em \fi%
      #6\nobreak\relax
    \hfill\hbox to\@pnumwidth{\@tocpagenum{#7}}\par
    \nobreak
    \endgroup
  \fi}
\makeatother

\begin{document}

\title[Variance for cubic moment]
{Quantum variance for cubic moment of Hecke--Maass cusp forms and Eisenstein series}

\author{Bingrong Huang}
\address{Data Science Institute and School of Mathematics \\ Shandong University \\ Jinan \\ Shandong 250100 \\China}
\email{brhuang@sdu.edu.cn}

\author{Liangxun Li}
\address{Data Science Institute and School of Mathematics \\ Shandong University \\ Jinan \\ Shandong 250100 \\China}
\email{lxli@mail.sdu.edu.cn}

\date{\today}
\keywords{Quantum variance, cubic moment, moments of $L$-functions, 
Hecke--Maass cusp form, Eisenstein series, random wave conjecture.}

\subjclass[2020]{11F12, 11F67, 58J51, 81Q50}


\begin{abstract}
In this paper, we give the upper bounds on the variance for cubic moment of Hecke--Maass cusp forms and Eisenstein series respectively. 
For the cusp form case, the bound comes from a large sieve inequality for symmetric cubes. We also give some nontrivial bounds for higher moments of symmetric cube $L$-functions.
For the Eisenstein series case, the upper bound comes from Lindel\"of-on-average type bounds for various $L$-functions. In particular, we establish the sharp upper bounds for the fourth moment of $\GL(2)\times \GL(2)$ $L$-functions and  the eighth moment of $\GL(2)$ $L$-functions around special points $1/2+it_j$.
Our proof is based on the work of Chandee and Li \cite{C-L20} about bounding the second moment of $\GL(4)\times \GL(2)$ $L$-functions.
\end{abstract}

\maketitle

\tableofcontents


\section{Introduction}
Understanding the mass distribution of automorphic forms is a central problem in the theory of quantum chaos. A common approach involves studying the Laplacian eigenfunctions on the modular surface $\X=\SL_2(\mathbb{Z})\backslash \mathbb{H}$, where $\mathbb{H}=\{z=x+iy\in \mathbb{C}: y>0\}.$
This is a finite-area hyperbolic surface equipped with the hyperbolic measure $\mu(z):=\frac{\dd x\dd y}{y^2}$ and the inner product $\langle f, g\rangle:=\int_{\X}f(z)\overline{g(z)}\dd \mu(z)$ for $L^2(\X)$.
The spectrum of the Laplacian operator $\Delta_{\mathbb{H}}:=-y^2(\frac{\partial^2}{\partial x^2}+\frac{\partial^2}{\partial y^2})$ on $\X$ decomposes into three part: the constants, the space of cusp forms,  and the space of Eisenstein series. 
Within the cusp forms, there is an orthonormal basis $\{\phi_j\}_{j\geq 1}$ of Hecke--Maass forms which are real valued joint of eigenfunctions of both the Laplacian operator and all Hecke operators.
The Eisenstein series $E_t(z):=E(z, 1/2+it)$ (for $t\in \mathbb{R}$) constitute the continuous spectrum of $\Delta_{\mathbb{H}}$. 
It is believed that these non-constant eigenfunctions on $\X$ are modeled by random waves and have a Gaussian value distribution as the eigenvalue tends to infinity. 
This motivates the study of Gaussian moments conjectures for both cusp forms and the Eisenstein series.
\begin{conjecture}\label{conj:Gauss_cuspform}
  Fix a smooth compactly supported function  $\psi$ on $\X$. (i.e. $\psi\in \mathcal{C}_c^{\infty}(\X).$) 
  Let $\{\phi_j\}_{j\geq 1}$ be an orthonormal basis of Hecke--Maass forms on $\X$. 
  Each $\phi_j$ has the spectral parameter $t_j$. Then for $n\in \mathbb{Z}_{\geq 1}$, we have
  \begin{equation}\label{eqn:Gauss_cusp}
  \int_{\X}\psi(z)\phi_j(z)^n\dd \mu(z)=\frac{c_n}{\vol(\X)^{\frac{n}{2}}}\int_{\X}\psi(z)\dd \mu(z)+o(1), \text{ as } t_j\rightarrow \infty
  \end{equation}
  where $\vol(\X)=\frac{\pi}{3}$, $c_n$ is the $n$-th moment of the normal distribution $\mathcal{N}(0,1)$, specifically, 
  \[
  c_n=\frac{1}{\sqrt{2\pi}}\int_{-\infty}^{\infty}x^ne^{-\frac{x^2}{2}}\dd x=
  \begin{cases}
        (n-1)!!, & \mbox{if $n$ is even}, \\
        0, & \mbox{if $n$ is odd}.
      \end{cases}
  \]
\end{conjecture}
This conjecture is easily to be proved when $n=1$.  The case of $n=2$ is called QUE. It was proposed by Rudnick and Sarnak \cite{RS94} and was settled by Lindenstrauss \cite{Lin06QUE} and Soundararajan \cite{Sound10QUE}.  For $n=3$, Watson \cite{Watson} proved the case of $\psi\equiv1$.  Later, Huang \cite{Huang24} solved this cubic moment problem for general $\psi\in \mathcal{C}_c^{\infty}(\X).$ For $n=4$, there are some remarkable results only with $\psi\equiv1$ so far. For example, Buttcane and Khan \cite{BK17L4} showed the asymptotic formula for the $L^4$-norm $\Vert\phi_j\Vert_4$ conditionally on GLH. Humphries and Khan \cite{HK22} proved a strong upper bound on this $L^4$-norm. Recently, Ki \cite{Ki23L4} showed the sharp upper bound $\Vert\phi_j\Vert_4\ll t_j^\varepsilon.$  

In the Eisenstein series case, let
$
\xi(2s)=\pi^{-s}\Gamma(s)\zeta(2s), e^{i\theta(t)}=\frac{\xi(1+2it)}{|\xi(1+2it)|}.
$
Then for $t\geq 2$, $\frac{e^{i\theta(t)}E_t(z)}{\sqrt{\log t}}$ is real and should exhibit statistics that are asymptotically Gaussian with mean $0$ and standard deviation $\sqrt{\frac{6}{\pi}}$. (See \cite[\S 7.3]{HR92}).
By a suitable normalization, we can formulate the following Gaussian moments conjecture for the Eisenstein series.
\begin{conjecture}\label{conj:Gauss_Es}
  Fix a smooth compactly supported function  $\psi$ on $\X$. Let $t\geq 2$. Then for $n\in \mathbb{Z}_{\geq 1}$, we have
  \begin{equation}\label{eqn:Gauss_Es}
  \int_{\X}\psi(z)\Big(\frac{e^{i\theta(t)}E_t(z)}{\sqrt{2\log t}}\Big)^n\dd \mu(z)=\frac{c_n}{\vol(\X)^{\frac{n}{2}}}\int_{\X}\psi(z)\dd \mu(z)+o(1),
  \text{ as } t\rightarrow \infty
  \end{equation}
  where $c_n$ is defined in Conjecture \ref{conj:Gauss_cuspform}.
\end{conjecture}

The rotation $e^{in \theta(t)}$ in \eqref{eqn:Gauss_Es} is dispensable. Since that if $n$ is odd, the right hand side of  \eqref{eqn:Gauss_Es} vanishes as $t\rightarrow \infty$. If $n=2k$ is even, \eqref{eqn:Gauss_Es} is equivalent to
\[
 \lim_{t \to \infty}\frac{1}{(\log t)^{k}}\int_{\X}\psi(z)|E_t(z)|^{2k}\dd \mu(z)=\Big(\frac{6}{\pi}\Big)^k(2k-1)!!\int_{\X}\psi(z)\dd \mu(z).
\]

Similarly, $n=1$ is easy to verify. The case of $n=2$ is called QUE for Eisenstein series $E_t$ which was proven by Luo and Sarnak \cite{L-S95}. The case of $n=3$ was proven by Guo \cite{Guo24} recently. For $n=4$, Djankovi\'c and Khan did a series of works \cite{D-K18}, \cite{D-K20L4reg}, \cite{D-K24L4truncate} on the regularized and truncated version of $L^4$-norm for $E_t$. Both versions of $L^4$-norm confirm the asymptotic behaviour.

Although the above asymptotics on moments agree with random wave conjecture for $\phi_j$ and normalized $E_t$, the error terms in Conjecture \ref{conj:Gauss_cuspform} and Conjecture \ref{conj:Gauss_Es} should not be expected to be sharp. For example, under GLH, we can prove
\[
\int_{\X}\psi(z)\phi_j(z)^2\dd \mu(z)=\frac{1}{\vol(\X)}\int_{\X}\psi(z)\dd \mu(z)+O_{\psi, \varepsilon}(t_j^{-\frac{1}{2}+\varepsilon})
\]
for any $\varepsilon>0$. In fact, the above error term is closely connected with the strength of subconvexity bound for $L(1/2, \sym^2\phi_j\times \phi_k)$ in the spectral aspect.
For cubic moment, Huang \cite{Huang24} proved a power saving result
\[
\int_{\X}\psi(z)\phi_j(z)^3\dd \mu(z)\ll_{\psi, \varepsilon} t_{j}^{-\frac{1}{12}+\varepsilon},
\]
by establishing the Lindel\"of on average bound for first moment of $\GL(3)\times \GL(2)$ $L$-functions in short intervals.
Later, it is improved by Guo \cite{Guo24} with power saving $O(t_{j}^{-\frac{1}{6}+\varepsilon})$. 
For the Eisenstein series case, this exponent can be improved to $-\frac{1}{3}$ in \cite{Guo24}.

Base on the above observation, we turn to consider the variance estimate for Gaussian moments conjecture, that is 
\[
\sum_{f\in \mathcal{F}}\Big|\int_{\X}\psi(z)f(z)^n\dd \mu(z)-\frac{c_n}{\vol(\X)^{\frac{n}{2}}}\int_{\X}\psi(z)\dd \mu(z)\Big|^2=o_{\psi}\big(|\mathcal{F}|\big)
\]
as $|\mathcal{F}|\to +\infty$, where $\mathcal{F}$ is a suitable spectral family of normalized $f$. 
When $n=2$ and $\psi\in \mathcal{C}_c^{\infty}(\X)$, the above estimate is called the Quantum Ergodicity for $f$.
It is first proved by Zelditch \cite{Zel91} with $f$ be the Hecke--Maass cusp forms and $\mathcal{F}=\{f:t_f\leq T\}.$ He obtained a bound $O_{\psi}\big(\frac{T^2}{\log T}\big).$ The error bound is improved by Luo and Sarnak \cite{L-S95} to $O_\psi(T^{1+o(1)})$, which is essentially optimal.
Zhao \cite{Zhao10} obtained the asymptotic formula for the variance with harmonic weight and smooth weight roughly like $|t_f-T|\leq T^{1-\varepsilon}.$
Later, Jung \cite{Jung16} proved the result in short spectral interval $|t_f-T|\leq T^{\frac{1}{3}}$ with error bound $T^{\frac{1}{3}+o(1)}$. 
For $f$ be the Eisenstein series, Huang \cite{Huang21} proved the asymptotic formulas for the quantum variance for matrix coefficients of observables. 
 
This paper focus on the case of $n=3$. We give the variance estimates for $\langle\psi, \phi^3\rangle$ with $\psi\equiv 1$ and $\langle \psi, E_t^3\rangle$
with $\psi\in \mathcal{C}_c^{\infty}(\X).$

\subsection{The variance for cubic moment of Hecke--Maass cusp forms}\label{subsec:VarHM}
Let $\phi\in \{\phi_j\}_{j\geq 1}$ with the Laplacian eigenvalue $\frac{1}{4}+t_\phi^2$ ($t_\phi\geq 1$).
According to Watson's formula \cite{Watson}, we have
\begin{equation}\label{eqn:Watson}
  |\langle 1, \phi^3\rangle|^2= \Big|\int_{\X}\phi(z)^3\dd \mu(z)\Big|^2=\frac{\Lambda(1/2, \phi\times\phi\times\phi)}{8\Lambda(1, \sym^2\phi)^3}
\end{equation}
By using the factorization of $L$-function
\[
L(s, \phi\times\phi\times\phi)=L(s, \sym^2\phi\times\phi)L(s, \phi)=L(s, \sym^3\phi)L(s, \phi)^2
\]
and the bounds for the $\Gamma$-factors and $L$-values at $1$, we get that
\begin{equation}\label{eqn:Watson-L}
  \Big|\int_{\X}\phi(z)^3\dd \mu(z)\Big|^2=t_\phi^{-2+o(1)}L(1/2, \sym^3\phi)L(1/2, \phi)^2.
\end{equation}
Here $\sym^3\phi$ is cuspidal on $\GL(4)$ by \cite{K-S00}, thus $L(s, \sym^3\phi)$ is a $L$-function on $\GL(4)$. Applying the convexity bound of $L(s, \sym^3\phi)$ ($L(1/2, \sym^3\phi)\ll t_\phi^{1+\varepsilon}$) and the Weyl bound of $\GL(2)$ $L$-function $L(s, \phi)$ ($L(1/2, \phi)\ll t_\phi^{\frac{1}{3}+\varepsilon}$), we get \cite[Theorem 5]{Watson}
\[
\langle1, \phi^3\rangle=\int_{\X}\phi(z)^3\dd \mu(z)\ll t_{\phi}^{-\frac{1}{6}+\varepsilon}.
\]
Assuming GLH, we have
\[
\langle1, \phi^3\rangle \ll t_{\phi}^{-1+\varepsilon}.
\]
Now we consider variance estimate for the above cubic moment when $\phi$ varies in the spectral family. And we can prove the following theorem.
\begin{theorem}\label{thm:Var}
Let $1\leq \Delta\leq T$, then we have
\[
\sum_{T\leq t_\phi\leq T+\Delta}\Big|\langle1, \phi^3\rangle\Big|^2\ll_{\varepsilon}
\begin{cases}
  T^{-\frac{1}{7}+\varepsilon}\Delta^{\frac{7}{8}}, & \mbox{ if } 1\leq \Delta\leq T^{\frac{1}{3}},\\
  T^{-\frac{1}{4}+\varepsilon}\Delta^{\frac{5}{4}}, & \mbox{ if } T^{\frac{1}{3}}< \Delta\leq T^{\frac{3}{7}},\\
  T^{-\frac{1}{10}+\varepsilon}\Delta^{\frac{9}{10}}, & \mbox{ if } T^{\frac{3}{7}}< \Delta\leq T.\\
\end{cases}
\]
\end{theorem}
By using the convexity bound for $\rm GL(4)$ $L$-function and the sharp upper bound for the second moment of $\rm GL(2)$ $L$-functions, one can easily get that
\begin{equation}\label{eqn:TrivialBoundforVar}
\sum_{T\leq t_\phi\leq T+\Delta}|\langle1, \phi^3\rangle|^2\ll_{\varepsilon}T^{\varepsilon}\Delta,
\end{equation}
for any $1\leq \Delta\leq T.$
It implies that there exists a density one subset of $\phi$ within $T\leq t_\phi\leq T+\Delta$, such that $\langle1, \phi^3\rangle\ll_{\varepsilon}t_\phi^{-1/2+\varepsilon}.$
Theorem \ref{thm:Var} tells us that for all short spectral intervals, we can improve the exponent $-1/2$.
This is beneficial to consider the general variance 
$\sum_{T\leq t_\phi\leq T+\Delta}|\langle\psi, \phi^3\rangle|^2$
for a fixed $\psi\in \mathcal{C}_c^{\infty}(\X).$ In fact, by the deduction in \cite{Huang24}, even if assuming GLH,  we can only prove
\begin{equation}\label{eqn:CubicMomentGLH}
\int_{\X}\psi(z)\phi(z)^3\dd \mu(z)=\frac{3}{\pi}\int_{\X}\psi(z)\dd \mu(z)\cdot\int_{\X}\phi(z)^3\dd \mu(z)
+\text{C.}+\text{E.}
\ll_{\psi, \varepsilon} t_\phi^{-1/2+\varepsilon},
\end{equation}
where $\text{C.}$ and $\text{E.}$ denote the cusp form contribution and Eisenstein series contribution respectively.
Optimistically, we suggest that the decay rate of $\langle\psi, \phi^3\rangle$ should be matched to the leading term $\frac{3}{\pi}\langle\psi, 1\rangle\langle1, \phi^3\rangle$. So we conjecture that 
\begin{equation}
\sum_{T\leq t_\phi\leq T+\Delta}|\langle\psi, \phi^3\rangle|^2=\Big(\frac{9}{\pi^2}|\langle\psi, 1\rangle|^2+o_\psi(1)\Big) \sum_{T\leq t_\phi\leq T+\Delta}|\langle1, \phi^3\rangle|^2 \quad \text{ as } T\rightarrow \infty.
\end{equation}
With the help of Theorem \ref{thm:Var}, we can break the $-1/2$-barrier in \eqref{eqn:CubicMomentGLH} in density one sense.

On the other hand, by \eqref{eqn:Watson-L}, Theorem \ref{thm:Var} is related to the estimate for $L(1/2, \sym^3\phi)$ on average. Applying the H\"older inequality and the known bounds for moments of $\GL(2)$ $L$-functions, 
one can see that any non-trivial estimates for the higher moment of $L(1/2, \sym^3\phi)$ can improve the bound $\Delta$ in \eqref{eqn:TrivialBoundforVar}.
In order to obtain the better upper bound, we want to seek for the saving on the trivial estimate for the average of $L(1/2, \sym^3\phi)$. Since $\sym^3\phi$ is automorphic on $\GL(4)$, we can prove the following large sieve inequality.
\begin{theorem}\label{thm:LS_GL4}
Let $1\leq \Delta\leq T$ and $ N\geq 2$.
Then for any complex sequence $\{a_{d, k, m, n}\}$, we have
\begin{equation*}
\sum_{T\leq t_\phi\leq T+\Delta}
    \left|\sum_{N< d^4k^3m^2n\leq 2N}\lambda_{\sym^3\phi}(k, m, n)a_{d, k, m, n}
    \right|^2
    \ll_{\varepsilon} (NT)^{\varepsilon}(N+ T^7\Delta^3)\sum_{N< d^4k^3m^2n\leq 2N}|a_{d, k, m, n}|^2,
\end{equation*}
where $\lambda_{\sym^3\phi}(\cdot, \cdot, \cdot)$ is the  Fourier coefficient of $\GL(4)$ automorphic form $\sym^3\phi.$
\end{theorem}
The proof of Theorem \ref{thm:LS_GL4} proceeds by using the duality principle and the analytic properties of the degree 16 Rankin--Selberg $L$-function
$L(s, \sym^3\phi_1\times \sym^3\phi_2)$. This strategy was already presented in \cite{DK00}, \cite{TZ21} and so on.
In our case, the second term $T^7\Delta^3$ comes from the product of the square root of the analytic conductor for $L(s, \sym^3\phi_1\times \sym^3\phi_2)$ and the family size of $\phi$ when $t_\phi$ in the short interval $T\leq t_\phi\leq T+\Delta.$ Note that the conductor-dropping phenomenon appears in this range of spectral parameters $t_\phi$.
As consequences of the above large sieve, we can get the following non-trivial bounds for the higher moments of $L$-functions.
\begin{theorem}\label{thm:Lmoment}
Let $1\leq \Delta\leq T$, then we have
\[
\sum_{T\leq t_\phi\leq T+\Delta}|L\big(1/2, \sym^3\phi)|^8\ll_{\varepsilon} T^{\varepsilon}(T^8+T^7\Delta^3).
\]
\end{theorem}

\begin{theorem}\label{thm:Lmoment2}
Let $T\geq 1$, then we have
\[
\sum_{T\leq t_\phi\leq 2T}|L\big(1/2, \sym^3\phi)|^{10}\ll_{\varepsilon} T^{10+\varepsilon}.
\]
\end{theorem}
These bounds on the moment of $L$-functions are weak but better than the convexity bound for $L\big(1/2, \sym^3\phi)$ on average. It is helpful to improve the trivial bound \eqref{eqn:TrivialBoundforVar} on the variance estimate. We remark that Nelson's work \cite{Nelson21} implies the weak subconvexity bounds for $L$-functions on the spectral aspect, which can give a slight improvement on Theorem \ref{thm:Var}. Since that this improvement on the exponent is small and for simplicity in results, we don't discuss it in the paper.

\subsection{The variance for cubic moment of Eisenstein series}\label{subsec:VarES}
Let $\psi\in \mathcal{C}_c^{\infty}(\X)$ and $E_t$ with $t\geq 1$, we have the following cubic moment estimate
\begin{equation}\label{eqn:QUE_cubicmoment}
\langle \psi, E_t^3\rangle=\int_{\X}\psi(z)\overline{E_t(z)^3}\dd \mu(z)\ll_{\psi, \varepsilon} t^{-\frac{1}{3}+\varepsilon}.
\end{equation}
It was remarked by Huang \cite[Eqn (1.9)]{Huang24} with power saving $O(t^{-\frac{1}{6}+\varepsilon})$ previously and proven by Guo \cite[Theorem 1.2]{Guo24}.
Assuming GLH, we can prove that
\[
\langle \psi, E_t^3\rangle\ll_{\psi, \varepsilon} t^{-\frac{1}{2}+\varepsilon}.
\]
Unconditionally, we have the following variance estimate for $\langle \psi, E_t^3\rangle$ which corresponds the strength of Lindel\"of-on-average bound. 
\begin{theorem}\label{thm:VarES}
  Fix a smooth compactly supported function  $\psi$ on $\X$. Then for $T\geq 1,$ we have
  \begin{equation}\label{eqn:VarES}
  \int_{T}^{2T}\big|\langle \psi, E_t^3\rangle\big|^2 \dd t\ll_{\psi, \varepsilon}T^\varepsilon.
  \end{equation}
\end{theorem}
Attaching the observable $\psi$ on the cubic moment is a more interesting and difficult thing in computing the variance.
Since that by the Selberg decomposition for $\psi$:
\begin{equation}\label{eqn:Selbergdecomp}
  \psi(z)=\frac{3}{\pi}\langle \psi, 1\rangle+\sum_{k\geq 1}\langle \psi, \phi_k\rangle\phi_k(z)+\frac{1}{4\pi}\int_{\mathbb{R}}\langle \psi, E_\tau\rangle E_\tau(z)\dd \tau,
\end{equation}
it suffices to consider the variances for matrix coefficients:
\begin{equation}\label{eqn:VarMatrixCoefs}
\int_{T}^{2T}\big|\langle 1, E_t^3\rangle_{\reg}\big|^2\dd t, 
\quad 
\int_{T}^{2T}\big|\langle \phi_k, E_t^3\rangle\big|^2\dd t, 
\quad 
\int_{T}^{2T}\big|\langle E_\tau, E_t^3\rangle_{\reg}\big|^2\dd t,
\end{equation}
with $t_k, \tau\ll T^\varepsilon.$
Here $\langle \cdot, \cdot \rangle_{\reg}$ is the regularized inner product on $\X$ which is introduced by Zagier \cite{Zag82}. 
Applying the Rankin--Selberg theory, we shall face with the multiple averages of $L$-functions.
For example, in the case of $\phi_k$-contribution, with $\phi_k$ be even and $\phi_k$'s spectral parameter $t_k\ll T^\varepsilon,$
we apply (regularized) Parseval's identity on $\langle \phi_k, E_t^3\rangle=\langle \phi_k\overline{E_t}, E_t^2\rangle$ and we get
\begin{multline}\label{eqn:decp_phiE3}
\langle \phi_k, E_t^3\rangle=\frac{3}{\pi}\langle \phi_k\overline{E_t}, 1\rangle\langle 1,  E_t^2\rangle+
\sum_{j\geq 1}\langle \phi_k\overline{E_t}, \phi_j\rangle \langle \phi_j, E_t^2\rangle
+\frac{1}{4\pi}\int_{\mathbb{R}}\langle \phi_k\overline{E_t}, E_\nu\rangle \langle E_\nu, E_t^2\rangle_{\reg}\dd \nu+\text{tail terms}.
\end{multline}
For the cusp form part, using unfolding method and taking the absolute value of the mixed $L$-functions, roughly, we arrive at 
\[
\frac{1}{t^{\frac{3}{2}}}\quad \sideset{}{^{\even}}\sum_{|t_j-t|\ll t^\varepsilon}|L(1/2+it, \phi_k\times\phi_j)L(1/2+2it, \phi_j)|L(1/2, \phi_j),
\]
where the weight and truncated range come from the evaluation of $\Gamma$-factors in complete $L$-functions.
To consider the variance,  we shall deal with
\[
  \frac{1}{T^{3}}\int_{T}^{2T}\Big|\sideset{}{^{\even}}\sum_{|t_j-t|\ll T^\varepsilon}|L(1/2+it, \phi_k\times\phi_j)L(1/2+2it, \phi_j)|L(1/2, \phi_j)\Big|^2\dd t.
\]
Applying Cauchy--Schwarz with the second moment bound for $\GL(2)$ $L$-functions:
$$\sum_{|t_j-t|\ll T^\varepsilon}|L(1/2+2it, \phi_j)|^2\ll T^{1+\varepsilon},$$
it is roughly bounded by
\[
\frac{1}{T^{2}}\int_{T}^{2T}\sideset{}{^{\even}}\sum_{|t_j-t|\ll T^\varepsilon}|L(1/2+it, \phi_k\times\phi_j)|^2L(1/2, \phi_j)^2\dd t.
\]
By exchanging the order of integral and summation, it suffices to bound
\[
\frac{1}{T^{2}}\quad\sideset{}{^{\even}}\sum_{T\leq t_j\leq 2T}|L(1/2+it_j+i\alpha, \phi_k\times\phi_j)|^2L(1/2, \phi_j)^2,
\]
where $\alpha$ is a real shift with $|\alpha|\ll T^{\varepsilon}.$
Using Cauchy--Schwarz again and combining with the fourth moment bound
$\sum_{T\leq t_j\leq 2T}L(1/2, \phi_j)^4\ll T^{2+\varepsilon},$
we arrive at the upper bound
\[
\frac{1}{T}\,\Big(\quad\sideset{}{^{\even}}\sum_{T\leq t_j\leq 2T}|L(1/2+it_j+i\alpha, \phi_k\times\phi_j)|^4\Big)^{\frac{1}{2}}.
\]
Indeed, we can prove the following uniform Lindel\"of-on-average bound.
\begin{theorem}\label{thm:4m-of-L}
Let $T\geq 1$, $\alpha\in \mathbb{R}$ with $|\alpha|\ll T^{\varepsilon}$ and $\phi_k$ be even with $t_k\ll T^\varepsilon$, then we have
\[
\sideset{}{^{\even}}\sum_{T\leq t_j\leq 2T}|L(1/2+it_j+i\alpha, \phi_k\times\phi_j)|^4\ll_{\varepsilon} T^{2+\varepsilon},
\]
where $\sum^{\even}$ means that the sum runs through even $\phi_j$.
\end{theorem}
Using Theorem \ref{thm:4m-of-L} and combining the variance contribution from other parts in \eqref{eqn:decp_phiE3}, we get
\[
\int_{T}^{2T}\big|\langle \phi_k, E_t^3\rangle\big|^2\dd t\ll_{\varepsilon} T^\varepsilon,
\]
which is a key ingredient to Theorem \ref{thm:VarES}.
By a similar argument as above for the $E_\tau$-contribution in \eqref{eqn:VarMatrixCoefs},  to prove
\[
\int_{T}^{2T}\big|\langle E_\tau, E_t^3\rangle_{\reg}\big|^2\dd t\ll_{\varepsilon} T^\varepsilon,
\]
with $\tau\ll T^\varepsilon$, we need the eighth moment bound for $\GL(2)$ $L$-functions as follow.
\begin{theorem}\label{thm:8m-of-L}
Let $T\geq 1$, $\alpha\in \mathbb{R}$ with $|\alpha|\ll_{\varepsilon} T^{\varepsilon}$, then we have
\[
\sideset{}{^{\even}}\sum_{T\leq  t_j\leq 2T}|L(1/2+it_j+i\alpha, \phi_j)|^8\ll T^{2+\varepsilon}.
\]
\end{theorem}
Theorem \ref{thm:4m-of-L} and Theorem \ref{thm:8m-of-L} essentially are the fourth moment of $\GL(2)\times\GL(2)$ Rankin--Selberg $L$-functions and the eighth moment of $\GL(2)$ $L$-functions at special points $\frac{1}{2}+it_j$.
In this problem, Chandee and Li \cite{C-L20} established that
\begin{equation}\label{eqn:CL_work}
\sum_{ t_j\leq T}|L(1/2+it_j, F\times \phi_j)|^2\ll_{F, \varepsilon} T^{2+\varepsilon},
\end{equation}
where $F$ is a Hecke--Maass cusp form for $\GL_4(\mathbb{Z}).$ Previous works on this type sharp upper bound are available for $F$ of degree $n\leq 3$,
see Luo's work \cite{Luo95} on $n\leq 2$ and Young's work \cite{You13} on $n=3$.
These type Lindel\"of-on-average bounds do not imply the subconvexity bounds for $L$-functions due to a conductor-dropping phenomenon. 
For fixed $F$ of degree $4$, the analytic conductor for $L(1/2+it_j, F\times \phi_j)$ is $t_j^{4+\varepsilon}$. Therefore the results for Theorem \ref{thm:4m-of-L}, Theorem \ref{thm:8m-of-L} and \eqref{eqn:CL_work} only match the convexity bound for fixed $L(1/2+it_j, \phi_k\times \phi_j)$, $L(1/2+it_j, \phi_j)$ and $L(1/2+it_j, F\times \phi_j)$
respectively.

Our proof idea for Theorem \ref{thm:4m-of-L} and Theorem \ref{thm:8m-of-L} follows the work of Chandee and Li \cite{C-L20}. We view $L(1/2+it_j, \phi_k\times \phi_j)^2$ and $L(1/2+it_j, \phi_j)^4$ as the $\GL(4)\times\GL(2)$ $L$-functions in form.
After applying Young's result \cite[Theorem 7.1]{You13}, we use the degree $4$ type balanced Voronoi summation formulas to arrive the dual sum. Then we simplify the exponential sums from the hyper--Kloosterman sum and split the dual sum into small and big range. 
We can treat the small range by using the trivial bound for the integral transformation and large sieve inequality directly. 
For the big range, we need to analysis the integral transformation originated from Voronoi formula carefully to figure out the essential range of $n_1, n_2$ in two dual sums, i.e. $n_1$ is not far away from $n_2$. Then using the large sieve in pair, we get the final bound. The estimate of the big range roughly comes from a combination of the phase analysis and the large sieve inequality. A simple case for understanding this is \cite[Lemma 3.2]{You11} which uses the first derivative test for the addition phase in the large sieve.

There are serval differences in the techniques between our proof and Chandee and Li's \cite{C-L20}. The first one is the balanced Voronoi summation formulas. We shall use the formulas corresponding the isobaric sum $\Phi=\phi_k\boxplus\phi_k$ and $\E=1\boxplus1\boxplus1\boxplus1$. Since these are not in $\GL(4)$ case, we borrow the result of K{\i}ral and Zhou \cite{K-Z16}. This result is useful for us due that it is about getting the Voronoi summation formulas essentially based on the functional equations of the multiplicative twisted $L$-functions rather than the automorphy. 
It is worth mention that the residue term in Voronoi formula of $\E$ can not be ignored. We borrow the computation in \cite{C-G01} to dig out the singular part for the additive twisted $L$-function (see Remark \ref{Remark:Voronoi_Residue}) and use the Cauchy integral formula to bound this residue term. 
By bounding the other terms trivially, the contribution from the residue term matches our desired bound. A similar treatment about this is the case of $F=1\boxplus1\boxplus1$ in \cite[\S 10]{You13}. 
The second difference is about analysis of the integral transformation. 
Note that we need a small uniformity of $F=\Phi$ on the estimate \eqref{eqn:CL_work} to achieve Theorem \ref{thm:VarES}.
When the spectral parameter $t_\phi$ varies in $1\ll t_\phi\ll T^\varepsilon$, it is difficult to get the same effective approximation as \cite[Lemma 5.2]{C-L20}. Therefore we treat the integral transformation in the big range case by using the theory of oscillatory integrals,  such as stationary phase method. Although the techniques are different, we get the same result as desired.

\subsection{Structure of the paper and notations}
We first quickly give the proofs of the results in \S \ref{subsec:VarHM}.
In \S \ref{sec:LSgl4}, we prove Theorem \ref{thm:LS_GL4}. In \S \ref{sec:Momentgl4}, we prove Theorem \ref{thm:Lmoment}, Theorem \ref{thm:Lmoment2} and Theorem \ref{thm:Var}. 
Subsequently, we focus on the proof of Theorem \ref{thm:VarES}. In \S \ref{sec:VarES}, we give the proof of Theorem \ref{thm:VarES} by using Lemma \ref{lemma:cusp_contri} and \ref{lemma:Eis_contri}. In \S \ref{sec:TwoLemma}, 
we will prove Lemma \ref{lemma:cusp_contri} and \ref{lemma:Eis_contri} by using Theorem \ref{thm:4m-of-L} and Theorem \ref{thm:8m-of-L} respectively. In \S \ref{sec:preliminaries}, we do some preliminary job for proving Theorem \ref{thm:4m-of-L} and Theorem \ref{thm:8m-of-L}.
Then we finish the proof of Theorem \ref{thm:4m-of-L} and Theorem \ref{thm:8m-of-L} in \S \ref{sec:ProvingMoment}.

Throughout the paper, $\varepsilon$ is an arbitrarily small positive number and $A$ is an arbitrarily large positive number,
all of them may be different at each occurrence. As usual, we use the standard Landau and Vinogradov notations \( O(\cdot) \), \( o(\cdot) \), \(\ll\),  \(\gg\), $\asymp$ and $\sim$. Specifically, we express \( X \ll Y \), \( X = O(Y) \), or \( Y \gg X \) when there exists a constant \( C \) such that \( |X| \leq C|Y| \). If the constant $C=C_s$ depends on some object $s$, we write $X=O_s(Y)$. As \( N \to \infty \), \( X = o(Y) \) indicates that \( |X| \leq c(N)Y \) for some function \( c(N) \) that tends to zero.
We use \( X \asymp Y \) to denote that $c_1Y\leq X\leq c_2Y$ for some positive constant $c_1, c_2$. And  $X\sim Y$  denotes that $Y\leq X<2Y$.

\section{Large sieve for symmetric cubes}\label{sec:LSgl4}
In this section, we give the proof of Theorem \ref{thm:LS_GL4}.
\subsection{The $\rm GL(4)$ $L$-funtion and the $\rm GL(4)\times \rm GL(4)$ Rankin--Selberg $L$-function}
Let $F$ be a Hecke--Maass cusp form on $\rm SL_4(\mathbb{Z})$ with Fourier coefficient $\lambda_{F}(k, m, n)$.
The $\rm GL(4)$ $L$-function associated $F$ is defined as
\begin{equation}
L(s, F)=\sum_{n\geq 1}\frac{\lambda_F(n, 1, 1)}{n^s}, \quad \Re(s)>1.
\end{equation}
The functional equation for $L(s, F)$ is
\begin{equation}\label{eqn:FE4}
\Lambda(s, F):=\pi^{-2s}\prod_{1\leq i\leq 4}\Gamma\Big(\frac{s+\mu_j(F)}{2}\Big)L(s, F)=\epsilon_F\Lambda(1-s, F),
\end{equation}
where $\{\mu_j(F)\}_{1\leq j\leq 4}$ are the Satake parameters satisfying $\sum_{1\leq j\leq 4}\mu_j(F)=0$ and $\epsilon_F$ is the root number obeying $|\epsilon_F|=1$.
If $F=\sym^3\phi$ is a symmetric cube lift for $\rm GL(2)$ Hecke--Maass form $\phi$ with spectral parameter $t_\phi$, we have
$$\mu_1(F)=i3t_\phi, \quad\mu_2(F)=it_\phi,\quad \mu_3(\phi)=-it_\phi,\quad  \mu_{4}=-i3t_\phi.$$
Note that $\sym^3\phi$ is self dual, thus \eqref{eqn:FE4} becomes
\begin{equation}\label{eqn:FEsym3}
\Lambda(s, \sym^3\phi):=\gamma_{\sym^3\phi}(s)L(s, \sym^3\phi)=\epsilon_{\sym^3\phi}\Lambda(1-s, \sym^3\phi),
\end{equation}
where
\[
\gamma_{\sym^3\phi}(s)=\pi^{-2s}\prod_{\pm}\Gamma\Big(\frac{s\pm i3t_\phi}{2}\Big)
\prod_{\pm}\Gamma\Big(\frac{s\pm it_\phi}{2}\Big).
\]
The Rankin--Selberg $L$-function of two $\rm SL_4(\mathbb{Z})$ Hecke--Maass forms $F_1$ and $F_2$ is defined as
\begin{equation}
L(s, F_1\times F_2)=\sum_{d, k, m, n\geq 1}\frac{\lambda_{F_1}(k, m, n)\overline{\lambda_{F_2}(k, m, n)}}{(d^4k^3m^2n)^s}, \quad \Re(s)>1.
\end{equation}
It has meromorphic continuation to $s\in \mathbb{C}$ with possible pole at $s=1$, and satisfies the functional equation
\begin{equation}\label{eqn:FE44}
\Lambda(s, F_1\times F_2):=\gamma_{F_1, F_2}(s)L(s, F_1\times F_2)
=\epsilon_{F_1\times F_2}\Lambda(1-s, F_1\times F_2),
\end{equation}
where
\begin{equation}
\gamma_{F_1, F_2}(s):=\pi^{-8s}\prod_{1\leq i\leq 4}\prod_{1\leq j\leq 4}\Gamma\Big(\frac{s+\mu_i(F_1)+\mu_j(F_2)}{2}\Big)
\end{equation}
and $|\epsilon_{F_1\times F_2}|=1$.
If $F_1=\sym^3\phi_1$ and $F_2=\sym^3\phi_2$, then \eqref{eqn:FE44} becomes
 \begin{equation}\label{eqn:FE44sym}
 \begin{split}
\Lambda(s, \sym^3\phi_1\times \sym^3\phi_2)
&:=\gamma_{\sym^3\phi_1, \sym^3\phi_2}(s) L(s, \sym^3\phi_1\times \sym^3\phi_2)\\
&=\epsilon_{\sym^3\phi_1\times\sym^3\phi_2}\Lambda(1-s, \sym^3\phi_1\times \sym^3\phi_2)
\end{split}
\end{equation}
where
\begin{equation}
\begin{split}
\gamma_{\sym^3\phi_1, \sym^3\phi_2}(s)
:=\pi^{-8s}\prod_{\pm}&\prod_{\pm}\Gamma\Big(\frac{s\pm i3t_{\phi_1}\pm i3t_{\phi_2}}{2}\Big)
\prod_{\pm}\prod_{\pm}\Gamma\Big(\frac{s\pm i3t_{\phi_1}\pm it_{\phi_2}}{2}\Big)\\
&\cdot\prod_{\pm}\prod_{\pm}\Gamma\Big(\frac{s\pm it_{\phi_1}\pm i3t_{\phi_2}}{2}\Big)
\prod_{\pm}\prod_{\pm}\Gamma\Big(\frac{s\pm it_{\phi_1}\pm it_{\phi_2}}{2}\Big)
\end{split}
\end{equation}

\subsection{Proof of Theorem \ref{thm:LS_GL4}}\label{subsec:proveLS}
To prove Theorem \ref{thm:LS_GL4}, by using the duality principle, it suffices to prove
\begin{equation}
\sum_{N<d^4k^3m^2 n\leq 2N}
    \left|\sum_{T\leq t_\phi\leq T+\Delta}\lambda_{\sym^3\phi}(k, m, n)b_\phi
    \right|^2\ll (NK)^{\varepsilon}(N+ T^7\Delta^3)\sum_{T\leq t_\phi\leq T+\Delta}|b_\phi|^2,
\end{equation}
for any complex sequence $\{b_\phi\}$.
We select a smooth nonnegative bump function $W$ with compact support on $\mathbb{R}_{>0}$, satisfying $W(x)\geq 1$ for $1\leq x\leq 2$. Then
\begin{equation*}
\sum_{N<d^4k^3m^2 n\leq 2N}
    \left|\sum_{T\leq t_\phi\leq T+\Delta}\lambda_{\sym^3\phi}(k, m, n)b_\phi\right|^2
    \leq \sum_{d, k, m, n\geq 1}W\Big(\frac{d^4k^3m^2n}{N}\Big)
    \left|\sum_{T\leq t_\phi\leq T+\Delta}\lambda_{\sym^3\phi}(k, m, n)b_\phi\right|^2.
\end{equation*}
Opening the square and applying the Mellin inversion, we have
\begin{multline*}
\sum_{N<d^4k^3m^2 n\leq 2N}
    \left|\sum_{T\leq t_\phi\leq T+\Delta}\lambda_{\sym^3\phi}(k, m, n)b_\phi\right|^2\\
   \leq\sum_{T\leq t_{\phi_1}\leq T+\Delta}\sum_{T\leq t_{\phi_2}\leq T+\Delta}b_{\phi_1}\overline{b_{\phi_2}}
    \frac{1}{2\pi i}\int_{(2)}N^s\widetilde{W}(s)L(s, \sym^3\phi_1\times\sym^3\phi_2)\dd s,
\end{multline*}
where $\widetilde{W}(s)=\int_{0}^{+\infty}W(x)x^{s-1}\dd x.$
Here by repeated integrations, we have $\widetilde{W}(s)\ll _{W, \Re(s)}\frac{1}{(|s|+1)^A}.$
Next we shift the contour of the integration to the line $\Re(s)=-\varepsilon$. And we cross a potential pole at $s=1$ only, which exists if and only if
$\sym^3\phi_1=\sym^3\phi_2$ which implies $\phi_1=\phi_2$. This pole contributes
\begin{equation}\label{eqn:pole-contri}
N\sum_{T\leq t_\phi\leq T+\Delta}|b_\phi|^2\Res_{s=1}L(s, \sym^3\phi\times \sym^3\phi).
\end{equation}
Note that by \cite{Li10} we get $\Res_{s=1}L(s, \sym^3\phi\times \sym^3\phi)=t_\phi^{o(1)}$. Thus \eqref{eqn:pole-contri} is bounded by
$$T^\varepsilon N\sum_{T\leq t_\phi\leq T+\Delta}|b_\phi|^2.$$
The shifted integral contributes
\begin{equation}\label{eqn:shiftcontri}
\sum_{T\leq t_{\phi_1}, t_{\phi_2}\leq T+\Delta\atop \phi_1\neq \phi_2}b_{\phi_1}\overline{b_{\phi_2}}
    \frac{1}{2\pi i}\int_{(-\varepsilon)}N^s\widetilde{W}(s)L(s, \sym^3\phi_1\times\sym^3\phi_2)\dd s.
\end{equation}
Using the functional equation \eqref{eqn:FE44sym} and the rapid decay of $\widetilde{W}$, it becomes
\begin{multline*}
\sum_{T\leq t_{\phi_1}, t_{\phi_2}\leq T+\Delta\atop \phi_1\neq \phi_2}b_{\phi_1}\overline{b_{\phi_2}}
    \frac{1}{2\pi i}\int_{(1+\varepsilon)}N^{1-s}\widetilde{W}(1-s)
    \frac{\epsilon_{\sym^3\phi_1\times\sym^3\phi_2}\gamma_{\sym^3\phi_1, \sym^3\phi_2}(s)}{\gamma_{\sym^3\phi_1, \sym^3\phi_2}(1-s)}L(s, \sym^3\phi_1\times\sym^3\phi_2)\dd s\\
\ll N^{\varepsilon} \sum_{T\leq t_{\phi_1}, t_{\phi_2}\leq T+\Delta\atop \phi_1\neq \phi_2}|b_{\phi_1}b_{\phi_2}|
   \int_{|t| \ll (TN)^\varepsilon }\left|
    \frac{\gamma_{\sym^3\phi_1, \sym^3\phi_2}(1+\varepsilon+it)}{\gamma_{\sym^3\phi_1, \sym^3\phi_2}(-\varepsilon-it)}L(1+\varepsilon+it, \sym^3\phi_1\times\sym^3\phi_2)\right|\dd t\\
     +(TN)^{-A} \sum_{T\leq t_{\phi_1}, t_{\phi_2}\leq T+\Delta\atop \phi_1\neq \phi_2}|b_{\phi_1}b_{\phi_2}|.
\end{multline*}
By Stirling's formula, for $|t|\ll (TN)^\varepsilon$ and $T\leq t_{\phi_1}, t_{\phi_2}\leq T+\Delta$ we have
\begin{equation*}
\left|\frac{\gamma_{\sym^3\phi_1, \sym^3\phi_2}(1+\varepsilon+it)}{\gamma_{\sym^3\phi_1, \sym^3\phi_2}(-\varepsilon-it)}\right|\ll (T^{12}\Delta^4)^{\frac{1}{2}+\varepsilon}N^{\varepsilon}.
\end{equation*}
Thus \eqref{eqn:shiftcontri} is bounded by
\begin{equation*}
(TN)^\varepsilon T^6\Delta^2\sum_{T\leq t_{\phi_1}, t_{\phi_2}\leq T+\Delta\atop \phi_1\neq \phi_2}|b_{\phi_1}b_{\phi_2}|
    \ll  (TN)^\varepsilon T^7\Delta^3\sum_{T\leq t_{\phi}\leq T+\Delta}|b_\phi|^2.
\end{equation*}
Combining the above bounds together, we finish the proof of Theorem \ref{thm:LS_GL4}

\section{Moments of symmetric cube $L$-functions}\label{sec:Momentgl4}
\begin{lemma}\label{lemma:L-value}
  For $\Re(s)\gg 1$, we have
  \begin{equation}\label{eqn:L^4}
  L(s, \sym^3\phi)^4
  =\sum_{d=1}^{\infty}\sum_{m_1=1}^{\infty}\sum_{m_2=1}^{\infty}\sum_{m_3=1}^{\infty}\frac{\lambda_{\sym^3\phi}(m_1, m_2, m_3)\tau(m_1, m_2, m_3)}{(d^4m_1^3m_2^2m_3)^s},
  \end{equation}
  where $\tau(m_1, m_2, m_3)$ is defined by
  \[
  \tau(m_1,m_2, m_3):=\sum_{n_1\mid m_1}\sum_{n_2\mid m_2}\sum_{n_3\mid m_3}\tau(\frac{m_2n_3}{n_2},\frac{m_1n_2}{n_1}),
  \]
  with $\tau(m_1, m_2):=\sum_{k_1\mid m_1}\sum_{k_2\mid m_2}d_2(k_1k_2)$, $d_\ell$ is the $\ell$-fold divisor function.
Moreover, we have
\[
\tau(m_1, m_2, m_3)\ll (m_1m_2m_3)^\varepsilon, \quad \text{ for any } m_1, m_2, m_3\geq 1.
\]
\end{lemma}
\begin{proof}
Comparing the Dirchlet coefficients of \eqref{eqn:L^4}, it suffices to show
\begin{multline}\label{eqn:CompareDirichlet}
  \sum_{m_1m_2m_3m_4=n}\lambda_{\sym^3\phi}(m_1, 1, 1)\lambda_{\sym^3\phi}(m_2, 1, 1)\lambda_{\sym^3\phi}(m_3, 1, 1)\lambda_{\sym^3\phi}(m_4, 1, 1)\\=
  \sum_{d^4m_1^3m_2^2m_3=n}\lambda_{\sym^3\phi}(m_1, m_2, m_3)\tau(m_1, m_2, m_3).
\end{multline}
Since $\sym^3\phi$ is a Hecke--Maass cusp form on $\GL(4)$, we have the following Hecke relation
\begin{equation}\label{eqn:Hecke-relation}
  \lambda_{\sym^3\phi}(m, 1, 1)\lambda_{\sym^3\phi}(m_1, m_2, m_3)=
  \sum_{c_1c_2c_3c_4=m \atop c_1\mid m_1, c_2\mid m_2, c_3\mid m_3}\lambda_{\sym^3\phi}(\frac{m_1c_4}{c_1}, \frac{m_2c_1}{c_2}, \frac{m_3c_2}{c_3}).
\end{equation}
Since that, for $\Re(s)\gg 1$,
\begin{equation*}
\begin{split}
L(s, \sym^3\phi)
&=\sum_{m\geq 1}\frac{\lambda_{\sym^3\phi}(m, 1, 1)}{m^s}\\
&=\prod_{p}\Big(1-\frac{\alpha_\phi(p)^3}{p^{s}}\Big)^{-1}
\Big(1-\frac{\alpha_\phi(p)}{p^{s}}\Big)^{-1}
\Big(1-\frac{\beta_\phi(p)}{p^{s}}\Big)^{-1}
\Big(1-\frac{\beta_\phi(p)^3}{p^{s}}\Big)^{-1}
\end{split}
\end{equation*}
and $\alpha_\phi(p)\beta_\phi(p)=1$, $\alpha_\phi(p)+\beta_\phi(p)=\lambda_\phi(p)\in\mathbb{R},$
we get that $\lambda_{\sym^3\phi}(m, 1, 1)$ is real for each $m\geq 1.$ 
Thus by induction on Hecke relation, we have
\begin{equation}\label{eqn:conjugate}
  \lambda_{\sym^3\phi}(m_1, m_2, m_3)=\overline{\lambda_{\sym^3\phi}(m_3, m_2, m_1)}=\lambda_{\sym^3\phi}(m_3, m_2, m_1).
\end{equation}
Firstly, we claim the following relation
\begin{multline}\label{eqn:Hecke34}
  \sum_{m_1m_2m_3=m}\lambda_{\sym^3\phi}(m_1, 1, 1)\lambda_{\sym^3\phi}(m_2, 1, 1)\lambda_{\sym^3\phi}(m_3, 1, 1)
  \\=\sum_{m_1m_2^2m_3^3=m}\lambda_{\sym^3\phi}(m_1, m_2, m_3)\tau(m_1, m_2).
\end{multline}
Assuming \eqref{eqn:Hecke34} and applying it with $m=n/m_4$, we see the left hand side of \eqref{eqn:CompareDirichlet} is
\[
\sum_{m_4\mid n}\sum_{m_1m_2^2m_3^3=n/m_4}\lambda_{\sym^3\phi}(m_4, 1, 1)\lambda_{\sym^3\phi}(m_1, m_2, m_3)\tau(m_1, m_2).
\]
Using the Hecke relation \eqref{eqn:Hecke-relation}, and changing variables $m_1=c_1n_1$, $m_2=c_2n_2$, $m_3=c_3n_3$,  it becomes
\begin{equation*}
\begin{split}
&\sum_{m_1m_2^2m_3^3m_4=n}\sum_{c_1c_2c_3c_4=m_4 \atop c_1\mid m_1, c_2\mid m_2, c_3\mid m_3}\lambda_{\sym^3\phi}(\frac{m_1c_4}{c_1}, \frac{m_2c_1}{c_2}, \frac{m_3c_2}{c_3})\tau(m_1, m_2)\\
=&\sum_{c_4c_1^2c_2^3c_3^4n_1n_2^2n_3^3=n}\lambda_{\sym^3\phi}(n_1c_4, n_2c_1, n_3c_2)\tau(c_1n_1, c_2n_2)
\end{split}
\end{equation*}
Let $n_1c_4=k_1$, $n_2c_1=k_2$ and $n_3c_2=k_3$, we get
\[
\sum_{k_1k_2^2k_3^3c_3^4=n}\lambda_{\sym^3\phi}(k_1,k_2,k_3)\sum_{n_1\mid k_1}\sum_{n_2\mid k_2}\sum_{n_3\mid k_3}\tau(\frac{k_2n_1}{n_2}, \frac{k_3n_2}{n_3}).
\]
Then by using \eqref{eqn:conjugate}, it equals the right hand side of \eqref{eqn:CompareDirichlet}.
Therefore it remains to prove \eqref{eqn:Hecke34}.
By using \eqref{eqn:Hecke-relation} twice, we have
\begin{equation*}
\begin{split}
   &\sum_{m_1m_2m_3=m}\lambda_{\sym^3\phi}(m_1, 1, 1)\lambda_{\sym^3\phi}(m_2, 1, 1)\lambda_{\sym^3\phi}(m_3, 1, 1)\\
   =&\sum_{d^2n_1n_2m_3=n}\lambda_{\sym^3\phi}(n_1n_2, d, 1)\lambda_{\sym^3\phi}(m_3, 1, 1)\\
   =&\sum_{d^2n_1n_2m_3=n}\sum_{c_1c_2c_4=m_3\atop c_1\mid n_1n_2, c_2\mid d}\lambda_{\sym^3\phi}(\frac{n_1n_2c_4}{c_1}, \frac{dc_2}{c_1}, c_2)\\
   =&\sum_{c_2^2\ell^2km_3=n}\sum_{c_1c_2\mid m_3\atop c_1\mid k}\lambda_{\sym^3\phi}(\frac{km_3}{c_1^2c_2}, \ell c_1, c_2)d_2(k)\\
   =&\sum_{c_1^2c_2^3\ell^2\kappa r=n}\lambda_{\sym^3\phi}(\kappa r, \ell c_1, c_2)d_2(\kappa c_1).
\end{split}
\end{equation*}
Then we rewrite $\kappa r=m_1$, $\ell c_1=m_2$ and $c_2=m_3$, it becomes
\[
\sum_{m_3^3m_2^2m_1=n}\lambda_{\sym^3\phi}(m_1, m_2, m_3)\sum_{\kappa \mid m_1}\sum_{c_1\mid m_2}d_2(\kappa c_1).
\]
This finishes the proof of the claim \eqref{eqn:Hecke34}.
Finally the bound for $\tau$ comes from the divisor bound $d_{\ell}(n)\ll n^\varepsilon.$
\end{proof}

\begin{lemma}\label{lemma:L-value2}
  For $\Re(s)\gg 1$, we have
  \begin{equation}\label{eqn:L^5}
  L(s, \sym^3\phi)^5
  =\sum_{m_1=1}^{\infty}\sum_{m_2=1}^{\infty}\sum_{m_3=1}^{\infty}\sum_{m_4=1}^{\infty}\frac{\lambda_{\sym^3\phi}(m_2, m_3, m_4)\tau(m_1, m_2, m_3, m_4)}{(m_1^4m_2^3m_2^2m_4)^s},
  \end{equation}
  where $\tau(m_1, m_2, m_3, m_4)$ defined by
  \[
  \tau(m_1,m_2, m_3, m_4):=\sum_{d_1\mid m_1}\sum_{d_2\mid m_2}\sum_{d_3\mid m_3}\sum_{d_4\mid m_4}
  \tau(\frac{d_1m_2}{d_2}, \frac{d_2m_3}{d_3}, \frac{d_3m_4}{d_4}),
  \]
with $\tau(m_1, m_2, m_3)$ be defined in Lemma \ref{lemma:L-value}.
Moreover, we have
\[
\tau(m_1, m_2, m_3, m_4)\ll (m_1m_2m_3m_4)^\varepsilon, \quad \text{ for any } m_1, m_2, m_3, m_4\geq 1.
\]
\end{lemma}
\begin{proof}
  It suffices to show that
  \begin{multline}\label{eqn:CompareDirichlet2}
  \sum_{m_1m_2m_3m_4m_5=n}\lambda_{\sym^3\phi}(m_1, 1, 1)\lambda_{\sym^3\phi}(m_2, 1, 1)\lambda_{\sym^3\phi}(m_3, 1, 1)\lambda_{\sym^3\phi}(m_4, 1, 1)\lambda_{\sym^3\phi}(m_5, 1, 1)\\=
  \sum_{m_1^4m_2^3m_3^2m_4=n}\lambda_{\sym^3\phi}(m_2, m_3, m_4)\tau(m_1, m_2, m_3, m_4).
\end{multline}
Using \eqref{eqn:CompareDirichlet} and \eqref{eqn:conjugate}, the left hand side is equal to
\begin{equation*}
  \sum_{d^4n_1^3n_2^2n_3m_5=n}\lambda_{\sym^3\phi}(n_3, n_2, n_1)\tau(n_1, n_2, n_3)\lambda_{\sym^3\phi}(m_5, 1, 1).
\end{equation*}
By Hecke relation \eqref{eqn:Hecke-relation}, it becomes
\begin{equation*}
\begin{split}
  &\sum_{d^4n_1^3n_2^2n_3m_5=n}\tau(n_1, n_2, n_3)\sum_{c_1c_2c_3c_4=m_5\atop c_1\mid n_1, c_2\mid n_2, c_3\mid n_3}\lambda_{\sym^3\phi}(\frac{c_4n_3}{c_3}, \frac{c_3n_2}{c_2}, \frac{c_2n_1}{c_1})\\
  =&\sum_{d^4c_1^4k_1^3c_2^3k_2^2c_3^2k_3c_4=n}\lambda_{\sym^3\phi}(c_4k_3, c_3k_2, c_2k_1)\tau(k_1c_1, k_2c_2, k_3c_3).
\end{split}
\end{equation*}
Writing $dc_1=m_1$, $k_1c_2=m_2$, $k_2c_3=m_3$ and $k_3c_4=m_4$, by \eqref{eqn:conjugate}, it is
\begin{equation*}
\begin{split}
  &\sum_{m_1^4m_2^3m_3^2m_4=n}\lambda_{\sym^3\phi}(m_2, m_3, m_4)
  \sum_{dc_1=m_1}\sum_{k_1c_2=m_2}\sum_{k_2c_3=m_3}\sum_{k_3c_4=m_4}\tau(k_1c_1, k_2c_2, k_3c_3)\\
  &=\sum_{m_1^4m_2^3m_3^2m_4=n}\lambda_{\sym^3\phi}(m_2, m_3, m_4)\tau(m_1,m_2, m_3, m_4).
\end{split}
\end{equation*}
This finishes the proof of \eqref{eqn:L^5}. Finally by using the bound for $\tau(m_1, m_2, m_3)$ in Lemma \ref{lemma:L-value} and the divisor bound $d_{\ell}(n)\ll n^\varepsilon$, we have $\tau(m_1,m_2,m_3,m_4)\ll (m_1m_2m_3m_4)^\varepsilon.$
\end{proof}

\begin{lemma}\label{lemma:AFE}
  Let $T\leq t_\phi\leq 2T$, we have
  \begin{multline*}
  L(1/2, \sym^3\phi)^4\ll T^{\varepsilon}\int_{|t|\ll T^\varepsilon}\Big|\sum_{d^4m_1^3m_2^2m_1\ll T^{8+\varepsilon}}
  \frac{\lambda_{\sym^3\phi}(m_1, m_2, m_3)\tau(m_1, m_2, m_3)}{(d^4m_1^3m_2^2m_3)^{\frac{1}{2}+\varepsilon+it}}\Big|\dd t
  +O(T^{-2025}).
  \end{multline*}
\end{lemma}
\begin{proof}
  By using the functional equation \eqref{eqn:FEsym3}, Lemma \ref{lemma:L-value} and \cite[Theorem 5.3]{I-K04}, we have
  \[
  L(1/2, \sym^3\phi)^4=2\sum_{d, m_1,m_2, m_3\geq 1}\frac{\lambda_{\sym^3\phi}(m_1, m_2, m_3)\tau(m_1, m_2, m_3)}{(d^4m_1^3m_2^2m_3)^\frac{1}{2}}V(d^4m_1^3m_2^3m_3, t_\phi),
  \]
  where
  \[
  V(d^4m_1^3m_2^3m_3, t_\phi)=\frac{1}{2\pi i}\int_{(3)}\Big(\frac{\gamma_{\sym^3\phi}(1/2+s)}{\gamma_{\sym^3\phi}(1/2)}\Big)^{4}\frac{e^{s^2}}{(d^4m_1^3m_2^2m_3)^s}\frac{\dd s}{s}.
  \]
  By shifting the contour to the right, we see that the contribution of $d^4m_1^3m_2^3m_3< T^{8+\varepsilon}$ is small, say $O(T^{-2025})$. When $d^4m_1^3m_2^3m_3\leq T^{8+\varepsilon},$ we shifted the contour to $\Re(s)=\varepsilon$, by the rapidly decay of $e^{s^2}$ as $|\Im(s)|\gg T^\varepsilon$ and exchange the order of summation and integral, we get
   \begin{multline*}
   L(1/2, \sym^3\phi)^4\ll \int_{|t|\ll T^\varepsilon}\Big|\sum_{d^4m_1^3m_2^2m_1\ll T^{8+\varepsilon}}
  \frac{\lambda_{\sym^3\phi}(m_1, m_2, m_3)\tau(m_1, m_2, m_3)}{(d^4m_1^3m_2^2m_3)^{\frac{1}{2}+\varepsilon+it}}\Big|\\
  \times \Big|\frac{\gamma_{\sym^3\phi}(1/2+\varepsilon+it)}{\gamma_{\sym^3\phi}(1/2)}\Big|^{4}\dd t
  +O(T^{-2025}).
  \end{multline*}
  By Stirling's formula, we finish the proof of Lemma \ref{lemma:AFE}.
\end{proof}
\begin{lemma}\label{lemma:AFE2}
  Let $T\leq t_\phi\leq 2T$, we have
  \begin{multline*}
  L(1/2, \sym^3\phi)^5\ll T^{\varepsilon}\int_{|t|\ll T^\varepsilon}\Big|\sum_{m_1^4m_2^3m_3^2m_4\ll T^{10+\varepsilon}}
  \frac{\lambda_{\sym^3\phi}(m_2, m_3, m_4)\tau(m_1, m_2, m_3, m_4)}{(m_1^4m_2^3m_3^2m_4)^{\frac{1}{2}+\varepsilon+it}}\Big|\dd t
  +O(T^{-2025}).
  \end{multline*}
\end{lemma}
\begin{proof}
  This is a similar argument of the proof of Lemma \ref{lemma:AFE}.
\end{proof}
By Lemma \ref{lemma:AFE} and Theorem \ref{thm:LS_GL4}, and the bound for $\tau$ in Lemma \ref{lemma:L-value}, we have
\begin{equation*}
  \begin{split}
     \sum_{T\leq t_{\phi}\leq T+\Delta}&|L(1/2, \sym^3\phi)|^8\\
     \ll &T^\varepsilon\sum_{T\leq t_{\phi}\leq T+\Delta} \int_{|t|\ll T^\varepsilon}\Big|\sum_{d^4m_1^3m_2^2m_3\ll T^{8+\varepsilon}}\frac{\lambda_{\sym^3\phi}(m_1, m_2, m_3)\tau(m_1, m_2, m_3)}{(d^4m_1^3m_2^2m_3)^{\frac{1}{2}+\varepsilon+it}}\Big|^2\dd t+O(T^{-100})\\
     \ll &T^{\varepsilon}(T^8+T^7\Delta^3)\sup_{N\ll T^{8+\varepsilon}}\sum_{d^4m_1^3m_2^2m_3\sim N}\Big|\frac{\tau(m_1, m_2, m_3)}{(d^4m_1^3m_2^2m_3)^{\frac{1}{2}}}\Big|^2\ll T^{\varepsilon}(T^8+T^7\Delta^3).
  \end{split}
\end{equation*}
Here the dyadic sum is bounded by
\begin{multline*}
N^{-1+\varepsilon}\sum_{d^4m_1^3m_2^2m_3\sim N}1\ll
N^{-1+\varepsilon}\sum_{m_3\ll N}\sum_{m_2\ll (N/m_3)^{1/2}}\sum_{m_1\ll (N/m_2^2m_3)^{1/3}}\sum_{d\sim (N/m_1^3m_2^2m_3)^{1/4}}1\\
\ll N^{-\frac{3}{4}+\varepsilon}\sum_{m_1\ll N}\sum_{m_2\ll (N/m_3)^{1/2}}\sum_{m_3\ll (N/m_2^2m_3)^{1/3}}\frac{1}{(m_1^3m_2^2m_3)^{\frac{1}{4}}}\\
\ll N^{-\frac{2}{3}+\varepsilon}\sum_{m_3\ll N}\sum_{m_2\ll (N/m_3)^{1/2}}\frac{1}{m_2^{\frac{2}{3}}m_3^{\frac{1}{3}}}
\ll N^{-\frac{1}{2}+\varepsilon}\sum_{m_3\ll N}\frac{1}{m_3^{\frac{1}{2}}}\ll N^\varepsilon.
\end{multline*}
This finishes the proof of Theorem \ref{thm:Lmoment}.

By Lemma \ref{lemma:AFE2} and Theorem \ref{thm:LS_GL4}, and the bound for $\tau$ in Lemma \ref{lemma:L-value2}, we have
\begin{equation*}
  \begin{split}
     \sum_{T\leq t_{\phi}\leq 2T}&|L(1/2, \sym^3\phi)|^{10}\\
     \ll &T^\varepsilon\sum_{T\leq t_{\phi}\leq 2T} \int_{|t|\ll T^\varepsilon}\Big|\sum_{m_1^4m_2^3m_3^2m_4\ll T^{10+\varepsilon}}\frac{\lambda_{\sym^3\phi}(m_2, m_3, m_4)\tau(m_1, m_2, m_3, m_4)}{(m_1^4m_2^3m_3^2m_4)^{\frac{1}{2}+\varepsilon+it}}\Big|^2\dd t+O(T^{-100})\\
     \ll &T^{10+\varepsilon}\sup_{N\ll T^{10+\varepsilon}}\sum_{m_1^4m_2^3m_3^2m_4\sim N}\Big|\frac{\tau(m_1, m_2, m_3, m_4)}{(m_1^4m_2^3m_3^2m_4)^{\frac{1}{2}}}\Big|^2\ll T^{10+\varepsilon}.
  \end{split}
\end{equation*}
This finishes the proof of Theorem \ref{thm:Lmoment2}.

Now we prove Theorem \ref{thm:Var}.  By \eqref{eqn:Watson-L}, we have
\begin{equation}\label{eqn:Var=L-Moment}
   \sum_{T\leq t_{\phi}\leq T+\Delta}|\langle 1, \phi^3\rangle|^2\ll T^{-2}\sum_{T\leq t_{\phi}\leq T+\Delta}L(1/2, \sym^3\phi)L(1/2, \phi)^2.
\end{equation}
When $1\leq \Delta\leq T^{\frac{1}{3}}$, applying the H\"older inequality and the non-negative of $L(1/2, \phi)$, we have
\begin{multline*}
\sum_{T\leq t_{\phi}\leq T+\Delta}L(1/2, \sym^3\phi)L(1/2, \phi)^2\\
\leq
\Big(\sum_{T\leq t_{\phi}\leq T+\Delta}1\Big)^{\frac{5}{24}}
\Big(\sum_{T\leq t_{\phi}\leq T+\Delta}L(1/2, \sym^3\phi)^8\Big)^{\frac{1}{8}}
\Big(\sum_{T\leq t_{\phi}\leq T+\Delta}L(1/2, \phi)^{2\cdot\frac{3}{2}}\Big)^{\frac{2}{3}}.
\end{multline*}
Using Theorem \ref{thm:Lmoment} and the cubic moment for $L(1.2, \phi)$ (\cite{Ivic01}):
\[
\sum_{T\leq t_{\phi}\leq T+\Delta}L(1/2, \phi)^{3}\ll T^{1+\varepsilon}\Delta,
\quad \text{ for } 1\leq \Delta\leq T,
\]
we have, for $1\leq \Delta \leq T^{\frac{1}{3}}$,
\[
\sum_{T\leq t_{\phi}\leq T+\Delta}|\langle 1, \phi^3\rangle|^2\ll T^{-\frac{1}{8}+\varepsilon}\Delta^{\frac{7}{8}}.
\]
When $T^{\frac{1}{3}}< \Delta\leq T,$ applying the Cauchy--Schwarz inequality on \eqref{eqn:Var=L-Moment}, we have
\begin{equation*}
   \sum_{T\leq t_{\phi}\leq T+\Delta}|\langle 1, \phi^3\rangle|^2\ll T^{-2}\Big(\sum_{T\leq t_{\phi}\leq T+\Delta}|L(1/2, \sym^3\phi)|^2\Big)^\frac{1}{2}
   \Big(\sum_{T\leq t_{\phi}\leq T+\Delta}|L(1/2, \phi)|^4\Big)^{\frac{1}{2}}.
\end{equation*}
By the H\"older inequality and Theorem \ref{thm:Lmoment}, for  $T^{\frac{1}{3}}<\Delta\leq T$,  we have
\begin{equation*}
\sum_{T\leq t_{\phi}\leq T+\Delta}|L(1/2, \sym^3\phi)|^2\ll \Big(\sum_{T\leq t_{\phi}\leq T+\Delta}|L(1/2, \sym^3\phi)|^8\Big)^{\frac{1}{4}}\Big(\sum_{T\leq t_{\phi}\leq T+\Delta}1\Big)^{\frac{3}{4}}\\
\ll 
T^{\frac{5}{2}+\varepsilon}\Delta^{\frac{3}{2}},
\end{equation*}
By the H\"older inequality and Theorem \ref{thm:Lmoment2}, we have
\begin{equation*}
\sum_{T\leq t_{\phi}\leq T+\Delta}|L(1/2, \sym^3\phi)|^2\ll \Big(\sum_{T\leq t_{\phi}\leq T+\Delta}|L(1/2, \sym^3\phi)|^{10}\Big)^{\frac{1}{5}}\Big(\sum_{T\leq t_{\phi}\leq T+\Delta}1\Big)^{\frac{4}{5}}
\ll T^{\frac{14}{5}+\varepsilon}\Delta^{\frac{4}{5}}.
\end{equation*}
Therefore, we get
\begin{equation*}
\sum_{T\leq t_{\phi}\leq T+\Delta}|L(1/2, \sym^3\phi)|^2\ll
\begin{cases}
T^{\frac{5}{2}+\varepsilon}\Delta^{\frac{3}{2}} & \mbox{ if }  T^{\frac{1}{3}}<\Delta\leq T^{\frac{3}{7}},\\
T^{\frac{14}{5}+\varepsilon}\Delta^{\frac{4}{5}}& \mbox{ if }  T^{\frac{3}{7}}<\Delta\leq T.
\end{cases}
\end{equation*}
Combining with the fourth moment of $L(1/2, \phi)$ (\cite{J-M05}):
\[
\sum_{T\leq t_{\phi}\leq T+\Delta}|L(1/2, \phi)|^4\ll
  T^{1+\varepsilon}\Delta, 
\quad \text{ for }  T^{\frac{1}{3}}<\Delta\leq T,
\]
we have
\[
\sum_{T\leq t_{\phi}\leq T+\Delta}|\langle 1, \phi^3\rangle|^2\ll
\begin{cases}
  T^{-\frac{1}{4}+\varepsilon}\Delta^{\frac{5}{4}}, & \mbox{ if } T^{\frac{1}{3}}< \Delta\leq T^{\frac{3}{7}},\\
   T^{-\frac{1}{10}+\varepsilon}\Delta^{\frac{9}{10}}, & \mbox{ if } T^{\frac{3}{7}}< \Delta\leq T.\\
\end{cases}
\]
This completes the proof of Theorem \ref{thm:Var}.

\section{The variance for cubic moment of Eisenstein series}\label{sec:VarES}
We recall the definition of Eisenstein series:
\[
E(z, s):=\sum_{\gamma\in \Gamma_{\infty}\slash \Gamma}\Im(\gamma z)^s=\frac{y^s}{2}\sum_{c, d\in \mathbb{Z}\atop (c, d)=1}\frac{1}{|cz+d|^{2s}}, \quad (\Re(s)>1)
\]
where $\Gamma_{\infty}$ is the stabilizer of the cusp $\infty$ in $\Gamma=\SL_2(\mathbb{Z}).$ it has a meromorphic continuation to the whole complex plane $\mathbb{C}.$ Moreover, $E(z,s)$ has the following Fourier expansion
\[
E(z, s)=y^s+\frac{\xi(2s-1)}{\xi(2s)}y^{1-2s}+\frac{2}{\xi(2s)}\sum_{n\neq 0}\tau_{s-\frac{1}{2}}(|n|)\sqrt{y}K_{s-\frac{1}{2}}(2\pi |n|y)e(nx),
\]
where $\xi(s):=\pi^{-\frac{s}{2}}\Gamma(\frac{s}{2})\zeta(s)$ is the complete Riemann zeta function and $\tau_u(n)=\sum_{ab=n}(\frac{a}{b})^{iu}$ is the generalized divisor function.
\subsection{Theory of regularized inner product}
Let $F(z)$ be a continuous $\Gamma$-invariant  function on $\mathbb{H}$, we call that $F$ is renormalizable if there is a function $\Phi(y)$ on $\mathbb{R}_{> 0}$ of the form
\begin{equation}\label{eqn:Phi_def}
\Phi(y)=\sum_{j=1}^{l}\frac{c_j}{n_j!}y^{\alpha_j}\log^{n_j}y,
\end{equation}
with $c_j, \alpha_j\in \mathbb{C}$ and $n_j\in \mathbb{Z}_{\geq 0}$, such that
\[
F(z)=\Phi(y)+O(y^{-N})
\]
as $y\to \infty$, and for any $N>0$.

Let the Fourier expansion of $F$ be
\[
F(x+iy)=\sum_{n\in \mathbb{Z}}a_n(y)e(nx)
\]
Assume that $F$ is renormalizable with $\alpha_j\neq 0, 1$ for any $j$, then the function
\[
R(F, s):=\int_{0}^{\infty}(a_0(y)-\Phi(y))y^{s-2}\dd y \quad (\Re(s)\gg 1)
\]
can be meromorphically continued to all $s$ and has a simple pole at $s=1$.
Then we can define the regularized integral with
\[
\int_{\X}^{\reg}F(z)\dd \mu(z):=\frac{\pi}{3}\Res_{s=1}R(F, s).
\]
Moreover, then the function $F(z)E(z, s)$ with $s\neq 0,1$ is also renormalizable and
\[
\int_{\X}^{\reg}F(z)E(z, s)\dd \mu(z)=R(F,s).
\]
Under the assumption that no $\alpha_j=1$, let $\mathcal{E}_{\Phi}(z)$ denote a linear combination of Eisenstein series $E(z, \alpha_j)$ (or suitable derivatives thereof) corresponding to all the exponents in \eqref{eqn:Phi_def} with ${\Re}(\alpha_j) > 1/2$, i.e. such that $F(z) - \mathcal{E}_{\Phi}(z)=O(y^{1/2})$. Then the third equivalent definition of regularization is given by
\begin{equation} \label{eqn:reg_subtract_Eisen}
\int_{\X}^{\reg} F(z) d\mu(z)=\int_{\Gamma \backslash \mathbb{H}} (F(z) -\mathcal{E}_{\Phi}(z)) d\mu(z).
\end{equation}

Now, let $G(z)$ be another renormalizable $\Gamma$-invariant function such that $G(z)=\Psi(y) + O(y^{-N})$ as $y \rightarrow \infty$ for any $N>0$, where $\Psi(y)=\sum_{k=1}^{l_1} \frac{d_k}{m_k!} y^{\beta_k} \log^{m_k}y$ with $d_k, \beta_k \in \mathbb{C}$. Then the product $F(z) \overline{G(z)}$ is also a renormalizable $\Gamma$-invariant function and
if $\alpha_j+ \overline{\beta_k} \neq 1$, for all $\alpha_j$ and $\beta_k$ appearing in $\Phi$ and $\Psi$ respectively,    the regularized inner product of  $F$ and $G$ can be defined as
$$
\langle F, G \rangle_{\reg} :=\int_{\X}^{\reg} F(z) \overline{G(z)} \dd\mu(z)=\int_{\Gamma \backslash \mathbb{H}} (F(z) \overline{G(z)} - \mathcal{E}_{\Phi \overline{\Psi}}(z)) \dd\mu(z).
$$
\begin{lemma}\label{lemma:EE_orth}
Let $s_1, s_2\in \mathbb{C}\backslash\{0,1\}$ with $s_1\neq s_2, 1-s_2$, we have
\[
\int_{\X}^{\reg}E(z, \frac{1}{2} + s_1) E(z, \frac{1}{2} + s_2)\dd\mu(z)=0.
\]
\end{lemma}
\begin{proof}
  See \cite[Page 428]{Zag82}.
\end{proof}
\begin{lemma}\label{lemma:3Eisen}
For all $s_1, s_2, s_3\in \mathbb{C}$, we have
\begin{multline}\label{eqn:3Eisen}
\int_{\X}^{\reg} E(z, \frac{1}{2} + s_1) E(z, \frac{1}{2} + s_2) E(z, \frac{1}{2} + s_3)  \dd\mu(z)=\\
=\frac{\xi(\frac{1}{2} +s_1 +s_2 +s_3) \xi(\frac{1}{2} +s_1 -s_2 +s_3) \xi(\frac{1}{2} +s_1 +s_2 -s_3) \xi(\frac{1}{2} +s_1 -s_2 -s_3)}{\xi(1+ 2 s_1) \xi(1+ 2 s_2) \xi(1+ 2 s_3)}.
\end{multline}
\end{lemma}
\begin{proof}
See \cite[Page 431]{Zag82}. This is also mentioned in \cite[Eqn. (3.6)]{D-K18}.
\end{proof}
Let $\phi$ be in the orthonormal basis $\{\phi_j\}_{j\geq 1}$, the spectral parameter of $\phi$ be $t_{\phi}$, the $n$-th Hecke eigenvalue of $\phi$ be $\lambda_\phi(n)$, then it has the Fourier expansion:
\[
\phi(z)=\rho_{\phi}(1)\sum_{n\neq 0}\lambda_{\phi}(n)\sqrt{y}K_{it_{\phi}}(2\pi|n|y)e(nx).
\]
where $\rho_\phi(1)$ obeys $|\rho_\phi(1)|^2=\frac{2\cosh(\pi t_\phi)}{L(1, \sym^2\phi)}\neq 0$ under our normalization. And we have
$
|\rho_\phi(1)|^2=t_\phi^{o(1)}\exp(\pi t_\phi).
$
We can also introduce the following useful triple product formula involving $\phi$ which is a comparison of Lemma \ref{lemma:3Eisen}. Notice that $\phi$ is rapidly decaying at the cusp $\infty$, we have
$$\langle E(\cdot, 1/2+s_1)E(\cdot, 1/2+s_2), \phi\rangle_{\reg}=\langle E(\cdot, 1/2+s_1)E(\cdot, 1/2+s_2), \phi\rangle.$$
\begin{lemma}\label{lemma:E^2phi}
Let $s_1, s_2\neq \pm 1/2$, if $\phi$ is even, then we have
\begin{equation}\label{eqn:E^2phi}
  \int_{\X}E(z, \frac{1}{2} + s_1) E(z, \frac{1}{2} + s_2)\phi(z)  \dd\mu(z)=
  \frac{\rho_\phi(1)}{2}\frac{\Lambda(1/2+s_1+s_2,\phi)\Lambda(1/2+s_1-s_2,\phi)}{\xi(1+2s_1)\xi(1+2s_2)},
\end{equation}
where $\Lambda(s, \phi):=\pi^{-s}\Gamma(\frac{s+it_\phi}{2})\Gamma(\frac{s-it_\phi}{2})L(s, \phi)$ is the complete $L$-function associated to $\phi$. In the case of odd $\phi$, the triple product is $0$.
\end{lemma}
\begin{proof}
  See \cite[Lemma 4.1]{D-K18}.
\end{proof}
\begin{lemma}\label{Lemma:Regular_Planch} Let $F(z)$ and $G(z)$ be renormalizable functions on $\Gamma\backslash \mathbb{H}$ such that $F - \Phi$ and $G -\Psi$ are of rapid decay as $y \rightarrow \infty$, for some $\Phi(y)=\sum_{j=1}^l \frac{c_j}{n_j!} y^{\alpha_j} \log^{n_j} y$ and $\Psi(y)=\sum_{k=1}^{l_1} \frac{d_k}{m_k!} y^{\beta_k} \log^{m_k}y$. Moreover, let $\alpha_j \neq 1$, $\beta_k \neq 1$, ${\rm Re}(\alpha_j) \neq \frac{1}{2}$, ${\rm Re}(\beta_k) \neq \frac{1}{2}$, $\alpha_j + \overline{\beta_k} \neq 1$ and $\alpha_j \neq \overline{\beta_k}$, for all $j, k$. Then the following formula holds:
\begin{multline}
\langle F, G \rangle_{\reg}
=\langle F, \sqrt{3/ \pi}  \rangle_{\reg} \langle  \sqrt{3/ \pi} , G  \rangle_{\reg} +
\sum_{k\geq 1} \langle F , \phi_k \rangle \langle  \phi_k, G \rangle  \\+ \frac{1}{4 \pi} \int_{-\infty}^{\infty} \langle F, E_{\tau}\rangle_{\reg}  \langle  E_{\tau}, G \rangle_{\reg} \dd \tau
+ \langle F, \mathcal{E}_{\Psi}  \rangle_{\reg}  +  \langle  \mathcal{E}_{\Phi}, G \rangle_{\reg}.
\end{multline}
\end{lemma}
\begin{proof}
  See \cite[Proposition 3.1]{D-K18}.
\end{proof}
Now we compute the cubic moment $\langle \psi, E_t^3\rangle$ through the method of regularized integral.
Let  $F=\psi$ and $G=E_t^3$.
Since that $F=\psi$ is compactly supported, the corresponding $\mathcal{E}_{\Phi}=0$.
By computing the constant term of $G=E_t^3$, we find that the corresponding $\mathcal{E}_{\Psi}$ is
\begin{multline}\label{eqn:E_E^3}
  \mathcal{E}_{\Psi}(z)=E(z, 3/2+3it)+\frac{2\xi(2it)}{\xi(1+2it)}E(z, 3/2+it)\\
  +3\frac{\xi(2it)^2}{\xi(1+2it)^2}E(z,3/2+it)+
\frac{\xi(2it)^3}{\xi(1+2it)^3}E(z, 3/2+3it).
\end{multline}
Applying Lemma \ref{Lemma:Regular_Planch} and  notice that $\langle \psi, E_{\tau}\rangle_{\reg}=\langle  \psi, E_{\tau}\rangle$, $\langle  \psi, \mathcal{E}_{\Psi}\rangle_{\reg}=\langle  \psi, \mathcal{E}_{\Psi}\rangle$, we have
\begin{multline}\label{eqn:Planch}
\langle \psi, E_t^3 \rangle
=\langle \psi, E_t^3 \rangle_{\reg}
=\langle \psi, \sqrt{3/ \pi}  \rangle_{\reg} \langle  \sqrt{3/ \pi} , E_t^3 \rangle_{\reg} +
\sum_{k\geq 1} \langle  \psi , \phi_k \rangle \langle  \phi_k, E_t^3 \rangle  \\+ \frac{1}{4 \pi} \int_{-\infty}^{\infty} \langle  \psi, E_{\tau}\rangle \langle  E_{\tau}, E_t^3 \rangle_{\reg} \dd\tau
+ \langle  \psi, \mathcal{E}_{\Psi}\rangle.
\end{multline}
The constant term, by \eqref{eqn:reg_subtract_Eisen} and Lemma \ref{lemma:3Eisen}, is
\[
\langle \psi, \sqrt{3/ \pi}  \rangle_{\reg} \langle  \sqrt{3/ \pi} , E_t^3 \rangle_{\reg}=
\frac{3}{\pi}\langle \psi,  1\rangle\frac{\xi(1/2+3it)\xi(1/2+it)^2\xi(1/2-it)}{\xi(1-2it)^3},
\]
The last term, by \eqref{eqn:E_E^3}, using the unfolding method, is
\begin{multline}\label{eqn:tail_term}
\langle  \psi, \mathcal{E}_{\Psi}  \rangle=\int_{\X} \psi(z)\overline{\mathcal{E}_{\Psi}(z)}\dd \mu(z)
=\int_{0}^{\infty}a_{0, \psi}(y)\Big(y^{3/2-3it}+\frac{2\xi(-2it)}{\xi(1-2it)}y^{3/2-it}
  \\+3\frac{\xi(-2it)^2}{\xi(1-2it)^2}y^{3/2-it}+
\frac{\xi(-2it)^3}{\xi(1-2it)^3}y^{3/2-3it}\Big)\frac{\dd y}{y^2}.
\end{multline}
where $a_{0, \psi}(y)=\int_{0}^{1}\psi(x+iy)\dd x$ is the constant term in Fourier expansion of $\psi$.
Since that $\psi$ is smooth compactly supported in $\X$, thus $a_{0, \psi}(y)$ is smooth compactly supported in $\mathbb{R}_{>0}$ and $a_{0, \psi}^{(j)}(y)\ll_{\psi, j} 1.$ Note that for real $\sigma\ll 1$ and $t\geq 1$, by Stirling's formula,
$$
 |\xi(\sigma+it)|\asymp |\zeta(\sigma+it)|e^{-\frac{\pi}{4}|t|}|t|^{\frac{\sigma-1}{2}}(1+O_{\sigma}(t^{-1})). $$
 Combining with  the bound $\zeta(it)\ll (|t|+1)^{\frac{1}{2}+\varepsilon}$ and
\begin{equation}\label{eqn:1/zeta}
\frac{1}{\zeta(1+it)}\ll (1+|t|)^\varepsilon,
\end{equation}
we have $\frac{\xi(2it)}{\xi(1+2it)}\ll (|t|+1)^{\varepsilon}.$
Integrating by parts in \eqref{eqn:tail_term} and using the above estimate to bound $\xi$-terms,  we have, with $T\leq t\leq 2T$,
\[
\langle  \psi, \mathcal{E}_{\Psi}  \rangle\ll_{\psi, A}T^{-A},
\]
for any $A>0$.
Combining the above estimate with \eqref{eqn:Planch}, we have, for $T\leq t\leq 2T$,
\begin{multline}\label{eqn:3E_Planch}
  \langle \psi, E_t^3 \rangle
=\frac{3}{\pi}\langle \psi,  1\rangle\frac{\xi(1/2+3it)\xi(1/2+it)^2\xi(1/2-it)}{\xi(1-2it)^3} +
\sum_{k\geq 1} \langle  \psi , \phi_k \rangle \langle  \phi_k, E_t^3 \rangle  \\+ \frac{1}{4 \pi} \int_{-\infty}^{\infty} \langle  \psi, E_{\tau}\rangle  \langle  E_{\tau}, E_t^3 \rangle_{\reg} \dd\tau
+ O_{\psi, A}(T^{-A}).
\end{multline}


\subsection{Proof of Theorem \ref{thm:VarES}}
From \eqref{eqn:3E_Planch}, it remains to handle the contributions of cusp forms and Eisenstein series i.e. the matrix coefficients of observable $\psi$. In order to estimate the variance for $\langle \psi, E_t^3 \rangle$, we shall estimate the variance for these contributions first. Hence in order to prove  Theorem \ref{thm:VarES}, we need the following two lemmas.
\begin{lemma}{\label{lemma:cusp_contri}}
Let $\phi_k$ with spectral parameter $t_k\ll T^\varepsilon$, then we have
\begin{equation*}
  \int_{T}^{2T}\Big|\langle  \phi_k, E_t^3 \rangle \Big|^2\dd t\ll_{\psi, \varepsilon} T^\varepsilon.
\end{equation*}
\end{lemma}
\begin{lemma}{\label{lemma:Eis_contri}}
Let $E_\tau$ with $\tau\ll T^\varepsilon$, then we have
\begin{equation*}
  \int_{T}^{2T}\Big| \langle  E_{\tau}, E_t^3 \rangle_{\reg} \Big|^2\dd t\ll_{\psi, \varepsilon} T^\varepsilon.
\end{equation*}
\end{lemma}
Now we prove Theorem \ref{thm:VarES} under Lemma \ref{lemma:cusp_contri} and Lemma \ref{lemma:Eis_contri}.
By \eqref{eqn:3E_Planch} and the elementary inequality, we have
\begin{multline}\label{eqn:3m_Planc}
\int_{T}^{2T}\Big|\langle \psi, E_t^3 \rangle\Big|^2\dd t\ll_{\psi}
\int_{T}^{2T}\Big|\frac{\xi(1/2+3it)\xi(1/2+it)^2\xi(1/2-it)}{\xi(1-2it)^3}\Big|^2\dd t\\
+\int_{T}^{2T}\Big|\sum_{k\geq 1} \langle  \psi , \phi_k \rangle \langle  \phi_k, E_t^3 \rangle \Big|^2\dd t
+\int_{T}^{2T}\Big|\int_{-\infty}^{\infty} \langle  \psi, E_{\tau}\rangle  \langle  E_{\tau}, E_t^3 \rangle_{\reg} \dd\tau\Big|^2\dd t
+O(T^{-2025}).
\end{multline}
By Stirling's formula and $|\zeta(1-2it)|=T^{o(1)}$ with $T\leq t\leq 2T$, we have
\[
\Big|\frac{\xi(1/2+3it)\xi(1/2+it)^2\xi(1/2-it)}{\xi(1-2it)^3}\Big|^2\ll \frac{1}{T^2}|\zeta(1/2+it)|^6|\zeta(1/2+3it)|^2
\]
Thus the constant term contribution is bounded by
\[
\frac{1}{T^2}\int_{T}^{2T}|\zeta(1/2+it)|^6|\zeta(1/2+3it)|^2\dd t.
\]
Using Cauchy--Schwarz and combining with the Fourth and Twelfth moment estimate for $\zeta$-function \cite{HB78}, it is
\[
\ll \frac{1}{T^2}\Big(\int_{T}^{2T}|\zeta(1/2+it)|^{12}\dd t\int_{T}^{2T}|\zeta(1/2+3it)|^{4}\dd t\Big)^\frac{1}{2}\ll T^{-\frac{1}{2}+\varepsilon}.
\]
It remains to deal with the contributions from cusp forms and Eisenstein series.
Note that
\[
\langle  \psi, \phi_k \rangle
=\frac{1}{(1/4+t_k^2)^{\ell}}\langle  \psi , \Delta_{\mathbb{H}}^{\ell} \phi_k \rangle
=\frac{1}{(1/4+t_k^2)^{\ell}}\langle  \Delta_{\mathbb{H}}^{\ell} \psi , \phi_k \rangle,
\]
\[
\langle  \psi, E_\tau \rangle
=\frac{1}{(1/4+\tau^2)^{\ell}}\langle  \psi , \Delta_{\mathbb{H}}^{\ell} E_\tau  \rangle
=\frac{1}{(1/4+\tau^2)^{\ell}}\langle  \Delta_{\mathbb{H}}^{\ell} \psi , E_\tau  \rangle.
\]
By the Cauchy--Schwarz inequality and QUE for Eisenstein series, we have
\begin{equation*}
  \langle  \psi, \phi_k \rangle\ll_{\psi, A}t_k^{-A},\quad
  \langle  \psi, E_\tau \rangle\ll_{\psi, A}(1+|\tau|)^{-A}.
\end{equation*}
By using the Plancherel formula again and unfolding methods, together with the convexity bounds of $L$-values, we have
\begin{equation*}
   \langle  \phi_k, E_t^3 \rangle\ll (t_kT)^{O(1)},\quad
  \langle   E_\tau, E_t^3 \rangle_{\reg}\ll ((1+|\tau|)T)^{O(1)}.
\end{equation*}
Therefore we can truncate the $t_k$-sum and $\tau$-integral in $t_k\leq T^\varepsilon$ and $|\tau|\leq T^\varepsilon$. Applying Cauchy--Schwarz and Lemma \ref{lemma:cusp_contri}, we have
\begin{equation*}
\begin{split}
\int_{T}^{2T}\Big|\sum_{k\geq 1} \langle  \psi , \phi_k \rangle& \langle  \phi_k, E_t^3 \rangle \Big|^2\dd t
\ll_{\psi}\int_{T}^{2T}\Big|\sum_{t_k\leq T^\varepsilon} \langle  \psi , \phi_k \rangle \langle  \phi_k, E_t^3 \rangle \Big|^2\dd t+O(T^{-10})\\
&\ll_{\psi}T^\varepsilon\sum_{t_k\leq T^\varepsilon} |\langle  \psi , \phi_k \rangle|^2
\ll_{\psi}T^\varepsilon\sum_{t_k\leq T^\varepsilon}\frac{1}{t_k^4}\ll_{\psi, \varepsilon} T^\varepsilon.
\end{split}
\end{equation*}
Similarly, by Lemma \ref{lemma:Eis_contri}, we have
\begin{equation*}
\begin{split}
\int_{T}^{2T}\Big|\int_{-\infty}^{\infty} \langle  \psi, E_{\tau}\rangle&  \langle  E_{\tau}, E_t^3 \rangle_{\reg} \dd\tau\Big|^2\dd t
\ll_{\psi}\int_{T}^{2T}\Big|\int_{-T^\varepsilon}^{T^\varepsilon} \langle  \psi, E_{\tau}\rangle  \langle  E_{\tau}, E_t^3 \rangle_{\reg} \dd\tau\Big|^2\dd t+O(T^{-10})\\
&\ll_{\psi}T^\varepsilon\int_{-T^\varepsilon}^{T^\varepsilon}|\langle  \psi, E_{\tau}\rangle |^2\dd\tau
\ll_{\psi}T^\varepsilon\int_{-T^\varepsilon}^{T^\varepsilon}\frac{\dd \tau}{(1+|\tau|)^4}\ll_{\psi, \varepsilon} T^\varepsilon.
\end{split}
\end{equation*}
Combining the above estimates of each term in \eqref{eqn:3m_Planc}, we get
\[
\int_{T}^{2T}\Big|\langle \psi, E_t^3 \rangle\Big|^2\dd t\ll_{\psi, \varepsilon} T^\varepsilon.
\]
%
\section{The variance for matrix coefficients}\label{sec:TwoLemma}
In this section, we give the proof of Lemma \ref{lemma:cusp_contri} and Lemma \ref{lemma:Eis_contri} assuming Theorem \ref{thm:4m-of-L} and \ref{thm:8m-of-L} respectively.
\subsection{Moments of $\GL(2)$ $L$-functions}
We introduce some Lindel\"of-on-average estimates on $\GL(2)$ $L$-functions.
These bounds can be proved by using the following $\GL(2)$ spectral large sieve.
\begin{lemma}\label{lemma:LS_gl2}
Let $T\geq 1$ and $1\leq \Delta\leq T$, then for any complex sequence $\{a_n\}_{n\geq 1}$ we have
\[
\sum_{|t_j-T|\leq \Delta}\Big|\sum_{n\leq N}a_n\lambda_j(n)\Big|^2\ll (NT)^{\varepsilon}(T\Delta+N)\sum_{n\leq N}|a_n|^2.
\]
\end{lemma}
\begin{proof}
  See \cite{Jut00} and also \cite{J-M05}.
\end{proof}
\begin{lemma}\label{lemma:2m_gl2}
Let $T\geq 1$ and $T\leq t\leq 2T$, we have
\[
\sum_{|t_j-t|\ll T^\varepsilon}|L(1/2+2it, \phi_j)|^2\ll T^{1+\varepsilon}.
\]
\end{lemma}
\begin{proof}
By a similar argument in Lemma \ref{lemma:AFE}, we have
\[
|L(1/2+2it, \phi_j)|\ll T^\varepsilon\int_{|u|\ll T^\varepsilon}\sum_{\pm}\Big|\sum_{n\ll T^{1+\varepsilon}}\frac{\lambda_j(n)}{n^{1/2+\varepsilon\pm2it}}\Big|\dd u+O(T^{-2025}).
\]
Insert it into the second moment, by Cauchy--Schwarz, we get
\begin{equation*}
\sum_{|t_j-t|\ll T^\varepsilon}|L(1/2+2it, \phi_j)|^2\ll T^{\varepsilon}\int_{|u|\ll T^\varepsilon}\sum_{\pm}\sum_{|t_j-t|\ll T^\varepsilon}\Big|\sum_{n\ll T^{1+\varepsilon}}\frac{\lambda_j(n)}{n^{1/2+\varepsilon\pm2it}}\Big|^2\dd u+O(T^{-10}).
\end{equation*}
Applying Lemma \ref{lemma:LS_gl2}, it is bounded by $T^{1+\varepsilon}.$
\end{proof}
\begin{lemma}\label{lemma:4m_gl2}
Let $T\geq 1$, we have
\[
\sum_{T\leq t_j\leq 3T}|L(1/2, \phi_j)|^4\ll T^{2+\varepsilon}.
\]
\end{lemma}
\begin{proof}
It is similar to the proof of Lemma \ref{lemma:2m_gl2}. See also \cite{J-M05}.
\end{proof}
\subsection{Proof of Lemma \ref{lemma:cusp_contri}}
Firstly, we shall compute $\langle  \phi_k, E_t^3 \rangle.$ Let $F=\phi_k\overline{E_t}=\phi_kE_{-t}$ which rapidly decays at the cusp $\infty$, the corresponding $\mathcal{E}_{\Phi}=0$. Let $G(z)=G(z, s_1 ,s_2)=E(z,1/2+s_1)E(z, 1/2+s_2)$, and it suffices to compute
\[
\langle  F(\cdot), G(\cdot, s_1, s_2) \rangle=\int_{\X}F(z)\overline{G(z, s_1,s_2)}\dd \mu(z)
\]
at $s_1=s_2=it.$
By computing the constant term of $G(z)$, we find that the corresponding function of $G$ is
\begin{multline}\label{eqn:ECofE^2}
\mathcal{E}_{\Psi}(z)=E(z,1+s_1+s_2)+
\frac{\xi(2s_1)}{\xi(1+2s_1)}E(z,1-s_1+s_2)\\
+\frac{\xi(2s_2)}{\xi(1+2s_2)}E(z,1+s_1-s_2)
+\frac{\xi(2s_1)\xi(2s_2)}{\xi(1+2s_1)\xi(1+2s_2)}E(z, 1-s_1-s_2).
\end{multline}
By \eqref{eqn:reg_subtract_Eisen}, since that $\phi_k$ rapidly decays at the cusp $\infty$, we have
\[
\langle  \phi_kE_{-t},  E(\cdot, 1/2+s_1)E(\cdot, 1/2+s_2)\rangle
=\langle  \phi_kE_{-t},  E(\cdot, 1/2+s_1)E(\cdot, 1/2+s_2)\rangle_{\reg}.
\]
Then applying Lemma \ref{Lemma:Regular_Planch}, we have
\begin{equation}\label{eqn:Planch_phiE^3}
\begin{split}
  \langle  \phi_kE_{-t},  &E(\cdot, 1/2+s_1)E(\cdot, 1/2+s_2)\rangle
=\langle \phi_kE_{-t}, \sqrt{3/ \pi}  \rangle_{\reg} \langle  \sqrt{3/ \pi}, E(\cdot, 1/2+s_1)E(\cdot, 1/2+s_2)\rangle_{\reg}\\
&+ \sum_{j\geq 1} \langle \phi_kE_{-t}, \phi_j \rangle \langle  \phi_j, E(\cdot, 1/2+s_1)E(\cdot, 1/2+s_2) \rangle  \\
&+ \frac{1}{4 \pi} \int_{-\infty}^{\infty} \langle \phi_kE_{-t}, E_{\nu}\rangle  \langle  E_{\nu}, E(\cdot, 1/2+s_1)E(\cdot, 1/2+s_2)\rangle_{\reg} \dd \nu
+ \langle \phi_kE_{-t}, \mathcal{E}_{\Psi}  \rangle.
\end{split}
\end{equation}
The constant term is vanishing since $\langle \phi_kE_{-t}, \sqrt{3/ \pi}  \rangle_{\reg} =\sqrt{3/ \pi} \langle \phi_k, E_t\rangle=0.$
The tail term, using triple product formula in Lemma \ref{lemma:E^2phi}, is
\begin{equation}\label{eqn:tail_phiE^3}
  \begin{split}
     \langle  \phi_kE_{-t},  &\mathcal{E}_{\Psi}\rangle
     =\int_{\X}\phi_kE_{-t}\overline{\mathcal{E}_{\Psi}}\dd \mu(z)
     =\frac{\rho_k(1)}{2}\frac{\Lambda(1-it+\overline{s_1}+\overline{s_2}, \phi_k)
     \Lambda(-it-\overline{s_1}-\overline{s_2}, \phi_k)}{\xi(1-2it)\xi(2+2\overline{s_1}+2\overline{s_2})}\\
     &+\frac{\rho_k(1)}{2}\frac{\xi(2\overline{s_1})}{\xi(1+2\overline{s_1})}
     \frac{\Lambda(1-it-\overline{s_1}+\overline{s_2}, \phi_k)
     \Lambda(-it+\overline{s_1}-\overline{s_2}, \phi_k)}{\xi(1-2it)\xi(2-2\overline{s_1}+2\overline{s_2})}\\
     &+\frac{\rho_k(1)}{2}\frac{\xi(2\overline{s_2})}{\xi(1+2\overline{s_2})}
     \frac{\Lambda(1-it+\overline{s_1}-\overline{s_2}, \phi_k)
     \Lambda(-it-\overline{s_1}+\overline{s_2}, \phi_k)}{\xi(1-2it)\xi(2+2\overline{s_1}-2\overline{s_2})}\\
     &+\frac{\rho_k(1)}{2}\frac{\xi(2\overline{s_1})}{\xi(1+2\overline{s_1})}\frac{\xi(2\overline{s_2})}{\xi(1+2\overline{s_2})}
     \frac{\Lambda(1-it-\overline{s_1}-\overline{s_2}, \phi_k)\Lambda(-it+\overline{s_1}+\overline{s_2}, \phi_k)}{\xi(1-2it)\xi(2-2\overline{s_1}-2\overline{s_2})}.
  \end{split}
\end{equation}
Since that the right hand side is well-defined for $s_1=s_2=it$, thus when $G=E_t^2$, using the functional equation $\Lambda(s, \phi_k)=\varepsilon_{\phi_k}\Lambda(1-s, \phi_k)$ ($\varepsilon_{\phi_k}=1$ if $\phi_k$ is even, equals $-1$ otherwise), we have
\begin{equation}\label{eqn:tail_phiE^3t}
  \begin{split}
     \varepsilon_{\phi_k}\langle \phi_kE_{-t},  &\mathcal{E}_{\Psi}\rangle
     =\frac{\rho_k(1)}{2}\frac{\Lambda(1-3it, \phi_k)\Lambda(1-it, \phi_k)}{\xi(1-2it)\xi(2-4it)}\\
     &+\frac{6\rho_k(1)}{\pi}\frac{\xi(1+2it)}{\xi(1-2it)}
     \frac{\Lambda(1+it, \phi_k)\Lambda(1-it, \phi_k)}{\xi(1-2it)}\\
     &+\frac{\rho_k(1)}{2}\Big(\frac{\xi(1+2it)}{\xi(1-2it)}\Big)^2
     \frac{\Lambda(1+it, \phi_k)\Lambda(1+3it, \phi_k)}{\xi(1-2it)\xi(2+4it)}.
  \end{split}
\end{equation}
By the unfolding method in the Rankin--Selberg theory (see e.g. \cite[\S 7.2]{Gol06}), we get
\[
\langle \phi_kE_{-t} , \phi_j \rangle =
\begin{cases}
  \frac{\rho_k(1)\rho_j(1)}{4}\frac{\Lambda(1/2+it, \phi_k\times\phi_j)}{\xi(1-2it)},
  & \mbox{ if $\phi_j$ and $\phi_k$ has the same parity,}  \\
  0, & \mbox{ otherwise.}
\end{cases}
\]
By \eqref{eqn:E^2phi}, we get
\[
\langle \phi_j, E(\cdot, 1/2+s_1)E(\cdot, 1/2+s_2) \rangle =
\begin{cases}
  \frac{\rho_j(1)}{2}\frac{\Lambda(1/2+\overline{s_1}+\overline{s_2},\phi_j)
  \Lambda(1/2+\overline{s_1}-\overline{s_2},\phi_j)}{\xi(1+2\overline{s_1})\xi(1+2\overline{s_2})},
  & \mbox{ for even } \phi_j, \\
  0, & \mbox{ for odd } \phi_j.
\end{cases}
\]
Combining the above formulas, the cusp form contribution in \eqref{eqn:Planch_phiE^3} becomes
\begin{multline*}
\sum_{j\geq 1} \langle \phi_kE_{-t} , \phi_j \rangle \langle  \phi_j, E(\cdot, 1/2+s_1)E(\cdot, 1/2+s_2) \rangle
\\
=\frac{\delta_{\even}(\phi_k)\rho_k(1)}{8}\sideset{}{^{\even}}\sum_{j\geq 1}\frac{\rho_j(1)^2\Lambda(1/2+it, \phi_k\times\phi_j)}{\xi(1-2it)}\frac{\Lambda(1/2+\overline{s_1}+\overline{s_2},\phi_j)
\Lambda(1/2+\overline{s_1}-\overline{s_2},\phi_j)}{\xi(1+2\overline{s_1})\xi(1+2\overline{s_2})},
\end{multline*}
where $\delta_{\even}(\phi_k)=1$ if $\phi_k$ is even, otherwise  $\delta_{\even}(\phi_k)=0$. Here $\sideset{}{^{\even}}\sum$ means that the sum runs through all even $\phi_j$.  Taking $s_1=s_2=it$, we have
\begin{multline}
\sum_{j\geq 1} \langle \phi_kE_{-t} , \phi_j \rangle \langle  \phi_j, E_t^2\rangle
\\
=\frac{\delta_{\even}(\phi_k)\rho_k(1)}{8}\sideset{}{^{\even}}\sum_{j\geq 1}\frac{\rho_j(1)^2\Lambda(1/2+it, \phi_k\times\phi_j)\Lambda(1/2+2it,\phi_j)
\Lambda(1/2,\phi_j)}{\xi(1-2it)^3}.
\end{multline}

Similarly,  by Lemma \ref{lemma:E^2phi} and Lemma \ref{lemma:3Eisen}, the Eisenstein series contribution in \eqref{eqn:Planch_phiE^3} becomes
\begin{equation*}
\begin{split}
\frac{1}{4 \pi}&\int_{-\infty}^{\infty}\langle \phi_kE_{-t}, E_{\nu}\rangle  \langle  E_{\nu}, E(\cdot, 1/2+s_1)E(\cdot, 1/2+s_2)\rangle_{\reg} \dd \nu
\\&=\delta_{\even}(\phi_k)\frac{\rho_k(1)}{8\pi}\int_{-\infty}^{\infty}
\frac{\Lambda(1/2-it-i\nu,\phi_k)
  \Lambda(1/2+it-i\nu,\phi_k)}{\xi(1-2it)\xi(1-2i\nu)}\\
&\quad \cdot
\frac{\xi(1/2+i\nu +\overline{s_1} +\overline{s_2}) \xi(1/2+i\nu -\overline{s_1}+\overline{s_2})\xi(1/2+i\nu+\overline{s_1} -\overline{s_2}) \xi(1/2+i\nu -\overline{s_1} -\overline{s_2})}{\xi(1+ 2i\nu) \xi(1+ 2\overline{s_1}) \xi(1+ 2\overline{s_2})}
\dd \nu.
\end{split}
\end{equation*}
Taking $s_1=s_2=it$, we have
\begin{equation*}
\begin{split}
\frac{1}{4 \pi}&\int_{-\infty}^{\infty}\langle \phi_kE_{-t}, E_{\nu}\rangle  \langle  E_{\nu}, E_t^2\rangle_{\reg} \dd \nu
\\&=\delta_{\even}(\phi_k)\frac{\rho_k(1)}{8\pi}\int_{-\infty}^{\infty}
\frac{\prod_{\pm}\Lambda(1/2\pm it-i\nu,\phi_k)\prod_{\pm}\xi(1/2+iv\pm2it)
   \xi(1/2+i\nu)^2}
  {|\xi(1-2i\nu)|^2\xi(1-2it)^3}
\dd \nu.
\end{split}
\end{equation*}
Combining the above formulas together, we have the following explicit expression about $\langle  \phi_k, E_t^3 \rangle$.
\begin{lemma}\label{lemma:explicit_phiE^3}
For even $\phi_k \in \{\phi_j\}_{j\geq 1}$, we have
\begin{equation*}
\begin{split}
\langle  \phi_k, E_t^3 \rangle
=&\frac{\rho_k(1)}{8}\sideset{}{^{\even}}\sum_{j\geq 1}\frac{\rho_j(1)^2\Lambda(1/2+it, \phi_k\times\phi_j)\Lambda(1/2+2it,\phi_j)
\Lambda(1/2,\phi_j)}{\xi(1-2it)^3}\\
&+\frac{\rho_k(1)}{8\pi}\int_{-\infty}^{\infty}
\frac{\prod_{\pm}\Lambda(1/2\pm it-i\nu,\phi_k)
\prod_{\pm}\xi(1/2+i\nu \pm2it) \xi(1/2+i\nu)^2}
{|\xi(1-2i\nu)|^2\xi(1-2it)^3}
\dd \nu\\
&+\frac{\rho_k(1)}{2}\frac{\Lambda(1-3it, \phi_k)\Lambda(1-it, \phi_k)}{\xi(1-2it)\xi(2-4it)}
+\frac{6\rho_k(1)}{\pi}\frac{\xi(1+2it)}{\xi(1-2it)}
\frac{\Lambda(1+it, \phi_k)\Lambda(1-it, \phi_k)}{\xi(1-2it)}\\
&+\frac{\rho_k(1)}{2}\Big(\frac{\xi(1+2it)}{\xi(1-2it)}\Big)^2
     \frac{\Lambda(1+it, \phi_k)\Lambda(1+3it, \phi_k)}{\xi(1-2it)\xi(2+4it)}.\\
\end{split}
\end{equation*}
And for odd $\phi_k \in \{\phi_j\}_{j\geq 1}$, we have
\begin{equation*}
\begin{split}
\langle  \phi_k, E_t^3& \rangle
=-\frac{\rho_k(1)}{2}\frac{\Lambda(1-3it, \phi_k)\Lambda(1-it, \phi_k)}{\xi(1-2it)\xi(2-4it)}
-\frac{6\rho_k(1)}{\pi}\frac{\xi(1+2it)}{\xi(1-2it)}\frac{\Lambda(1+it, \phi_k)\Lambda(1-it, \phi_k)}{\xi(1-2it)}\\
&-\frac{\rho_k(1)}{2}\Big(\frac{\xi(1+2it)}{\xi(1-2it)}\Big)^2
     \frac{\Lambda(1+it, \phi_k)\Lambda(1+3it, \phi_k)}{\xi(1-2it)\xi(2+4it)}.\\
\end{split}
\end{equation*}
\end{lemma}
By using Stirling's formula, \eqref{eqn:1/zeta} and the bound
\[
\frac{1}{L(1, \sym^2\phi_j)}\ll t_j^\varepsilon,
\]
we have, for $t\geq 1$,
\begin{multline*}
\frac{\rho_k(1)\rho_j(1)^2\Lambda(1/2+it, \phi_k\times\phi_j)\Lambda(1/2+2it,\phi_j)
\Lambda(1/2,\phi_j)}{\xi(1-2it)^3}\\
\ll \frac{(tt_kt_j)^\varepsilon\exp(-\frac{\pi}{2}Q(t_j; t_k, t))}
{t_j^\frac{1}{2}\Big(\prod_{\pm,\pm}(1+|t\pm t_j\pm t_k|)\prod_{\pm}(1+|2t\pm t_j|)\Big)^{\frac{1}{4}}}
|L(1/2+it, \phi_k\times\phi_j)L(1/2+2it, \phi_j)L(1/2, \phi_j)|,
\end{multline*}
where
\begin{equation}\label{eqn:Q}
Q(t_j; t_k, t)=
-|t_k|-2|t_j|+\frac{1}{2}\sum_{\pm, \pm}|t_j\pm t\pm t_k|+\frac{1}{2}\sum_{\pm}|t_j\pm 2t|-3|t|
\end{equation}
is an even function on variable $t_j$. Specifically, with $0\leq t_k\leq t$,
\[
Q(t_j; t_k, t)=
\begin{cases}
  t-t_k-t_j, & \mbox{if } 0\leq t_j< t-t_k, \\
  0, & \mbox{if } t-t_k\leq t_j<t+t_k,\\
  t_j-t_k-t, & \mbox{if } t+t_k\leq t_j< 2t, \\
  2t_j-3t-t_k, & \mbox{if } 2t\leq t_j. \\
\end{cases}
\]
For the continuous spectrum part, we have
\begin{multline*}
\frac{\rho_k(1)\prod_{\pm}\Lambda(1/2\pm it-i\nu,\phi_k)
\prod_{\pm}\xi(1/2+i\nu \pm2it) \xi(1/2+i\nu)^2}
{|\xi(1-2i\nu)|^2\xi(1-2it)^3}\\
\ll \frac{(tt_k(1+|\nu|))^\varepsilon\exp(-\frac{\pi}{2}Q(\nu; t_k, t))
|\zeta(1/2+i\nu)|^2\prod_{\pm, \pm}|L(1/2+it\pm i\nu, \phi_k)\zeta(1/2+2it\pm i\nu)|}
{(1+|\nu|)^\frac{1}{2}\Big(\prod_{\pm,\pm}(1+|t\pm \nu\pm t_k|)\prod_{\pm}(1+|2t\pm \nu|)\Big)^{\frac{1}{4}}}.
\end{multline*}
When $t_k\ll t^{1-\varepsilon}$, by using \eqref{eqn:1/zeta}, the trivial bounds for $L$-functions and Stirling's formula, we have
\[
\frac{1}{\xi(1-2it)}\ll  t^{O(1)}e^{\frac{\pi}{2} |t|}, \quad 
\frac{1}{\xi(2-4it)}\ll  t^{O(1)}e^{\pi |t|}, \quad 
\Lambda(1+it, \phi_k\times\phi_j)\ll t^{O(1)}e^{-\pi|t|}.
\]
Thus the tail terms in Lemma \ref{lemma:explicit_phiE^3} is rapidly decaying as $t\to \infty.$ Thus combining the above bounds with Lemma \ref{lemma:explicit_phiE^3} in the case of $t_k\ll T^\varepsilon$ and $T\leq t\leq 2T$,  note that one can truncate $t_j$-sum and $\nu$-integral in the region $[t-T^\varepsilon, t+T^\varepsilon]$ which produce a negligible error term, we then have
\begin{lemma}\label{lemma:bound_phiE^3}
Let $T\geq 1$, $T\leq t\leq 2T$ and $\phi_k \in \{\phi_j\}_{j\geq 1}$ with $t_k\ll T^{\varepsilon},$
for even $\phi_k$, we have
\begin{equation*}
\begin{split}
\langle  \phi_k, E_t^3 \rangle
\ll&T^{-\frac{3}{2}+\varepsilon}\sideset{}{^{\even}}\sum_{|t_j-t|\ll T^\varepsilon}|L(1/2+it, \phi_k\times\phi_j)L(1/2+2it, \phi_j)L(1/2, \phi_j)|
\\
&+T^{-\frac{3}{2}+\varepsilon}\int_{|\nu-t|\ll  T^\varepsilon}
|\zeta(1/2+i\nu)|^2\prod_{\pm, \pm}|L(1/2+it\pm i\nu, \phi_k)\zeta(1/2+2it\pm i\nu)|
\dd \nu+O(T^{-A}).
\end{split}
\end{equation*}
And for odd $\phi_k$, we have
$\langle  \phi_k, E_t^3 \rangle\ll T^{-A}$ for any $A>1$.
\end{lemma}

Now we give the proof of Lemma \ref{lemma:cusp_contri}.
\begin{proof}[Proof of Lemma \ref{lemma:cusp_contri} by using Theorem \ref{thm:4m-of-L}]
Using Lemma \ref{lemma:bound_phiE^3}, we can assume that $\phi_k$ is even and we have
\begin{multline}\label{eqn:bound_phiE^3}
\int_{T}^{2T}|\langle  \phi_k, E_t^3 \rangle|^2\dd t
\ll T^{-3+\varepsilon}\int_{T}^{2T}\Big(\sideset{}{^{\even}}\sum_{|t_j-t|\ll T^\varepsilon}|L(1/2+it, \phi_k\times\phi_j)L(1/2+2it, \phi_j)L(1/2, \phi_j)|\Big)^2\dd t\\
+T^{-3+\varepsilon}\int_{T}^{2T}\Big(\int_{|\nu-t|\ll  T^\varepsilon}
|\zeta(1/2+i\nu)|^2\prod_{\pm, \pm}|L(1/2+it\pm i\nu, \phi_k)\zeta(1/2+2it\pm i\nu)|\Big)^2\dd \nu.
\end{multline}
The first term, using Cauchy--Schwarz and Lemma \ref{lemma:2m_gl2}, is
\begin{equation*}
\begin{split}
 T^{-3+\varepsilon}&\int_{T}^{2T}\Big(\sideset{}{^{\even}}\sum_{|t_j-t|\ll T^\varepsilon}|L(1/2+it, \phi_k\times\phi_j)L(1/2+2it, \phi_j)L(1/2, \phi_j)|\Big)^2\dd t\\
 &\ll
 T^{-3+\varepsilon}\int_{T}^{2T}\sideset{}{^{\even}}\sum_{|t_j-t|\ll T^\varepsilon}|L(1/2+it, \phi_k\times\phi_j)|^2L(1/2, \phi_j)|^2
 \sum_{|t_j-t|\ll T^\varepsilon}|L(1/2+2it, \phi_j)|^2\dd t\\
 &\ll T^{-2+\varepsilon}\int_{T}^{2T}\sideset{}{^{\even}}\sum_{|t_j-t|\ll T^\varepsilon}|L(1/2+it, \phi_k\times\phi_j)|^2L(1/2, \phi_j)|^2\dd t\\
 &\ll T^{-2+\varepsilon}\sideset{}{^{\even}}\sum_{T/2\leq t_j\leq 3T}|L(1/2, \phi_j)|^2\int_{|t_j-t|\ll T^\varepsilon}|L(1/2+it, \phi_k\times\phi_j)|^2\dd t,
\end{split}
\end{equation*}
by exchanging the $t_j$-sum and $t$-integral. Then using Cauchy--Schwarz, Theorem \ref{thm:4m-of-L} and Lemma \ref{lemma:4m_gl2}, it is bounded by
\begin{equation*}\label{eqn:cusp_boundbyL}
\begin{split}
T^{-2+\varepsilon}&\int_{|\alpha|\ll T^{2\varepsilon}}\quad\sideset{}{^{\even}}\sum_{T/2\leq t_j\leq 3T}|L(1/2+it_j+i\alpha, \phi_k\times\phi_j)|^2
|L(1/2, \phi_j)|^2\dd \alpha\\
&\ll T^{-2+\varepsilon}\sup_{|\alpha|\ll T^{2\varepsilon}}
\Big(\sideset{}{^{\even}}\sum_{T/2\leq t_j\leq 3T}|L(1/2+it_j+i\alpha, \phi_k\times\phi_j)|^4\Big)^{\frac{1}{2}}
\Big(\sum_{T/2\leq t_j\leq 3T}|L(1/2, \phi_j)|^4\Big)^{\frac{1}{2}}\\
&\ll T^\varepsilon.
\end{split}
\end{equation*}
The last term in \eqref{eqn:bound_phiE^3},  using the convexity bound for $L(s, \phi_k)$ and Cauchy--Schwarz, is bounded by
\begin{multline*}
T^{-2+\varepsilon}\int_{T}^{2T}\Big(\int_{|\nu-t|\ll  T^\varepsilon}
|\zeta(1/2+i\nu)|^2|\zeta(1/2+2it+i\nu)\zeta(1/2+2it+i\nu)|\Big)^2\dd \nu\\
\ll T^{-2+\varepsilon}\int_{T}^{2T}
\int_{|\nu-t|\ll  T^\varepsilon}
|\zeta(1/2+i\nu)|^4|\zeta(1/2+2it+i\nu)\zeta(1/2+2it+i\nu)|^2\dd \nu\dd t\\
\ll T^{-2+\varepsilon}\int_{T/2}^{3T}
|\zeta(1/2+i\nu)|^4\int_{|\nu-t|\ll  T^\varepsilon}|\zeta(1/2+2it+i\nu)\zeta(1/2+2it+i\nu)|^2\dd t\dd \nu.
\end{multline*}
Then using the Weyl bound $\zeta(1/2+it)\ll (1+|t|)^{1/6+\varepsilon},$ and the fourth moment bound for $\zeta$, it is bounded by
$O(T^{-\frac{1}{3}+\varepsilon}).$
Combining these two estimates for terms in \eqref{eqn:bound_phiE^3}, we complete the proof of Lemma \ref{lemma:cusp_contri}.
\end{proof}

\subsection{Proof of Lemma \ref{lemma:Eis_contri}}
Similar to the previous subsection, we shall compute $\langle  E_\tau, E_t^3 \rangle_{\reg}$ with $|\tau|<t$.
Let new $F=E_\tau\overline{E_t}=E_\tau E_{-t}$, by \eqref{eqn:ECofE^2}, the corresponding function $\mathcal{E}_{\Phi}$ is
\begin{multline*}
\mathcal{E}_{\Phi}(z)=E(z,1+i\tau-it)+
\frac{\xi(2i\tau)}{\xi(1+2i\tau)}E(z,1-i\tau-it)\\
+\frac{\xi(-2it)}{\xi(1-2it)}E(z,1+i\tau+it)
+\frac{\xi(2i\tau)\xi(-2it)}{\xi(1+2i\tau)\xi(1-2it)}E(z, 1-i\tau+it).
\end{multline*}
Let $G(z)=G(z, s_1 ,s_2)=E(z,1/2+s_1)E(z, 1/2+s_2)$, its corresponding function $\mathcal{E}_{\Psi}$ is given by \eqref{eqn:ECofE^2}.
It suffices to compute
\[
\langle  F(\cdot), G(\cdot, s_1, s_2) \rangle  =\int_{\X}^{\reg}F(z)\overline{G(z, s_1,s_2)}\dd \mu(z)
\]
at $s_1=s_2=it.$
Applying Lemma \ref{Lemma:Regular_Planch}, we have
\begin{equation}\label{eqn:Planch_E,E^3}
\begin{split}
\langle  E_\tau E_{-t},  &E(\cdot, 1/2+s_1)E(\cdot, 1/2+s_2)\rangle_{\reg}
=\langle E_\tau E_{-t}, \sqrt{3/ \pi}  \rangle_{\reg} \langle  \sqrt{3/ \pi}, E(\cdot, 1/2+s_1)E(\cdot, 1/2+s_2)\rangle_{\reg}\\
&+ \sum_{j\geq 1} \langle E_\tau E_{-t} , \phi_j \rangle \langle  \phi_j, E(\cdot, 1/2+s_1)E(\cdot, 1/2+s_2) \rangle  \\
&+ \frac{1}{4 \pi} \int_{-\infty}^{\infty} \langle E_\tau E_{-t}, E_{\nu}\rangle_{\reg}  \langle  E_{\nu}, E(\cdot, 1/2+s_1)E(\cdot, 1/2+s_2)\rangle_{\reg} \dd \nu\\
&+ \langle E_\tau E_{-t}, \mathcal{E}_{\Psi}  \rangle_{\reg}+\langle \mathcal{E}_{\Phi}, E(\cdot, 1/2+s_1)E(\cdot, 1/2+s_2) \rangle_{\reg}.
\end{split}
\end{equation}
By Lemma \ref{lemma:EE_orth}, the constant term vanishes.
The two tail terms, using the regularized triple product formula in Lemma \ref{lemma:3Eisen}, are
\begin{equation*}
  \begin{split}
     \langle  &E_\tau E_{-t},  \mathcal{E}_{\Psi}\rangle_{\reg}
     =\int_{\X}^{\reg} E_\tau E_{-t}\overline{\mathcal{E}_{\Psi}}\dd \mu(z)\\
     &=\frac{\prod_{\pm, \pm}\xi(1+\overline{s_1}+\overline{s_2}\pm i\tau\pm it)                                                                         }
     {\xi(1+2i\tau)\xi(1-2it)\xi(2+2\overline{s_1}+2\overline{s_2})}+\frac{\xi(2\overline{s_1})}{\xi(1+2\overline{s_1})}
     \frac{\prod_{\pm, \pm}\xi(1-\overline{s_1}+\overline{s_2}\pm i\tau\pm it)}
     {\xi(1+2i\tau)\xi(1-2it)\xi(2-2\overline{s_1}+2\overline{s_2})}\\
     &\quad+\frac{\xi(2\overline{s_2})}{\xi(1+2\overline{s_2})}
     \frac{\prod_{\pm, \pm}\xi(1+\overline{s_1}-\overline{s_2}\pm i\tau\pm it)}
     {\xi(1+2i\tau)\xi(1-2it)\xi(2+2\overline{s_1}-2\overline{s_2})}\\
     &\quad\quad +\frac{\xi(2\overline{s_1})\xi(2\overline{s_2})}{\xi(1+2\overline{s_1})\xi(1+2\overline{s_2})}
     \frac{\prod_{\pm, \pm}\xi(1-\overline{s_1}-\overline{s_2}\pm i\tau\pm it)}
     {\xi(1+2i\tau)\xi(1-2it)\xi(2-2\overline{s_1}-2\overline{s_2})},
  \end{split}
\end{equation*}
and
\begin{equation*}
  \begin{split}
     &\langle  \mathcal{E}_{\Phi},  E(\cdot, 1/2+s_1)E(\cdot, 1/2+s_2) \rangle_{\reg}
     =\int_{\X}^{\reg}\mathcal{E}_{\Phi}\overline{E(\cdot, 1/2+s_1)E(\cdot, 1/2+s_2) }\dd \mu(z)\\
     &=\frac{\prod_{\pm, \pm}\xi(1+i\tau-it\pm\overline{s_1}\pm\overline{s_2})}
     {\xi(2+2i\tau-2it)\xi(2+2\overline{s_1})\xi(2+2\overline{s_2})}
     +\frac{\xi(2i\tau)}{\xi(1+2i\tau)}
     \frac{\prod_{\pm, \pm}\xi(1-i\tau-it\pm\overline{s_1}\pm\overline{s_2})}
     {\xi(2-2i\tau-2it)\xi(2+2\overline{s_1})\xi(2+2\overline{s_2})}\\
     &\quad+\frac{\xi(-2it)}{\xi(1-2it)}
     \frac{\prod_{\pm, \pm}\xi(1+i\tau+it\pm\overline{s_1}\pm\overline{s_2})}
     {\xi(2+2i\tau+2it)\xi(2+2\overline{s_1})\xi(2+2\overline{s_2})}\\
     &\quad\quad+\frac{\xi(2i\tau)\xi(-2it)}{\xi(1+2i\tau)\xi(1-2it)}
     \frac{\prod_{\pm, \pm}\xi(1-i\tau+it\pm\overline{s_1}\pm\overline{s_2})}
     {\xi(2-2i\tau+2it)\xi(2+2\overline{s_1})\xi(2+2\overline{s_2})}.
  \end{split}
\end{equation*}
Taking $s_1=s_2=it$ in the function $G(\cdot, s_1,s_2)$ and combining the function equation $\xi(s)=\xi(1-s)$, these become
\begin{multline*}
     \langle  E_\tau E_{-t},  \mathcal{E}_{\Psi}\rangle_{\reg}
     =\frac{\prod_{\pm}\xi(1+3it\pm i\tau)\prod_{\pm}\xi(1+it\pm i\tau)}                                                                   
     {\xi(1+2i\tau)\xi(1+2it)\xi(2-4it)}
     +\frac{12}{\pi}\frac{\xi(1+2it)}{\xi(1-2it)}
     \frac{\prod_{\pm,\pm}\xi(1\pm i\tau\pm it)}
     {\xi(1+2i\tau)\xi(1-2it)}\\
     +\Big(\frac{\xi(1+2it)}{\xi(1-2it)}\Big)^2
     \frac{\prod_{\pm}\xi(1+3it\pm i\tau)\prod_{\pm}\xi(1+it\pm i\tau)}
     {\xi(1+2i\tau)\xi(1-2it)\xi(2+4it)},
\end{multline*}
and
\begin{equation*}
\begin{split}
     \langle  \mathcal{E}_{\Phi},  E_t^2 \rangle_{\reg}
     =&\frac{\xi(1+i\tau-3it)\xi(1+i\tau-it)^2\xi(1+i\tau+it)}
     {\xi(2+2i\tau-2it)\xi(2-2it)^2}\\
     &+\frac{\xi(1-2i\tau)}{\xi(1+2i\tau)}
     \frac{\xi(1-i\tau-3it)\xi(1-i\tau-it)^2\xi(1-i\tau+it)}
     {\xi(2-2i\tau-2it)\xi(2-2it)^2}\\
      &\quad +\frac{\xi(1+2it)}{\xi(1-2it)}
     \frac{\xi(1+i\tau+3it)\xi(1+i\tau+it)^2\xi(1+i\tau-it)}
     {\xi(2+2i\tau+2it)\xi(2-2it)^2}\\
     &\quad\quad+\frac{\xi(1-2i\tau)\xi(1+2it)}{\xi(1+2i\tau)\xi(1-2it)}
     \frac{\xi(1-i\tau+3it)\xi(1-i\tau+it)^2\xi(1-i\tau-it)}
     {\xi(2-2i\tau+2it)\xi(2-2it)^2}.
\end{split}
\end{equation*}
For the cusp form contribution, using Lemma \ref{lemma:E^2phi}, we have
\begin{multline*}
  \sum_{j\geq 1} \langle E_\tau E_{-t} , \phi_j \rangle \langle  \phi_j, E(\cdot, 1/2+s_1)E(\cdot, 1/2+s_2) \rangle
  \\=\frac{1}{4}\sideset{}{^{\even}}\sum_{j\geq 1}\frac{\rho_j(1)^2
  \prod_{\pm}\Lambda(1/2+i\tau\pm it, \phi_j)\prod_{\pm}\Lambda(1/2+\overline{s_1}\pm \overline{s_2}, \phi_j)}
  {\xi(1+2i\tau)\xi(1-2it)\xi(1+2\overline{s_1})\xi(1+2\overline{s_2})}.
\end{multline*}
Taking $s_1=s_2=it$, it becomes
\begin{equation*}
  \sum_{j\geq 1} \langle E_\tau E_{-t} , \phi_j \rangle \langle  \phi_j, E_t^2 \rangle
  =\frac{1}{4}\sideset{}{^{\even}}\sum_{j\geq 1}\frac{\rho_j(1)^2
  \Lambda(1/2, \phi_j)\Lambda(1/2+2it, \phi_j)\prod_{\pm}\Lambda(1/2+i\tau\pm it, \phi_j)}
  {\xi(1+2i\tau)\xi(1-2it)^3}.
\end{equation*}
Similarly, the Eisenstein series contribution, using Lemma \ref{lemma:3Eisen}, is 
\begin{multline*}
  \frac{1}{4 \pi} \int_{-\infty}^{\infty} \langle E_\tau E_{-t}, E_{\nu}\rangle_{\reg}  \langle  E_{\nu}, E(\cdot, 1/2+s_1)E(\cdot, 1/2+s_2)\rangle_{\reg} \dd \nu\\
  =\frac{1}{4 \pi} \int_{-\infty}^{\infty}\frac{\prod_{\pm, \pm}\xi(1/2+i\nu\pm i\tau\pm it)
  \prod_{\pm, \pm}\xi(1/2+i\nu\pm \overline{s_1}\pm \overline{s_2})}
  {|\xi(1+2i\nu)|^2\xi(1+2i\tau)\xi(1-2it)\xi(1+2\overline{s_1})\xi(1+2\overline{s_2})}\dd \nu.
\end{multline*}
Taking $s_1=s_2=it$, it becomes
\begin{equation*}
\frac{1}{4 \pi} \int_{-\infty}^{\infty}\frac{\xi(1/2+i\nu)^2
\prod_{\pm}\xi(1/2+i\nu\pm 2it)\prod_{\pm, \pm}\xi(1/2+i\nu\pm i\tau\pm it)}
  {|\xi(1+2i\nu)|^2\xi(1+2i\tau)\xi(1-2it)^3}\dd \nu.
\end{equation*}
Combining all above contributions, 
we have the following explicit expression about $\langle  E_\tau, E_t^3 \rangle_{\reg}$.
\begin{lemma}\label{lemma:explicit_EE^3}
We have
\begin{equation*}
\begin{split}
\langle  &E_\tau, E_t^3 \rangle_{\reg}
=\frac{1}{4}\sideset{}{^{\even}}\sum_{j\geq 1}\frac{\rho_j(1)^2
  \Lambda(1/2, \phi_j)\Lambda(1/2+2it, \phi_j)\prod_{\pm}\Lambda(1/2+i\tau\pm it, \phi_j)}
  {\xi(1+2i\tau)\xi(1-2it)^3}\\
&+\frac{1}{4 \pi} \int_{-\infty}^{\infty}\frac{\xi(1/2+i\nu)^2
\prod_{\pm}\xi(1/2+i\nu\pm 2it)\prod_{\pm, \pm}\xi(1/2+i\nu\pm i\tau\pm it)}
  {|\xi(1+2i\nu)|^2\xi(1+2i\tau)\xi(1-2it)^3}\dd \nu\\
&+\frac{\prod_{\pm}\xi(1+3it\pm i\tau)\prod_{\pm}\xi(1+it\pm i\tau)}                                                                   
     {\xi(1+2i\tau)\xi(1+2it)\xi(2-4it)}
     +\frac{12}{\pi}\frac{\xi(1+2it)}{\xi(1-2it)}
     \frac{\prod_{\pm,\pm}\xi(1\pm i\tau\pm it)}
     {\xi(1+2i\tau)\xi(1-2it)}\\
&+\Big(\frac{\xi(1+2it)}{\xi(1-2it)}\Big)^2
     \frac{\prod_{\pm}\xi(1+3it\pm i\tau)\prod_{\pm}\xi(1+it\pm i\tau)}
     {\xi(1+2i\tau)\xi(1-2it)\xi(2+4it)}\\
&+\frac{\xi(1+i\tau-3it)\xi(1+i\tau-it)^2\xi(1+i\tau+it)}
     {\xi(2+2i\tau-2it)\xi(2-2it)^2}\\
     &+\frac{\xi(1-2i\tau)}{\xi(1+2i\tau)}
     \frac{\xi(1-i\tau-3it)\xi(1-i\tau-it)^2\xi(1-i\tau+it)}
     {\xi(2-2i\tau-2it)\xi(2-2it)^2}\\
      &+\frac{\xi(1+2it)}{\xi(1-2it)}
     \frac{\xi(1+i\tau+3it)\xi(1+i\tau+it)^2\xi(1+i\tau-it)}
     {\xi(2+2i\tau+2it)\xi(2-2it)^2}\\
     &+\frac{\xi(1-2i\tau)\xi(1+2it)}{\xi(1+2i\tau)\xi(1-2it)}
     \frac{\xi(1-i\tau+3it)\xi(1-i\tau+it)^2\xi(1-i\tau-it)}
     {\xi(2-2i\tau+2it)\xi(2-2it)^2}.\\
\end{split}
\end{equation*}
\end{lemma}
Then by using Stirling's formula and \eqref{eqn:1/zeta}
we have, for $t\geq 1$,
\begin{multline*}
\frac{\rho_j(1)^2
\Lambda(1/2, \phi_j)\Lambda(1/2+2it, \phi_j)\prod_{\pm}\Lambda(1/2+i\tau\pm it, \phi_j)}
{\xi(1+2i\tau)\xi(1-2it)^3}\\
\ll \frac{((1+|\tau|)tt_j)^\varepsilon\exp(-\frac{\pi}{2}Q(t_j; \tau, t))
|L(1/2+2it, \phi_j)|L(1/2, \phi_j)\prod_{\pm}|L(1/2+i\tau\pm it, \phi_j)|}
{t_j^\frac{1}{2}\Big(\prod_{\pm,\pm}(1+|t\pm t_j\pm \tau|)\prod_{\pm}(1+|2t\pm t_j|)\Big)^{\frac{1}{4}}},
\end{multline*}
where $Q(t_j; \tau, t)$ has been defined in \eqref{eqn:Q} already.
Similarly, we have
\begin{multline*}
\frac{\xi(1/2+i\nu)^2
\prod_{\pm}\xi(1/2+i\nu\pm 2it)\prod_{\pm, \pm}\xi(1/2+i\nu\pm i\tau\pm it)}
  {|\xi(1+2i\nu)|^2\xi(1+2i\tau)\xi(1-2it)^3}\\
\ll \frac{(t(1+|\tau|)(1+|\nu|))^\varepsilon\exp(-\frac{\pi}{2}Q(\nu; t_k, t))}
{(1+|\nu|)^\frac{1}{2}\Big(\prod_{\pm,\pm}(1+|t\pm \nu\pm \tau|)\prod_{\pm}(1+|2t\pm \nu|)\Big)^{\frac{1}{4}}}\\
\cdot |\zeta(1/2+i\nu)|^2\prod_{\pm}|\zeta(1/2+i\nu\pm 2it)|\prod_{\pm, \pm}|\zeta(1/2+i\nu\pm i\tau\pm it)|.
\end{multline*}
Also, by using the same argument in the previous subsection, we have, when $|\tau|\ll t^{1-\varepsilon}$, the seven tail terms in Lemma \ref{lemma:explicit_EE^3} exponentially decays as $t\to \infty$. Therefore
in the case of $|\tau|\ll T^\varepsilon$ and $T\leq t\leq 2T$,  note that one can truncate $t_j$-sum and $\nu$-integral in the region $[t-T^\varepsilon, t+T^\varepsilon]$ which produce a negligible error term,  by combining the above estimate together, we have
\begin{lemma}\label{lemma:bound_EE^3}
Let $T\geq 1$, $T\leq t\leq 2T$ and $|\tau|\ll T^{\varepsilon}.$ Then for any $A>1$, we have
\begin{equation*}
\begin{split}
\langle  E_\tau, &E_t^3 \rangle_{\reg}
\ll T^{-\frac{3}{2}+\varepsilon}\sideset{}{^{\even}}\sum_{|t_j-t|\ll T^\varepsilon}
|L(1/2+2it, \phi_j)|L(1/2, \phi_j)\prod_{\pm}|L(1/2+i\tau\pm it, \phi_j)|
\\
&+T^{-\frac{3}{2}+\varepsilon}\int_{|\nu-t|\ll  T^\varepsilon}
|\zeta(1/2+i\nu)|^2\prod_{\pm}|\zeta(1/2+i\nu\pm 2it)|\prod_{\pm, \pm}|\zeta(1/2+i\nu\pm i\tau\pm it)|
\dd \nu+O(T^{-A}).
\end{split}
\end{equation*}
\end{lemma}

Now we give the proof of Lemma \ref{lemma:Eis_contri}.
\begin{proof}[Proof of Lemma \ref{lemma:Eis_contri} by using Theorem \ref{thm:8m-of-L}]
Using Lemma \ref{lemma:bound_EE^3}, we have
\begin{multline}\label{eqn:bound_EE^3}
\int_{T}^{2T}|\langle  E_\tau, E_t^3 \rangle_{\reg}|^2\dd t
\ll T^{-3+\varepsilon}\int_{T}^{2T}\Big(\sideset{}{^{\even}}\sum_{|t_j-t|\ll T^\varepsilon}
|L(1/2+2it, \phi_j)|L(1/2, \phi_j)\prod_{\pm}|L(1/2+i\tau\pm it, \phi_j)|\Big)^2\dd t\\
+T^{-3+\varepsilon}\int_{T}^{2T}\Big(\int_{|\nu-t|\ll  T^\varepsilon}
|\zeta(1/2+i\nu)|^2\prod_{\pm}|\zeta(1/2+i\nu\pm 2it)|\prod_{\pm, \pm}|\zeta(1/2+i\nu\pm i\tau\pm it)|\Big)^2\dd \nu.
\end{multline}
The first term, using Cauchy--Schwarz and Lemma \ref{lemma:2m_gl2}, is
\begin{equation*}
\begin{split}
 T^{-3+\varepsilon}&\int_{T}^{2T}\Big(\sideset{}{^{\even}}\sum_{|t_j-t|\ll T^\varepsilon}|L(1/2+2it, \phi_j)|L(1/2, \phi_j)\prod_{\pm}|L(1/2+i\tau\pm it, \phi_j)|\Big)^2\dd t\\
 &\ll
 T^{-3+\varepsilon}\int_{T}^{2T}\sideset{}{^{\even}}\sum_{|t_j-t|\ll T^\varepsilon}\Big(\prod_{\pm}|L(1/2+i\tau\pm it, \phi_j)|\Big)^2|L(1/2, \phi_j)|^2
 \sum_{|t_j-t|\ll T^\varepsilon}|L(1/2+2it, \phi_j)|^2\dd t\\
 &\ll T^{-2+\varepsilon}\int_{T}^{2T}\sideset{}{^{\even}}\sum_{|t_j-t|\ll T^\varepsilon}\Big(\prod_{\pm}|L(1/2+i\tau\pm it, \phi_j)|\Big)^2|L(1/2, \phi_j)|^2\dd t\\
 &\ll T^{-2+\varepsilon}\sideset{}{^{\even}}\sum_{T/2\leq t_j\leq 3T}|L(1/2, \phi_j)|^2\int_{|t_j-t|\ll T^\varepsilon}\Big(\prod_{\pm}|L(1/2+it\pm i\tau, \phi_j)|\Big)^2\dd t,
\end{split}
\end{equation*}
by exchanging the $t_j$-sum and $t$-integral. Then using H\"older's inequality, Theorem \ref{thm:4m-of-L} and Lemma \ref{lemma:4m_gl2}, it is bounded by
\begin{equation*}\label{eqn:Eisen_boundbyL}
\begin{split}
T^{-2+\varepsilon}&\int_{|\alpha|\ll T^{\varepsilon}}\quad\sideset{}{^{\even}}\sum_{T/2\leq t_j\leq 3T}\Big(\prod_{\pm}|L(1/2+it_j\pm i\tau+i\alpha, \phi_j)|\Big)^2
|L(1/2, \phi_j)|^2\dd \alpha\\
&\ll T^{-2+\varepsilon}\Big(\sideset{}{^{\even}}\sum_{T/2\leq t_j\leq 3T}
|L(1/2, \phi_j)|^4\Big)^{\frac{1}{2}}\int_{|\alpha|\ll T^{\varepsilon}}\Big(\sideset{}{^{\even}}\sum_{T/2\leq t_j\leq 3T}
|L(1/2+it_j+i\tau+i\alpha, \phi_j)|^{8}\Big)^{\frac{1}{4}}\\
&\hspace*{15em}
\cdot\Big(\sideset{}{^{\even}}\sum_{T/2\leq t_j\leq 3T}
|L(1/2+it_j-i\tau+i\alpha, \phi_j)|^{8}\Big)^{\frac{1}{4}}\dd \alpha\\
&\ll T^\varepsilon.
\end{split}
\end{equation*}
The last term in \eqref{eqn:bound_EE^3},  using the Weyl bound $\zeta(1/2+it)\ll (1+|t|)^{1/6+\varepsilon},$ and Cauchy--Schwarz, is bounded by
\begin{equation*}
\begin{split}
T^{-\frac{7}{3}+\varepsilon}&\int_{T}^{2T}\Big(\int_{|\nu-t|\ll  T^\varepsilon}
|\zeta(1/2+i\nu)|^2\prod_{\pm}|\zeta(1/2+i\nu\pm 2it)|\Big)^2\dd \nu\dd t\\
&\ll T^{-\frac{7}{3}+\varepsilon}\int_{T}^{2T}
\int_{|\nu-t|\ll  T^\varepsilon}
|\zeta(1/2+i\nu)|^4|\zeta(1/2+2it+i\nu)\zeta(1/2+2it+i\nu)|^2\dd \nu\dd t\\
&\ll T^{-\frac{7}{3}+\varepsilon}\int_{T/2}^{3T}
|\zeta(1/2+i\nu)|^4\int_{|\nu-t|\ll  T^\varepsilon}|\zeta(1/2+2it+i\nu)\zeta(1/2+2it+i\nu)|^2\dd t\dd \nu\\
&\ll T^{-\frac{5}{3}+\varepsilon}\int_{T/2}^{3T}
|\zeta(1/2+i\nu)|^4\dd \nu\ll T^{-\frac{2}{3}+\varepsilon}.
\end{split}
\end{equation*}
Combining these two estimates for terms in \eqref{eqn:bound_EE^3}, we complete the proof of Lemma \ref{lemma:Eis_contri}.
\end{proof}

So far, it is sufficient to establish the estimates in Theorem \ref{thm:4m-of-L} and Theorem \ref{thm:8m-of-L}. We shall complete these proofs in the remaining part of this paper.

\section{Preliminaries on bounding the moments of $L$-functions}\label{sec:preliminaries}
In this section, we do some preparatory job on proving Theorem \ref{thm:4m-of-L} and Theorem \ref{thm:8m-of-L}. 
From now on, we use $\phi$ to highlight the even Hecke--Maass cusp form $\phi_k$ with $t_k\ll T^{\varepsilon}$ (i.e. $\phi=\phi_k$).
We use the notation $\Phi=\phi\boxplus\phi$ and $\E=1\boxplus1\boxplus1\boxplus1$ which stand for the isobaric sums of two $\GL(2)$ and four $\GL(1)$ objects respectively.
Therefore $\Phi$ and $\E$ are in the special situation of degree $4$. 

\subsection{$L$-values around special points}
Firstly, we do some explicit computation on the Dirichlet coefficients of $L$-functions.
\begin{lemma}\label{lemma:L-value 2m4m}
Let $\Re(s)\gg 1$, we have
\begin{equation*}
  L(s, \phi\times \phi_j)^2=\sum_{m\geq 1}\sum_{n\geq 1}\frac{A_{\Phi}(n, m, 1)\lambda_j(n)}{m^{2s}n^s},
\end{equation*}
\begin{equation*}
  L(s,  \phi_j)^4=\sum_{m\geq 1}\sum_{n\geq 1}\frac{\tau(n, m, 1)\lambda_j(n)}{m^{2s}n^s},
\end{equation*}
where $\Phi=\phi\boxplus\phi$ is isobaric sum of two $\phi$, 
\[
\begin{split}
   A_\Phi(n, m, 1) & =\sum_{\ell d=m}\tau(\ell)\sum_{n_1n_2=n\atop n_1,n_2\geq 1}\lambda_\phi(dn_1)\lambda_\phi(dn_2), \\
   \tau(n, m, 1) & =\sum_{\ell d=m}\tau(\ell)\sum_{n_1n_2=n\atop n_1,n_2\geq 1}\tau(dn_1)\tau(dn_2),
\end{split}
\]
$\tau(\cdot)=d_2(\cdot)$ is the divisor function, both of them are real-valued. 
\end{lemma}
\begin{proof}
Recall that
\[
L(s, \phi\times \phi_j)=\sum_{\ell\geq 1}\sum_{n\geq 1}\frac{\lambda_{\phi}(n)\lambda_j(n)}{\ell^{2s}n^s},
\]
Thus
\begin{equation*}
L(s, \phi\times \phi_j)=\sum_{\ell\geq 1}\frac{\tau(\ell)}{\ell^{2s}}\sum_{n_1\geq1}\sum_{n_2\geq1}
\frac{\lambda_{\phi}(n_1)\lambda_{\phi}(n_2)\lambda_j(n_1)\lambda_j(n_2)}{n_1^sn_2^s}.
\end{equation*}
Using Hecke relation $\lambda_j(n_1)\lambda_j(n_2)=\sum_{d\mid (n_1, n_2)}\lambda_j(\frac{n_1n_2}{d^2})$, this becomes
\[
\sum_{\ell\geq 1}\frac{\tau(\ell)}{\ell^{2s}}\sum_{d\geq 1}\frac{1}{d^{2s}}\sum_{n\geq 1}\frac{\lambda_j(n)}{n^s}\sum_{n_1n_2=n}\lambda_\phi(dn_1)\lambda_\phi(dn_2)=
\sum_{m\geq 1}\sum_{n\geq 1}\frac{A_{\Phi}(n, m, 1)\lambda_j(n)}{m^{2s}n^s}.
\]
This is the desired Dirichlet series expression for $ L(s, \phi\times \phi_j)^2$.
Since the Hecke eigenvalues $\lambda_\phi(n)$ is real valued, $ A_\Phi(n, m, 1)$ is real-valued too.
The Dirichlet series for $L(s,  \phi_j)^4$ can be similarly obtained by using Hecke relation for $\tau$.
\end{proof}

\begin{lemma}\label{lemma:AFE_gl8}
Let $T\geq 1$ and $\phi_j$ be even with $T\leq t_j\leq 2T$,  $t_k, \tau\ll T^\varepsilon$ and real $\alpha$ with $|\alpha|\ll T^\varepsilon$, we have
\begin{equation*}
|L(1/2+it_j+i\alpha, \phi_k\times \phi_j)|^2
\ll T^\varepsilon\sum_{N:\text{dyadic}\leq T^{2+\varepsilon}\atop \pm}\int_{|v|\leq T^\varepsilon}\Big|\sum_{m\geq 1}\sum_{n\geq 1}\frac{A_{\Phi}(n, m, 1)\lambda_j(n)}{(m^2n)^{\frac{1}{2}\pm it_j}}W_1\Big(\frac{m^2n}{N}\Big)\Big|\dd v+O(T^{-2025}),
\end{equation*}
\begin{equation*}
|L(1/2+it_j+i\alpha, \phi_j)|^4
\ll T^\varepsilon\sum_{N:\text{dyadic}\leq T^{2+\varepsilon}\atop \pm}\int_{|v|\leq T^\varepsilon}\Big|\sum_{m\geq 1}\sum_{n\geq 1}\frac{\tau(n, m, 1)\lambda_j(n)}{(m^2n)^{\frac{1}{2}\pm it_j}}W_2\Big(\frac{m^2n}{N}\Big)\Big|\dd v+O(T^{-2025}),
\end{equation*}
where the functions $W_1, W_2$ depends on $\alpha, v, \varepsilon$, 
both of them are smooth and supported in $[1/2, 2]$, satisfy $W_1^{(j)}, W_2^{(j)}\ll_{j, \varepsilon} T^{\varepsilon}$ for any integer $j\geq 1.$
\end{lemma}
\begin{proof}
Recall that
\[
\Lambda(s, \phi\times \phi_j)^2:=\pi^{-4s}\prod_{\pm, \pm}\Gamma\Big(\frac{s\pm it_k\pm it_j}{2}\Big)L(s, \phi\times \phi_j)^2
=\Lambda(1-s, \phi\times \phi_j)^2,
\]
\[
\Lambda(s, \phi_j)^4:=\pi^{-4s}\prod_{\pm}\left(\Gamma\Big(\frac{s\pm it_j}{2}\Big)\right)^2L(s, \phi_j)^4
=\Lambda(1-s, \phi_j)^4.
\]
Here we give the proof of the first approximate functional equation. The proof of the second one is similar.
For $\Re(s)\gg 1$, we write $L(s, \phi\times \phi_j)^2$ as  $\sum_{k\geq 1}a_{\phi, \phi_j}(k)k^{-s}$.
By Lemma \ref{lemma:L-value 2m4m}, we have
\[
a_{\phi, \phi_j}(k)=\sum_{m^2n=k}A_{\Phi}(n, m, 1)\lambda_j(n).
\]
Combining the above functional equation with the method of \cite[Theorem 5.3]{I-K04}, we have
\begin{multline*}
L(1/2+it_j+i\alpha, \phi\times \phi_j)^2
=\int_{(3)}\pi^{-4u}\sum_{N:\text{dyadic}}\sum_{k\geq 1}\frac{a_{\phi, \phi_j}(k)}{k^{\frac{1}{2}+it_j+i\alpha+u}}W\Big(\frac{k}{N}\Big)
\frac{\prod_{\pm, \pm}\Gamma\Big(\frac{\frac{1}{2}+it_j+i\alpha+u\pm it_k\pm it_j}{2}\Big)}
{\prod_{\pm, \pm}\Gamma\Big(\frac{\frac{1}{2}+it_j+i\alpha\pm it_k\pm it_j}{2}\Big)}\frac{e^{u^2}}{u}\dd u\\
+\int_{(3)}\pi^{-4u}\sum_{N:\text{dyadic}}\sum_{k\geq 1}\frac{a_{\phi, \phi_j}(k)}{k^{\frac{1}{2}-it_j-i\alpha+u}}W\Big(\frac{k}{N}\Big)
\frac{\prod_{\pm, \pm}\Gamma\Big(\frac{\frac{1}{2}-it_j-i\alpha+u\pm it_k\pm it_j}{2}\Big)}
{\prod_{\pm, \pm}\Gamma\Big(\frac{\frac{1}{2}+it_j+i\alpha\pm it_k\pm it_j}{2}\Big)}\frac{e^{u^2}}{u}\dd u,
\end{multline*}
where $W$ is a smooth compactly supported function satisfying $\supp(W)\subset [1/2, 2]$, $W^{(j)}\ll_j 1$ and we use the smooth dyadic partition $\sum_{N:\text{dyadic}}W(\frac{x}{N})=1$ for any $x\geq1.$ 

By shifting the integral line far to the right,  using Stirling 's formula, the contribution of $N>T^{2+\varepsilon}$ is small, say $O(T^{-2025})$. 
When $k\leq T^{2+\varepsilon}$, we shift the integral line to $\Re(u)=\varepsilon$. Due to the rapidly decay of $e^{u^2}$ as $\Im(u)$ large, we can truncate the $u$-integral in the region $|\Im(u)|\leq T^{\varepsilon}$.
Then exchanging the order of summations and integral, using Stirling 's formula to bound the $\Gamma$-factors by $T^\varepsilon$, we get
\begin{multline*}
|L(1/2+it_j+i\alpha, \phi\times \phi_j)|^2
\leq T^\varepsilon\sum_{N:\text{dyadic}\leq T^{2+\varepsilon}}\int_{|v|\leq T^\varepsilon}\Big|\sum_{k\geq 1}\frac{a_{\phi, \phi_j}(k)}{k^{\frac{1}{2}+it_j+i\alpha+\varepsilon+iv}}W\Big(\frac{k}{N}\Big)\Big|\dd v\\
+T^\varepsilon\sum_{N:\text{dyadic}\leq T^{2+\varepsilon}}\int_{|v|\leq T^\varepsilon}\Big|\sum_{k\geq 1}\frac{a_{\phi, \phi_j}(k)}{k^{\frac{1}{2}-it_j-i\alpha+\varepsilon+iv}}W\Big(\frac{k}{N}\Big)\Big|\dd v+O(T^{-2025}).
\end{multline*} 
Let $W_1(\frac{k}{N})=W(\frac{k}{N})(\frac{k}{N})^{-i\alpha-iv-\varepsilon}$ or $W(\frac{k}{N})(\frac{k}{N})^{i\alpha-iv-\varepsilon}$, we complete the proof. 
\end{proof}
\subsection{Balanced Voronoi summation formula}
We need a balanced Voronoi summation formula for $\Phi=\phi\boxplus\phi$ and $\E=1\boxplus1\boxplus1\boxplus1$.
These can be viewed as the analogies of $\GL(4)$ case. 

Let $q_1,q_2, r\in \mathbb{Z}_{\geq 1}$,  $a, n\in \mathbb{Z}$, assume that $d_1\mid q_1r$ and $d_2\mid \frac{q_1q_2r}{d_1}$. The $\GL(4)$ associated hyper-Kloosterman sum is defined to be 
\[
\Kl(a, n, r; q_1, q_2,d_1, d_2)=\sideset{}{^*}\sum_{x_1\mod \frac{q_1r}{d_1}}\sideset{}{^*}\sum_{x_2\mod \frac{q_1q_2r}{d_1d_2}}
e\Big(\frac{d_1x_1a}{r}+\frac{d_2x_2\overline{x_1}}{\frac{q_1r}{d_1}}+\frac{n\overline{x_2}}{\frac{q_1q_2r}{d_1d_2}}\Big).
\] 
We introduce the following useful result of K{\i}ral and Zhou \cite{K-Z16}.
\begin{theorem}{(\,\cite[Theorem 1.3 on $N=4$]{K-Z16}\,)}\label{theorem:KZ's Voronoi}
Let $F$ be a symbol and assume that $F$ come numbers $A(m_1,m_2,m_3)\in \mathbb{C}$ with natural numbers $m_1, m_2, m_3$ and $A(1,1,1)=1$. Assume that these coefficients $A(\cdot,\cdot,\cdot)$ satisfy the following Hecke relations:
\begin{equation}\label{eqn:A_mul}
A(m_1m_1',m_2m_2',m_3m_3')=A(m_1,m_2,m_3)A(m_1',m_2',m_3')\text{ for } (m_1m_2m_3, m_1'm_2'm_3')=1,
\end{equation}
\begin{equation}\label{eqn:A_hecke1}
A(n, 1, 1)A(m_1,m_2,m_3)=\sum_{d_0d_1d_2d_3=n\atop d_1\mid m_1, d_2\mid m_2, d_3\mid m_3}
A\Big(\frac{m_1d_0}{d_1}, \frac{m_2d_1}{d_2}, \frac{m_3d_2}{d_3}\Big),
\end{equation}
and
\begin{equation}\label{eqn:A_hecke2}
A(1, 1, n)A(m_3,m_2,m_1)=\sum_{d_0d_1d_2d_3=n\atop d_1\mid m_1, d_2\mid m_2, d_3\mid m_3}
A\Big(\frac{m_3d_2}{d_3}, \frac{m_2d_1}{d_2}, \frac{m_1d_0}{d_1}\Big).
\end{equation}
Further assume that they grow moderately as
\begin{equation}\label{eqn:A_growth}
A(m_1,m_2,m_3)\ll (m_1m_2m_3)^{\sigma} \text{ for some } \sigma>0.
\end{equation}
Let $\widetilde{F}$ be another symbol whose associated coefficients $B(\cdot, \cdot, \cdot)\in \mathbb{C}$ and $B(m_1, m_2, m_3)=A(m_3, m_2, m_1)$ and assume that they also satisfy the same properties. Further assume that there are two meromorphic functions $G_{+}(s)$ and $G_{-}(s)$ associated to the pair $(F, \widetilde{F})$, so that for a given primitive character $\chi^*$ modulo $q$, the $L$-function
\[
L(s, F\times\chi^*):=\sum_{n\geq 1}\frac{A(n, 1, 1)\chi^*(n)}{n^s} \text{ for } \Re(s)>\sigma+1
\]
admits the analytic continuation to the whole complex plane and satisfies the functional equation
\begin{equation}\label{eqn:F's FE}
  L(s, F\times \chi^*)=\tau(\chi^*)^4q^{-4s}G(s)L(1-s, \widetilde{F}\times \overline{\chi^*}),
\end{equation}
where $\tau(\chi^*)=\sum_{a\mod q}\chi^*(a)e(\frac{a}{q})$ is the Gauss sum, 
$G(s)=\begin{cases}
        G_{+}(s), & \mbox{if } \chi^*(-1)=1,\\
        G_{-}(s), & \mbox{if } \chi^*(-1)=-1.
      \end{cases}$
Under these assumptions, let $q_1,q_2, c\in \mathbb{Z}_{\geq 1}$,  $a\in \mathbb{Z}$ with $(a, c)=1$, $\overline{a}$ be the multiplicative inverse of $a$ modulo $c$.
For $\Re(s)>\sigma+1$, define 
\[
L(s, F,\frac{a}{c}; q_1, q_2):=\sum_{n\geq 1}\frac{A(n, q_1, q_2)}{n^s}e\Big(\frac{\overline{a}n}{c}\Big).
\]
Then $L(s, F,\frac{a}{c}; q_1, q_2)$ has analytic continuation to all $s\in \mathbb{C}$ and satisfies the Voronoi formula
\begin{equation}\label{eqn:Voronoi_F}
\begin{split}
L(s, F,\frac{a}{c}&; q_1, q_2)
   =\frac{G_{+}(s)-G_{-}(s)}{2}\sum_{d_1\mid q_1c}\sum_{d_2\mid\frac{q_1q_2c}{d_1}}
   \sum_{n\geq 1}\frac{A(d_1,d_2, n)\Kl(a, n, c; q_1, q_2, d_1, d_2)}{n^{1-s}c^{4s-1}d_1d_2}\frac{d_1^{3s}d_2^s}{q_1^{2s}q_2^s}\\
     &+\frac{G_{+}(s)+G_{-}(s)}{2}\sum_{d_1\mid q_1c}\sum_{d_2\mid\frac{q_1q_2c}{d_1}}
   \sum_{n\geq 1}\frac{A(d_1,d_2, n)\Kl(a, -n, c; q_1, q_2, d_1, d_2)}{n^{1-s}c^{4s-1}d_1d_2}\frac{d_1^{3s}d_2^s}{q_1^{2s}q_2^s},\\
\end{split}
\end{equation}
which the right side is absolutely convergent for $\Re(s)<-\sigma.$
\end{theorem}
\begin{remark}\label{Remark:Voronoi_Residue}
Let $B_{\widetilde{F}}(n, 1, 1)=A_F(1, 1, n)=A_F(n, 1, 1)$ for all $n\geq 1$,
we can extend the definition of $A_{\Phi}$ and $\tau=A_{\E}$ by Hecke relations \eqref{eqn:A_mul}, \eqref{eqn:A_hecke1} and \eqref{eqn:A_hecke2}. Actually, by computing about prime powers for $A_F$ with $F=\Phi$ or $\E$, we have, similarly as \cite[Lemma 3.4]{C-L20}
\begin{equation}\label{eqn:defA_Fgeneral}
A_F(n, \ell, k)=\mathop{\sum\sum\sum}_{\substack{d, e, f\\d\mid(k, \ell), e\mid(d,k/d)\\f\mid(k,n)}}\mu(df)\mu(e)A_F\Big(\frac{k}{dfe},1, 1\Big)
A_F\Big(\frac{dn}{ef}, \frac{\ell}{d}, 1\Big)
\end{equation}
for all integers $n, \ell, k\geq 1.$ Moreover $B_{\widetilde{F}}(n, \ell, k)=A_F(n, \ell, k)=A_F(k, \ell, n)$.
Then we apply Theorem \ref{theorem:KZ's Voronoi} for $F=\Phi$ and $F=\E$. Here the case of $F=\E$ is a little bit different since that for the principal Dirichlet character $\chi_0$, $L(s, \E\times \chi_0)$ is not analytic on $\mathbb{C}$ and $L(s, \E,\frac{a}{c}; q_1, q_2)$ only has the meromorphic continuation to whole complex plane $\mathbb{C}$. 
For example, when $q_1=m, q_2=1$, by a similar relation of \cite[Eqn (3.1)]{C-L20}, we have
\[
\tau(n, m, 1)=\tau(1, m, n)=\sum_{d\mid (n,m)\atop e\mid(d, n/d)}\mu(d)\mu(e)\tau(\frac{n}{de}, 1, 1)\tau(\frac{d}{e}, \frac{m}{d}, 1).
\]
Thus for $\Re(s)>1$,
\begin{equation*}
\begin{split}
L(s, \E, \frac{a}{c}; m, 1)&=\sum_{n\geq 1}\frac{\tau(n, m, 1)}{n^s}e\Big(\frac{\overline{a}n}{c}\Big)\\
&=\sum_{d\mid m}\frac{\mu(d)}{d^{s}}\sum_{e\mid d}\frac{\mu(e)}{e^s}\tau(1, \frac{m}{d}, \frac{d}{e})\sum_{n\geq 1}\frac{\tau(n, 1, 1)}{n^s}e\Big(\frac{\overline{a}den}{c}\Big)
\end{split}
\end{equation*}
Let $c'=\frac{c}{de}$ and $a'=\overline{a}\frac{de}{(de, c)}$, then $L(s, \E, \frac{\overline{a'}}{c'}; 1, 1)=\sum_{n\geq 1}\frac{\tau(n, 1, 1)}{n^s}e(\frac{\overline{a}den}{c})$.
Note that
\[
L(s, \E, \frac{\overline{a'}}{c'}; 1, 1)=\sum_{a_1, a_2, a_3, a_4=1}^{c'}e\Big(\frac{\overline{a'}a_1a_2a_3a_4}{c'}\Big)\zeta(s, a_1, c')\zeta(s, a_2, c')\zeta(s, a_3, c')\zeta(s, a_4, c'),
\]
where $\zeta(s, a_i, c')=\sum_{n\equiv a_i\mod c'\atop n\geq 1}\frac{1}{n^s}, \Re(s)>1$ is the Hurwitz-zeta function which can be analytically continued to the
whole complex plane to a meromorphic function with a single pole, at $s=1$. 
Moreover, from the computation by Conrey and Gonek in \cite[Page 589--591]{C-G01}, we see that  $L(s, \E, \frac{\overline{a'}}{c'}; 1, 1)$ has a meromorphic continuation to the whole complex plane and it has a unique 
pole of degree $4$ at $s=1$.
Its singular part is 
\[
\Big(\frac{c}{de}\Big)^{-s}\zeta(s)^4G_4(s, \frac{c}{de}),
\]
where 
\[
G_4(s, \frac{c}{de})=\sum_{\ell \mid \frac{c}{de}}\frac{\mu(\ell)}{\varphi(\ell)}\ell^s\sum_{\kappa\mid \ell }\frac{\mu(\kappa)}{\kappa^s}\prod_{p^\alpha||\frac{c\kappa}{de\ell}}
\Big((1-\frac{1}{p^s})^{4}\sum_{j\geq 0}\frac{d_4(p^{\alpha+j})}{p^{js}}\Big),
\]
which is independent of $a'$. Therefore, $L(s, \E, \frac{a}{c}; m, 1)$ has a meromorphic continuation to the whole complex plane with a unique pole of degree $4$ at $s=1$. Its singular part is, for  $\Re(s)> 0$,
\begin{equation}\label{eqn:singular}
\mathcal{L}(s; m, c)=\frac{\zeta(s)^4}{c^s}\sum_{d\mid m}\mu(d)\sum_{e\mid d}\mu(e)\tau(1, \frac{m}{d}, \frac{d}{e})\sum_{\ell \mid \frac{c}{de}}\frac{\mu(\ell)}{\varphi(\ell)}\ell^s\sum_{\kappa\mid \ell }\frac{\mu(\kappa)}{\kappa^s}\prod_{p^\alpha||\frac{c\kappa}{de\ell}}
\Big((1-\frac{1}{p^s})^{4}\sum_{j\geq 0}\frac{d_4(p^{\alpha+j})}{p^{js}}\Big)
\end{equation}
which is independent of $a$.
By using a simple refinement in the proof of K{\i}ral and Zhou \cite{K-Z16}, $L(s, \E, \frac{a}{c}; m, 1)$ also satisfies the Voronoi formula \eqref{eqn:Voronoi_F} in Theorem \ref{theorem:KZ's Voronoi}. 
\end{remark}

Now we deduce a general Voronoi formula of degree $4$. Let $W$ be a smooth compactly supported function on $\mathbb{R}_{>0}$. With the notation as above and $F=\Phi$ or $\E$, applying the Mellin inversion, we have
\[
\sum_{n\neq 0}A(n, m, 1)e\Big(\frac{\overline{a}n}{c}\Big)W(n)=\frac{1}{2\pi i}\int_{(\sigma+2)}\widetilde{W}(s)L(s, F,\frac{a}{c}; m, 1)\dd s,
\]
where $\sigma\gg 1$.
Since that $\widetilde{W}$ is analytic and $L(s, F,\frac{a}{c}; m, 1)$ is analytic except a possible pole at $s=1$, we shift the contour to $\Re(s)=-\sigma-1$. Using \eqref{eqn:Voronoi_F}, we have
\begin{equation*}
\begin{split}
\sum_{n\neq 0}A(n, m, 1)e\Big(\frac{\overline{a}n}{c}\Big)&W(n)=\Res_{s=1}\widetilde{W}(s)L(s, F,\frac{a}{c}; m, 1)\\
&+\frac{c}{2}\sum_{d_1\mid mc}\sum_{d_2\mid\frac{mc}{d_1}}\sum_{n\geq 1}\frac{A(d_1,d_2, n)\Kl(a, n, c; m, 1, d_1, d_2)}{nd_1d_2}\W_{+, F}\Big(\frac{nd_2^2d_1^3}{c^4m^2}\Big)\\
&-\frac{c}{2}\sum_{d_1\mid mc}\sum_{d_2\mid\frac{mc}{d_1}}\sum_{n\geq 1}\frac{A(d_1,d_2, n)\Kl(a, n, c; m, 1, d_1, d_2)}{nd_1d_2}\W_{-, F}\Big(\frac{nd_2^2d_1^3}{c^4m^2}\Big)\\
&+\frac{c}{2}\sum_{d_1\mid mc}\sum_{d_2\mid\frac{mc}{d_1}}\sum_{n\geq 1}\frac{A(d_1,d_2, n)\Kl(a, -n, c; m, 1, d_1, d_2)}{nd_1d_2}\W_{+, F}\Big(\frac{nd_2^2d_1^3}{c^4m^2}\Big)\\
&+\frac{c}{2}\sum_{d_1\mid mc}\sum_{d_2\mid\frac{mc}{d_1}}\sum_{n\geq 1}\frac{A(d_1,d_2, n)\Kl(a, -n, c; m, 1, d_1, d_2)}{nd_1d_2}\W_{-, F}\Big(\frac{nd_2^2d_1^3}{c^4m^2}\Big),
\end{split}
\end{equation*}
where 
\begin{equation}\label{eqn:Voronoi_transf}
\W_{\pm, F}(x)=\frac{1}{2\pi i}\int_{(-\sigma-1)}\widetilde{W}(s)G_{\pm}(s)x^s\dd s.
\end{equation}
When $x>0$, we can write 
\begin{equation}\label{eqn:W_F}
\W_{F}(x)=\W_{+, F}(x)-\W_{-, F}(x), \quad 
\W_{F}(-x)=\W_{+, F}(x)+\W_{-, F}(x).
\end{equation}
Then we get the following Voronoi formula from Theorem \ref{theorem:KZ's Voronoi}:
\begin{multline}\label{eqn:Voronoi_Fagain}
\sum_{n\neq 0}A(n, m, 1)e\Big(\frac{\overline{a}n}{c}\Big)W(n)
=\Res_{s=1}\widetilde{W}(s)L(s, F,\frac{a}{c}; m, 1)\\+\frac{c}{2}\sum_{d_1\mid mc}\sum_{d_2\mid\frac{mc}{d_1}}\sum_{n\neq 0}\frac{A(d_1,d_2, n)\Kl(a, n, c; q_1, q_2, d_1, d_2)}{|n|d_1d_2}\W_{F}\Big(\frac{nd_2^2d_1^3}{c^4m^2}\Big)
\end{multline}

Let $\chi^*$ be a primitive character modulo $q$, for $\Re(s)>1$, we have the degree $4$ twisted $L$-functions
\[
L(s, \Phi\times\chi^* )=L(s, \phi\times\chi^*)^2=\sum_{n\geq 1}\frac{A_{\Phi}(n, 1, 1)\chi^*(n)}{n^s}
\]
\[
L(s, \E\times\chi^* )=L(s, \chi^*)^4=\sum_{n\geq 1}\frac{\tau(n, 1, 1)\chi^*(n)}{n^s}.
\]
They satisfy the functional equations 
\begin{equation*}
\begin{split}
\Lambda(s, \Phi\times\chi^* ):&=\Big(\frac{q}{\pi}\Big)^{2s}\prod_{\pm}\Gamma\Big(\frac{s+\frac{1-\chi^*(-1)}{2}\pm it_\phi}{2}\Big)^2
L(s, \Phi\times\chi^*)\\
&=\frac{\tau(\chi^*)^4}{q^2}\Lambda(1-s, \Phi\times\overline{\chi^*}),\\
\Lambda(s, \E\times\chi^* ):&=\Big(\frac{q}{\pi}\Big)^{2s}\Gamma\Big(\frac{s+\frac{1-\chi^*(-1)}{2}}{2}\Big)^4
L(s, \E\times\chi^*)\\
&=\frac{\tau(\chi^*)^4}{q^2}\Lambda(1-s, \E\times\overline{\chi^*}).
\end{split}
\end{equation*}
where  $\tau(\chi^*)$ is the Gauss sum of $\chi^*$. These are 
\[
L(s, \Phi\times \chi^*)=\Big(\frac{\tau(\chi^*)^4}{q^{4s}}\Big)G_{\pm, \Phi}L(1-s, \Phi\times \overline{\chi^*}),
\]
\[
L(s, \E\times \chi^*)=\Big(\frac{\tau(\chi^*)^4}{q^{4s}}\Big)G_{\pm, \E}L(1-s, \E\times \overline{\chi^*})
\]
corresponding \eqref{eqn:F's FE}, where
\begin{equation}\label{eqn:Gamma_factors}
\begin{split}
G_{+, \Phi}&=\pi^{4s-2}\left(\frac{\prod_{\pm}\Gamma(\frac{1-s\pm it_\phi}{2})}
{\prod_{\pm}\Gamma(\frac{s\pm it_\phi}{2})}\right)^2,\quad 
G_{-, \Phi}=\pi^{4s-2}\left(\frac{\prod_{\pm}\Gamma(\frac{2-s\pm it_\phi}{2})}
{\prod_{\pm}\Gamma(\frac{1+s\pm it_\phi}{2})}\right)^2,\\
G_{+, \E}&=\pi^{4s-2}\left(\frac{\Gamma(\frac{1-s}{2})}
{\Gamma(\frac{s}{2})}\right)^4,\hspace{5em}
G_{-, \E}=\pi^{4s-2}\left(\frac{\Gamma(\frac{2-s}{2})}
{\Gamma(\frac{1+s}{2})}\right)^4.
\end{split}
\end{equation} 
Note that inserting the above $G_{\pm}$ in \eqref{eqn:W_+Fdef}, we can chose the integral line at any $\Re(s)<1$. 
Applying Theorem \ref{theorem:KZ's Voronoi} and the corresponding formula \eqref{eqn:Voronoi_Fagain} on $F=\Phi=\phi\boxplus\phi$ and $F=\E=1\boxplus1\boxplus1\boxplus1$ respectively. We conclude the following result:
\begin{lemma}\label{lemma:Voronoi}
Let $q_1,q_2, c\in \mathbb{Z}_{\geq 1}$,  $a\in \mathbb{Z}$ with $(a, c)=1$, $\overline{a}$ be the multiplicative inverse of $a$ modulo $c$, $W$ be a smooth compactly supported function on $\mathbb{R}_{>0}$. Then we have
\begin{equation}\label{eqn:Voronoi_2+2}
\sum_{n\neq 0}A_{\Phi}(n, m, 1)e\Big(\frac{\overline{a}n}{c}\Big)W(n)
=\frac{c}{2}\sum_{d_1\mid mc}\sum_{d_2\mid\frac{mc}{d_1}}\sum_{n\neq 0}\frac{A_{\Phi}(d_1,d_2, n)\Kl(a, n, c; m, 1, d_1, d_2)}{|n|d_1d_2}\W_{\Phi}\Big(\frac{nd_2^2d_1^3}{c^4m^2}\Big)
\end{equation}
and 
\begin{multline}\label{eqn:Voronoi_1+1+1+1}
\sum_{n\neq 0}\tau(n, m, 1)e\Big(\frac{\overline{a}n}{c}\Big)W(n)
=\Res_{s=1}\widetilde{W}(s)\mathcal{L}(s, m, c)\\+
\frac{c}{2}\sum_{d_1\mid mc}\sum_{d_2\mid\frac{mc}{d_1}}\sum_{n\neq 0}\frac{\tau(d_1,d_2, n)\Kl(a, n, c; m, 1, d_1, d_2)}{|n|d_1d_2}\W_{\E}\Big(\frac{nd_2^2d_1^3}{c^4m^2}\Big),
\end{multline}
where $\W_{\Phi}$ and $\W_{\E}$ are defined by \eqref{eqn:Voronoi_transf} and \eqref{eqn:W_F} with different gamma factors in \eqref{eqn:Gamma_factors}, 
$\mathcal{L}(s; m, c)$ is \eqref{eqn:singular}.
\end{lemma}
\subsection{Ramanujan bound on average}
We introduce some Ramanujan-on-average bounds for the coefficients of $F$. These are analogues of $\GL(4)$ case which was showed in \cite[\S 3]{C-L20}. But here is much simple.
\begin{lemma}\label{lemma:RamanujanOnAverage}
  Let $X, Y\geq 1$, then we have
  \[
  \sum_{n\sim X}\sum_{m\sim Y}|\tau(n, m, 1)|^2\ll (XY)^{1+\varepsilon},
  \]
  \[
  \sum_{n\sim X}\sum_{m\sim Y}|A_\Phi(n, m, 1)|^2\ll (XY)^{1+\varepsilon}T^\varepsilon.
  \]
\end{lemma}
\begin{proof}
  The averaged bound for $\tau$ is obtained by using the divisor bound $d_2(n)\ll n^\varepsilon$ trivially.
  By definition of $A_\Phi(n, m, 1)$, using Cauchy--Schwarz and elementary inequality, we have
  \begin{equation*}
  \begin{split}
  \sum_{n\sim X}\sum_{m\sim Y}&|A_\Phi(n, m, 1)|^2
  \ll \sum_{n\sim X}\sum_{m\sim Y}\sum_{\ell d=m}|\tau(\ell)|^2\sum_{\ell d=m}\Big|\sum_{n_1n_2=n}|\lambda_\phi(dn_1)|^2\Big|^2\\
  &\ll (XY)^\varepsilon\sum_{n_1n_2\sim X}\sum_{\ell d\sim Y}|\lambda_\phi(dn_1)|^4
  \ll (XY)^\varepsilon\sum_{n_2\ll X}\sum_{\ell \ll Y}\sum_{k\sim \frac{XY}{n_2\ell}}\lambda_\phi(k)^4.
  \end{split}
  \end{equation*}
  By using fourth moment uniform bound on $\GL(2)$ coefficients \cite[Lemma 3.6]{GRS13}:
  $$\sum_{k\sim K}\lambda_\phi(k)^4
  \ll K^{1+\varepsilon}t_\phi^\varepsilon,$$
  we get the desired bound.
\end{proof}

\begin{lemma}\label{lemma:RamanujanOnAverage2}
  For any $X\geq 1$ and positive integers $b, c$, we have
  \[
  \sum_{n\leq X}|\tau(c, b, n)|^2\ll X(Xbc)^{\varepsilon},
  \]
  \[
  \sum_{n\leq X}|A_\Phi(c, b, n)|^2\ll (Xbc)^{1+\varepsilon}T^\varepsilon.
  \]
\end{lemma}
\begin{proof}
The averaged bound for $\tau$ is obtained by using the divisor bound $d_2(n)\ll n^\varepsilon$ trivially.
The second bound can be obtained by a similar argument in \cite[Lemma 3.5]{C-L20}.
\end{proof}

\subsection{Spectral mean value theorem}
In order to bounding the $L$-function at $1/2+it_j$, we require two large sieve type results for the kernels $\lambda_j(n)n^{it_j}$.
The first theorem is due to Luo \cite[Theorem 1]{Luo95}.
\begin{theorem}\label{thm:Luo'sLS}
For any complex numbers $a_n$, we have
\[
\sum_{t_j\leq T}\frac{|\rho_j(1)|^2}{\cosh (\pi t_j)}\Big|\sum_{n\leq N}a_n\lambda_j(n)n^{it_j}\Big|^2
\ll(T^2+T^{\frac{3}{2}}N^{\frac{1}{2}}+N^{\frac{5}{4}})(NT)^\varepsilon\sum_{n\leq N}|a_n|^2.
\]
The implied constant depends on $\varepsilon$ only.
\end{theorem}
The second theorem is in \cite[Theorem 7.1]{You13} and \cite[Theorem 4.1]{C-L20}.
\begin{theorem}\label{thm:Young'sLS}
  Let $T\geq 2$ and the non-negative smooth function $w$ is defined by 
  \[
  w(t)=2\frac{\sinh((\pi-\frac{1}{T})t)}{\sinh(2\pi t)}.
  \] 
  Let 
  \[
     S(\mathcal{A}):=\sum_{t_j}w(t_j)|\rho_j(1)|^2\Big|\sum_{n\sim N}a_n\lambda_j(n)n^{it_j}\Big|^2,
  \]
  then for any $1\leq X\leq T$ and $N\gg T$, we have
  \[
  S(\mathcal{A})=S_1(\mathcal{A}; X)+O\Big(T^2+\frac{NT}{X}+\frac{N^{\frac{3}{2}}}{T}\Big)N^\varepsilon\Vert\mathcal{A}\Vert^2,
  \]
  where $\Vert\mathcal{A}\Vert^2=\sum_{n\sim N}|a_n|^2,$ and 
  \[
  S_1(\mathcal{A}; X)\ll T\sum_{r<X}\frac{1}{r^2}\sum_{0\neq |k|\ll rT^\varepsilon}\int_{-T^\varepsilon}^{T^\varepsilon}
  \min\Big\{\frac{1}{|u|}, \frac{r/|k|}{1+u^2}\Big\}\Big|\sum_{n\sim N}a_nS(k, n; r)e\Big(\frac{un}{rT}\Big)\Big|^2\dd u.
  \]
  Here $S(k, n; r)$ is the usual Kloosterman sum.
\end{theorem}
Note that by choosing weight function $w$ as above, we have
\[
w(t_j)|\rho_j(1)|^2\sim \frac{|\rho_j(1)|^2}{\cosh(\pi t_j)}\exp\Big(-\frac{t_j}{T}\Big)=T^{o(1)}
\]
for $t_j\sim T.$ This is convenient for applying Theorem \ref{thm:Young'sLS} directly.

\subsection{Stirling's formula}


For fixed $\sigma\in\mathbb{R}$, real $|t|\geq10$ and any $J>0$, we have Stirling's formula
\begin{equation}\label{eqn:Stirling_J}
  \Gamma(\sigma+it) = e^{-\frac{\pi}{2}|t|} |t|^{\sigma-\frac{1}{2}} \exp\left( it\log\frac{|t|}{e} \right) \left( g_{\sigma,J}(t) + O_{\sigma,J}(|t|^{-J}) \right),
\end{equation}
where
\[
  t^j \frac{\partial^j}{\partial t^j} g_{\sigma,J}(t) \ll_{j,\sigma,J} 1
\]
for all fixed $j\in \mathbb{N}_0$. 
Combining the functional equation $\Gamma(s)\Gamma(1-s)=\frac{\pi}{\sin(\pi s)}$ with above approximation \eqref{eqn:Stirling_J}, we have
\begin{equation}\label{eqn:Stirling_inverse}
  \frac{1}{\Gamma(\sigma+it)} = (e^{\frac{\pi}{2}|t|}- e^{-\frac{\pi}{2}|t|})|t|^{\frac{1}{2}-\sigma} \exp\left(-it\log\frac{|t|}{e} \right) \left( \tilde{g}_{\sigma,J}(t) + O_{\sigma,J}(|t|^{-J}) \right).
\end{equation}
And $\tilde{g}$ satisfy
\[
  t^j \frac{\partial^j}{\partial t^j} \tilde{g}_{\sigma,J}(t) \ll_{j,\sigma,J} 1
\]
for all fixed $j\in \mathbb{N}_0$.

More precisely, we have
\begin{equation}\label{eqn:Stirling_7}
  \log  \Gamma(z) = z\log z + \frac{1}{2} \log \frac{2\pi}{z} + \frac{1}{12 z} -\frac{1}{360z^3} + \frac{1}{1260z^5} + O(|z|^{-7}).
\end{equation}

\subsection{Oscillatory integrals}
Let $\mathcal{F}$ be an index set and $X=X_T:\mathcal{F}\rightarrow \mathbb{R}_{\geq1}$ be a function of $T\in\mathcal{F}$. A family of $\{w_T\}_{T\in\mathcal{F}}$ of smooth functions supported on a product of dyadic intervals in $\mathbb{R}_{>0}^d$ is called \emph{$X$-inert} if for each $j=(j_1,\ldots,j_d) \in \mathbb{Z}_{\geq0}^d$ we have
\[
  \sup_{T\in\mathcal{F}} \sup_{(x_1,\ldots,x_d) \in \mathbb{R}_{>0}^d}
  X_T^{-j_1-\cdots -j_d} \left| x_1^{j_1} \cdots x_d^{j_d} w_T^{(j_1,\ldots,j_d)} (x_1,\ldots,x_d) \right|
   \ll_{j_1,\ldots,j_d} 1.
\]


We will use the following integration by parts and stationary phase lemmas several times.

\begin{lemma}\label{lemma:repeated_integration_by_parts}
  Let $Y\geq1$. Let $X,\; V,\; R,\; Q>0$ and suppose that $w=w_T$ is a smooth function with  $\supp w \subseteq [\alpha,\beta]$ satisfying $w^{(j)}(\xi) \ll_j X V^{-j}$ for all $j\geq0$.
  Suppose that on the support of $w$, $h=h_T$ is smooth and satisfies that
  $h'(\xi)\gg R$ and $ h^{(j)}(\xi) \ll Y Q^{-j}$, for all $j\geq2.$
  Then for arbitrarily large $A$ we have
    \[
      I = \int_{\mathbb{R}} w(\xi) e(h(\xi))  \dd \xi  \ll_A (\beta-\alpha)  X \left[  \left(\frac{QR}{\sqrt{Y}}\right)^{-A} + (RV)^{-A}  \right].
    \]
\end{lemma}

\begin{proof}
  See \cite[Lemma 8.1]{BKY13}.
\end{proof}

\begin{lemma}\label{lemma:stationary_phase}
  Suppose $w_T$ is $X$-inert in $t_1,\ldots,t_d$, supported on  $t_i\asymp X_i$ for $i=1,2,\ldots,d$. Suppose that on the support of $w_T$, $h=h_T$ satisfies that
  \[
    \frac{\partial^{a_1+a_2+\cdots +a_d}}{\partial t_1^{a_1}\cdots \partial t_d^{a_d}} h(t_1,t_2,\ldots,t_d) \ll_{a_1,\ldots,a_d}  \frac{Y}{X_1^{a_1} X_2^{a_2}\cdots X_d^{a_d}},
  \]
  for all $a_1,\ldots,a_d\in \mathbb{Z}_{\geq0}$. Let
  \[
    I = \int_{\mathbb{R}} w_T(t_1,t_2,\ldots,t_d) e^{i h(t_1,t_2,\ldots,t_d)}  \dd t_1.
  \]
   Suppose $\frac{\partial^{2}}{\partial t_1^{2}} h(t_1,t_2,\ldots,t_d) \gg \frac{Y}{X_1^2}$
  for all $(t_1,t_2,\ldots,t_d)\in \supp w_T$, and there exists $t_0 \in\mathbb{R}$ such that $ \frac{\partial}{\partial t_1} h(t_0,t_2,\ldots,t_d)=0$.
  Suppose that $Y/X^2 \geq R \geq 1$. Then
  \[
    I
    = \frac{X_1}{\sqrt{Y}} e^{i h(t_0,t_2,\ldots,t_d)} W_T(t_2,\ldots,t_d) + O_A(X_1 R^{-A}),
  \]
  for some $X$-inert family of functions $W_T$ and any $A>0$.
\end{lemma}

\begin{proof}
  See \cite[\S 8]{BKY13} and \cite[\S 3]{KPY}.
\end{proof}

\section{Bounding the moments of $L$-functions}\label{sec:ProvingMoment}
In this section, we prove Theorem \ref{thm:4m-of-L} and  \ref{thm:8m-of-L}. All the details of proving Theorem \ref{thm:4m-of-L} will be given. Since it is roughly a parallel way to prove Theorem \ref{thm:8m-of-L}, we will omit the repeated proof steps of Theorem \ref{thm:8m-of-L}.
\subsection{Initial setup}
Let $F=\Phi$ or $\E$,
using Lemma \ref{lemma:AFE_gl8}, to prove Theorem \ref{thm:4m-of-L}, it suffices to show, for all $P\ll T^{2+\varepsilon}$,
\[
H:=\sum_{t_j}w(t_j)|\rho_j(1)|^2\Big|\sum_{m\geq 1}\sum_{n\geq 1}\frac{A_{F}(n, m, 1)\lambda_j(n)}{(m^2n)^{\frac{1}{2}+ it_j}}W_1\Big(\frac{m^2n}{P}\Big)\Big|^2\ll T^{2+\varepsilon}.
\]
Here we remove the condition that $\phi_j$ is even in the summation of spectral parameters. 
By Cauchy--Schwarz inequality, we have that,
\[
H\ll T^\varepsilon\sum_{m\ll \sqrt{P}}\frac{H_m}{m}
\]
where 
\[
H_m:=\sum_{t_j}w(t_j)|\rho_j(1)|^2\Big|\sum_{n\sim N}\frac{A_{F}(n, m, 1)\lambda_j(n)}{n^{\frac{1}{2}+ it_j}}W_1\Big(\frac{n}{N}\Big)\Big|^2
\]
and $N=\frac{P}{m^2}$. Now it is enough to show the following result.
\begin{proposition}\label{Prop:BoundHm}
  With the above notations and $N=\frac{P}{m^2}\ll \frac{T^{2+\varepsilon}}{m^2}$, we have
  \[
  H_m\ll T^{2+\varepsilon}\Big(1+\sum_{n\sim N}\frac{|A_{F}(n, m, 1)|^2}{n}\Big).
  \]
\end{proposition}
Thus from Proposition \ref{Prop:BoundHm} and Lemma \ref{lemma:RamanujanOnAverage}, we have that
\[
H\ll T^{2+\varepsilon}\sum_{m\ll \sqrt{P}}\frac{1}{m}\Big(1+\sum_{n\sim \frac{P}{m^2}}\frac{|A_{F}(n, m, 1)|^2}{n}\Big)\ll T^{2+\varepsilon}
\]
as desired.
\subsection{Reduction of Proposition \ref{Prop:BoundHm}}
When $N\ll T$, Proposition \ref{Prop:BoundHm} follows immediately from an application of Theorem \ref{thm:Luo'sLS}.
For $N\gg T$, we apply Theorem \ref{thm:Young'sLS} to $H_m$, with $\mathcal{A}=\{A_F(n, m, 1)W_1(\frac{n}{N})\}_{n\sim N}$,
\[
S(\mathcal{A})-S_1(\mathcal{A}; X)\ll T^{2+\varepsilon}\Big(\sum_{n\sim N}\frac{|A_F(n, m, 1)|^2}{n}\Big),
\]
upon choosing $X=\min\{T, \frac{N}{T}\}$ and using $N=\frac{P}{m^2}$, $P\ll T^{2+\varepsilon}.$
Using Lemma \ref{lemma:RamanujanOnAverage}, we have
\[
S(\mathcal{A})=S_1(\mathcal{A}; X)+O(T^{2+\varepsilon}).
\]
Thus we find that in order to bound $H_m$, we need to bound 
\[
T\sum_{r<X}\frac{1}{r^2}\sum_{0\neq |k|\ll rT^\varepsilon}\int_{-T^\varepsilon}^{T^\varepsilon}
  \min\Big\{\frac{1}{|u|}, \frac{r/|k|}{1+u^2}\Big\}\Big|\frac{1}{\sqrt{N}}\sum_{n}A_F(n, m, 1)S(k, n; r)w_3\Big(\frac{n}{N}\Big)e\Big(\frac{un}{rT}\Big)\Big|^2\dd u,
\] 
where $w_3(x)=\frac{W_1(x)}{\sqrt{x}}$, $N=\frac{P}{m^2}\ll \frac{T^{2+\varepsilon}}{m^2}$ and $1\leq X=\min\{T, \frac{N}{T}\}.$
Let $R\leq X$ and $K\ll RN^\varepsilon\ll RT^\varepsilon$. It is sufficient to consider the dyadic sum 
\begin{equation}\label{eqn:I(R,K;m)}
\mathcal{I}(R, K; m)=T\sum_{r\sim R}\frac{1}{r^2}\sum_{|k|\sim K}\int_{-T^\varepsilon}^{T^\varepsilon}
  g(u)\Big|\frac{1}{\sqrt{N}}\sum_{n}A_F(n, m, 1)S(k, n; r)w_3\Big(\frac{n}{N}\Big)e\Big(\frac{un}{rT}\Big)\Big|^2\dd u
\end{equation} 
where $$g(u)=g(u, R, K)=\min\Big\{\frac{1}{|u|}, \frac{R/K}{1+u^2}\Big\}.$$
It now suffices to prove the following lemma.
\begin{lemma}\label{lemma:I(R,K;m)}
For any fixed $1\leq m\ll T^{1+\varepsilon}$, let $T\ll N\ll \frac{T^{2+\varepsilon}}{m^2}$, $R\leq X$, $K\ll RT^\varepsilon$, and 
$1\leq X=\min\{T, \frac{N}{T}\}$. Then
\[
\mathcal{I}(R, K; m)\ll T^{2+\varepsilon}.
\]
\end{lemma}
\subsection{Fourier analysis on $\mathcal{I}(R, K; m)$}
From \eqref{eqn:I(R,K;m)}, opening the Kloosterman sum, we have
\begin{equation*}
\mathcal{I}(R, K; m)=\frac{T}{NR^2}\sum_{r\sim R}\sum_{|k|\sim K}\int_{-T^\varepsilon}^{T^\varepsilon}
  g(u)\Big|\sideset{}{^{*}}\sum_{a\mod r}e\Big(\frac{ak}{r}\Big)\sum_{n}A_F(n, m, 1)e\Big(\frac{\overline{a}n}{r}\Big)w_3\Big(\frac{n}{N}\Big)e\Big(\frac{un}{rT}\Big)\Big|^2\dd u.
\end{equation*}
Firstly, we apply the Voronoi summation formula (Lemma \ref{lemma:Voronoi}) to $n$-sum. 
Let
\begin{equation}\label{eqn:func_W}
W(x; u, r)=W(x)=w_3\Big(\frac{x}{N}\Big)e\Big(\frac{ux}{rT}\Big).
\end{equation}
Then 
\begin{multline}\label{eqn:UseVoronoi}
\sum_{n}A_F(n, m, 1)e\Big(\frac{an}{r}\Big)W(n)
=\delta_F\Res_{s=1}\widetilde{W}(s)\mathcal{L}(s, m, r)\\+
\frac{r}{2}\sum_{d_1\mid mr}\sum_{d_2\mid\frac{mc}{d_1}}\sum_{n\neq 0}\frac{A_F(d_1,d_2, n)\Kl(\overline{a}, n, r; m, 1, d_1, d_2)}{|n|d_1d_2}\W_{F}\Big(\frac{nd_2^2d_1^3}{r^4m^2}; u, r\Big),
\end{multline}
where $\delta_F=\begin{cases}
                  0, & \mbox{if } F=\Phi, \\
                  1, & \mbox{if } F=\E,
                \end{cases}$
$\W_{F}$ is defined analogously in Lemma $\ref{lemma:Voronoi}$ and depends on $u$ and $r$ due to the choice of function $W$.
Now we arrive at 
\begin{equation}\label{eqn:afterVoronoi}
\begin{split}
\mathcal{I}(R, K; m)&\ll\frac{T}{NR^2}\sum_{r\sim R}\sum_{|k|\sim K}
\int_{-T^\varepsilon}^{T^\varepsilon}g(u)\Big|\sideset{}{^{*}}\sum_{a\mod r}e\Big(\frac{ak}{r}\Big)\Big(\delta_F\Res_{s=1}\widetilde{W}(s)\mathcal{L}(s, m, r)\\
&\hspace{7em}+\frac{r}{2}\sum_{d_1\mid mr}\sum_{d_2\mid\frac{mc}{d_1}}\sum_{n\neq 0}\frac{A_F(d_1,d_2, n)\Kl(\overline{a}, n, r; m, 1, d_1, d_2)}{|n|d_1d_2}\W_{F}\Big(\frac{nd_2^2d_1^3}{r^4m^2}; u, r\Big)\Big)\Big|^2\dd u\\
&\ll\frac{T}{NR^2}\sum_{r\sim R}\sum_{|k|\sim K}
\int_{-T^\varepsilon}^{T^\varepsilon}g(u)\Big(|\delta_F\Res_{s=1}\widetilde{W}(s)\mathcal{L}(s, m, r)|^2\Big|S(0, k; r)\Big|^2\\
&\hspace{5em}+\Big|r\sideset{}{^{*}}\sum_{a\mod r}e\Big(\frac{ak}{r}\Big)\sum_{d_1\mid mr}\sum_{d_2\mid\frac{mc}{d_1}}\sum_{n\neq 0}\frac{A_F(d_1,d_2, n)\Kl(\overline{a}, n, r; m, 1, d_1, d_2)}{|n|d_1d_2}\W_{F}\Big(\frac{nd_2^2d_1^3}{r^4m^2}; u, r\Big)\Big|^2\Big)\dd u.\\
\end{split}
\end{equation}

Notice that when $F=\E$, $A_F=\tau$, there is a zero frequency term (residue term) in the Voronoi summation formula.
By Cauchy integral formula, we have, 
\begin{equation}\label{eqn:CauchyInt}
\Res_{s=1}\widetilde{W}(s)\mathcal{L}(s, m, r)=\frac{1}{2\pi i}\oint_{|s|=\frac{1}{\log^2 T}}\widetilde{W}(s)\mathcal{L}(s, m, r)\dd s
\end{equation}
Note that $\widetilde{W}(s)=N^s\int_{0}^{\infty}w_3(x)e(\frac{uNx}{rT})x^{s-1}\dd x$ is holomorphic and we have the Taylor expansion at $s=1$:
\[
\widetilde{W}(s)=\widetilde{W}(1)+\widetilde{W}^{'}(s)|_{s=1}(s-1)+\frac{\widetilde{W}^{(2)}(s)|_{s=1}}{2}(s-1)^2+\cdots.
\]
Since that $r\sim R$ and $N\gg RT$, for $\frac{RT}{N}\leq |u|\ll T^\varepsilon$, by repeated integration by parts, the above Taylor coefficients satisfy 
\[
\widetilde{W}^{(j)}(s)|_{s=1}\ll_{j, A, \varepsilon} N^{1+\varepsilon}\Big(\frac{uN}{RT}\Big)^{-A} 
\text{ for any integers } j , A\geq 0.
\] 
For $0\leq |u|\leq T^\varepsilon$, we have the trivial bound $\widetilde{W}(s)\ll_{j, \varepsilon} N^{1+\varepsilon}. $
So we have,  when $|s-1|=\frac{1}{\log^2T}$, 
\begin{equation}\label{eqn:ResW}
\widetilde{W}(s)\ll_{\varepsilon, A} 
N^{1+\varepsilon}\Big(\frac{uN}{RT}\Big)^{-A}
\end{equation}
for any integer $A\geq 0.$
From the explicit expression of $\mathcal{L}(s, m, r)$ in \eqref{eqn:singular} with $|s-1|= \frac{1}{\log^2 T}$, we have, 
$\zeta(s)^4\ll T^\varepsilon$ and $r^s=r(1+ O(\frac{1}{\log T}))$ since $r\sim R\ll X\leq T$ and similar arguments for other terms of $m$'s divisors. The Euler product in \eqref{eqn:singular} is
\[
\prod_{p^\alpha||\frac{r\kappa}{de\ell}}
\Big((1-\frac{1}{p^s})^{4}\sum_{j\geq 0}\frac{d_4(p^{\alpha+j})}{p^{js}}\Big)
\ll \prod_{p^\alpha||\frac{r\kappa}{de\ell}}\Big((1+\frac{1}{p^{1-\frac{1}{\log^2T}}})^{4}
\sum_{j\geq 0}\frac{d_4(p^{\alpha+j})}{p^{j(1-\frac{1}{\log^2 T})}}\Big)\ll T^\varepsilon.
\]
Therefore, for $|s-1|= \frac{1}{\log^2 T}$, we have 
\begin{equation}\label{eqn:ResL}
\mathcal{L}(s, m, r)\ll \frac{T^\varepsilon}{r}.
\end{equation}
Using \eqref{eqn:CauchyInt} with \eqref{eqn:ResW} and \eqref{eqn:ResL}, we have
\begin{equation}\label{eqn:ResBound}
\Res_{s=1}\widetilde{W}(s)\mathcal{L}(s, m, r)\ll_{\varepsilon, A} 
                        T^\varepsilon \frac{N}{r}\Big(\frac{uN}{RT}\Big)^{-A}, 
\end{equation}
for any integer $A\geq 0.$
Since $|S(0, k; r)|\leq (k, r)$ \cite[Eqn (3.5)]{I-K04},
we have the averaged bound for Ramanujan sums
\[
\sum_{|k|\sim K}|S(0, k; r)|^2\leq \sum_{|k|\sim K}|(k, r)|^2\ll \sum_{d\mid r}d^2\sum_{d\mid k\atop k\sim K}1
\ll \sum_{d\mid r}d^2\frac{K}{d}\ll RKT^\varepsilon.
\]
Therefore the contribution of residue term in \eqref{eqn:afterVoronoi} is bounded by
\begin{equation*}
\begin{split}
\frac{T^{1+\varepsilon}K}{N}\Big(\int_{\frac{RT^{1+\varepsilon}}{N}\leq |u|\ll T^\varepsilon}\frac{1}{|u|}\Big|\frac{N}{R}\Big(\frac{uN}{RT}\Big)^{-A}\Big|^2\dd u
+\int_{0\leq |u|\leq \frac{RT^{1+\varepsilon}}{N}}\frac{R/K}{1+u^2}\frac{N^2}{R^2}\dd u\Big)
\ll \frac{T^{1+\varepsilon}K}{N}\frac{R}{K}\frac{RT^{1+\varepsilon}}{N}\frac{N^2}{R^2}\ll T^{2+\varepsilon},
\end{split}
\end{equation*}
which corresponds our desired bound.

Now we deal with the non-zero frequency part in \eqref{eqn:afterVoronoi}. Writing
\begin{equation}\label{eqn:I_0}
\begin{split}
\mathcal{I}_0(R, K; m)&=\frac{T}{N}\sum_{r\sim R}\sum_{|k|\sim K}
\int_{-T^\varepsilon}^{T^\varepsilon}g(u)\Big|\sideset{}{^{*}}\sum_{a\mod r}e\Big(\frac{ak}{r}\Big)\\
&\hspace{5em}\times \sum_{d_1\mid mr}\sum_{d_2\mid\frac{mc}{d_1}}\sum_{n\neq 0}\frac{A_F(d_1,d_2, n)\Kl(\overline{a}, n, r; m, 1, d_1, d_2)}{|n|d_1d_2}\W_{F}\Big(\frac{nd_2^2d_1^3}{r^4m^2}; u, r\Big)\Big|^2\dd u.
\end{split}
\end{equation}
To prove Lemma \ref{lemma:I(R,K;m)}, it suffices to show that $\mathcal{I}_0(R, K; m)\ll T^{2+\varepsilon}$.

Firstly, we square out the expression of $\mathcal{I}_0(R, K; m)$ in \eqref{eqn:I(R,K;m)}, put a smooth weight in $k$, and use the fact that $\frac{1}{r^2}\asymp \frac{1}{R^2}$, we then get that
\begin{equation}
\begin{split}
\mathcal{I}_0(R, K; m)&\ll\frac{T}{N}\sum_{r\sim R}\sum_{k}w_1\Big(\frac{k}{K}\Big)
\int_{-T^\varepsilon}^{T^\varepsilon}g(u)\sideset{}{^{*}}\sum_{a_1\mod r}\, \sideset{}{^{*}}\sum_{a_2\mod r}e\Big(\frac{(a_1-a_2)k}{r}\Big)\\
&\hspace{5em}\times \sum_{d_1\mid mr}\sum_{d_2\mid\frac{mc}{d_1}}\sum_{n\neq 0}\frac{A_F(d_1,d_2, n)\Kl(\overline{a_1}, n, r; m, 1, d_1, d_2)}{|n|d_1d_2}\W_{F}\Big(\frac{nd_2^2d_1^3}{r^4m^2}; u, r\Big)\\
&\hspace{5em}\times \sum_{d_1'\mid mr}\sum_{d_2'\mid\frac{mc}{d_1'}}\sum_{n'\neq 0}\frac{A_F(d_1',d_2', n')\Kl(\overline{a_2}, n', r; m, 1, d_1', d_2')}{|n|d_1'd_2'}\W_{F}\Big(\frac{nd_2'^2d_1'^3}{r^4m^2}; u, r\Big)\dd u.
\end{split}
\end{equation}
where $w_1$ is a smooth compactly supported function. Then we apply Poisson summation formula to the $k$-sum. This gives 
\[
\sum_{k}w_1\Big(\frac{k}{K}\Big)e\Big(\frac{(a_1-a_2)k}{r}\Big)=K\sum_{j\equiv a_2-a_1\mod r}\widehat{w_1}\Big(\frac{Kj}{r}\Big).
\]
Therefore
\begin{equation*}
\begin{split}
\mathcal{I}_0(R, K; m)&\ll\frac{TK}{N}\sum_{r\sim R}\sum_{j}\widehat{w_1}\Big(\frac{Kj}{r}\Big)\sideset{}{^{*}}\sum_{a_1\mod r\atop (a_1+j,r)=1}\int_{-T^\varepsilon}^{T^\varepsilon}g(u)\\
&\hspace{5em}
\times\sum_{d_1\mid mr}\sum_{d_2\mid\frac{mc}{d_1}}\sum_{n\neq 0}\frac{A_F(d_1,d_2, n)\Kl(\overline{a_1}, n, r; m, 1, d_1, d_2)}{|n|d_1d_2}\W_{F}\Big(\frac{nd_2^2d_1^3}{r^4m^2}; u, r\Big)\\
&\hspace{5em}\times \sum_{d_1'\mid mr}\sum_{d_2'\mid\frac{mc}{d_1'}}\sum_{n'\neq 0}\frac{A_F(d_1',d_2', n')\Kl(\overline{(a_1+j)}, n', r; m, 1, d_1', d_2')}{|n|d_1'd_2'}\W_{F}\Big(\frac{nd_2'^2d_1'^3}{r^4m^2}; u, r\Big)\dd u.
\end{split}
\end{equation*}
Using the elementary inequality and remove the condition $(a_1, r)=1$ or $(a_1+j,r)=1$ by positivity, we have
\begin{equation*}
\begin{split}
\mathcal{I}_0(R, K; m)&\ll\frac{TK}{N}\sum_{r\sim R}\sum_{j}\widehat{w_1}\Big(\frac{Kj}{r}\Big)\sideset{}{^{*}}\sum_{a\mod r}\int_{-T^\varepsilon}^{T^\varepsilon}g(u)\\
&\hspace{5em}
\times\Big|\sum_{d_1\mid mr}\sum_{d_2\mid\frac{mc}{d_1}}\sum_{n\neq 0}\frac{A_F(d_1,d_2, n)\Kl(\overline{a}, n, r; m, 1, d_1, d_2)}{|n|d_1d_2}\W_{F}\Big(\frac{nd_2^2d_1^3}{r^4m^2}; u, r\Big)\Big|^2\dd u.
\end{split}
\end{equation*}
Since $r\sim R$ and $R\gg K$, by rapid decay of $\widehat{w_1}$, we truncate the $j$-sum in $|j|\ll \frac{RT^{1+\varepsilon}}{K}$, thus
\[
\mathcal{I}_1(R, K; m)\ll\frac{T^{1+\varepsilon}R}{N}\sum_{r\sim R}\,\sideset{}{^{*}}\sum_{a\mod r}\int_{-T^\varepsilon}^{T^\varepsilon}g(u)
\Big|\sum_{d_1\mid mr}\sum_{d_2\mid\frac{mc}{d_1}}\sum_{n\neq 0}\frac{A_F(d_1,d_2, n)\Kl(\overline{a}, n, r; m, 1, d_1, d_2)}{|n|d_1d_2}\W_{F}\Big(\frac{nd_2^2d_1^3}{r^4m^2}; u, r\Big)\Big|^2\dd u.
\]
After Cauchy--Schwarz inequality in $d_1, d_2$-sum and considering only positive $n$ due to symmetry, now we need to bound 
\begin{multline}\label{eqn:I_1}
\mathcal{I}_1(R, K; m):=\frac{T^{1+\varepsilon}R}{N}\int_{-T^\varepsilon}^{T^\varepsilon}g(u)
\sum_{r\sim R}\, \sideset{}{^{*}}\sum_{a\mod r}\sum_{d_1\mid mr}\sum_{d_2\mid\frac{mc}{d_1}}\frac{1}{(d_1d_2)^2}\\
\times\Big|\sum_{n>0}\frac{A_F(d_1,d_2, n)\Kl(\overline{a}, n, r; m, 1, d_1, d_2)}{n}\W_{\pm, F}\Big(\frac{nd_2^2d_1^3}{r^4m^2}; u, r\Big)\Big|^2\dd u,
\end{multline}
where $\W_{\pm, F}$ is \eqref{eqn:Voronoi_transf} with gamma factors in \eqref{eqn:Gamma_factors} in different cases for $F$.

\subsection{Simplifying exponential sums}
Now we deal with the exponential sums in the hyper-Kloosterman sum. Moreover, the bound for $\W_{-, F}$ can be evaluated in the same way as $\W_{+, F}$, so we consider only $\W_{+, F}$. By the Cauchy--Schwarz inequality, changing variable for $a$ to $\overline{a}$ and completing summation over $a$, we have $\mathcal{I}_1(R, K; m)$  is bounded by
\begin{equation*}
\begin{split}
&\ll\frac{T^{1+\varepsilon}R}{N}\int_{-T^\varepsilon}^{T^\varepsilon}g(u)
\sum_{r\sim R}\sum_{a\mod r}\sum_{d_1\mid mr}\sum_{d_2\mid\frac{mc}{d_1}}\frac{1}{(d_1d_2)^2}\\
&\hspace{5em}\times\Big|\sum_{n>0}\frac{A_F(d_1,d_2, n)\Kl(a, n, r; m, 1, d_1, d_2)}{n}\W_{+, F}\Big(\frac{nd_2^2d_1^3}{r^4m^2}; u, r\Big)\Big|^2\dd u\\
&=\frac{T^{1+\varepsilon}R}{N}\int_{-T^\varepsilon}^{T^\varepsilon}g(u)\sum_{r\sim R}\sum_{a\mod r}\sum_{d_1\mid mr}\sum_{d_2\mid\frac{mc}{d_1}}\frac{1}{(d_1d_2)^2}\\
&\hspace{5em}\times \sum_{n_1>0}\frac{A_F(d_1,d_2, n_1)A_F(d_1,d_2, n_2)}{n_1n_2}\W_{+, F}\Big(\frac{n_1d_2^2d_1^3}{r^4m^2}; u, r\Big)\overline{\W_{+, F}\Big(\frac{nd_2^2d_1^3}{r^4m^2}; u, r\Big)}\\
&\times \sideset{}{^*}\sum_{x_1\mod \frac{mr}{d_1}}\,\sideset{}{^*}\sum_{x_1'\mod \frac{mr}{d_1}}\,
\sideset{}{^*}\sum_{x_2\mod \frac{mr}{d_1d_2}}\,\sideset{}{^*}\sum_{x_2'\mod \frac{mr}{d_1d_2}}
e\Big(\frac{d_1(x_1-x_1')a}{r}+\frac{d_2(x_2\overline{x_1}-x_2'\overline{x_1'})}{\frac{mr}{d_1}}+\frac{n_1\overline{x_2}-n_2\overline{x_2'}}{\frac{mr}{d_1d_2}}\Big)\dd u.
\end{split}
\end{equation*}
Next we sum over $a$ and see that $d_1x_1\equiv d_1x_1'\mod r$ by orthogonality, which implies $x_1\equiv x_1' \mod\frac{r}{(r, d_1)}$. Thus we may write 
$x_1'=x_1+\frac{r}{(r, d_1)}y$, where $y$ runs through residues mod $\frac{(r, d_1)m}{d_1}$, such that $(x_1+\frac{r}{(r, d_1)}y, \frac{rm}{d_1})=1$. For simplicity, 
let $\sideset{}{^\sharp}\sum_{y \mod\frac{(r, d_1)m}{d_1}}$ denote the sum over such $y$. Thus our sum becomes
\begin{multline}\label{eqn:Aftersumming_a}
  \frac{T^{1+\varepsilon}R}{N}\int_{-T^\varepsilon}^{T^\varepsilon}g(u)
\sum_{r\sim R}r\sum_{d_1\mid mr}\sum_{d_2\mid\frac{mc}{d_1}}\frac{1}{(d_1d_2)^2}\sideset{}{^*}\sum_{x_1\mod \frac{mr}{d_1}}\,
\sideset{}{^\sharp}\sum_{y \mod\frac{(r, d_1)m}{d_1}}S_1S_2\dd u\\
\ll \frac{T^{1+\varepsilon}R^2}{N}\int_{-T^\varepsilon}^{T^\varepsilon}g(u)
\sum_{r\sim R}\sum_{d_1\mid mr}\sum_{d_2\mid\frac{mc}{d_1}}\frac{1}{(d_1d_2)^2}\sideset{}{^*}\sum_{x_1\mod \frac{mr}{d_1}}\,
\sideset{}{^\sharp}\sum_{y \mod\frac{(r, d_1)m}{d_1}}(|S_1|^2+|S_2|^2)\dd u,
\end{multline}
where 
\[
S_1=\sum_{n_1>0}\frac{A_F(d_1,d_2, n_1)}{n_1}\W_{+, F}\Big(\frac{n_1d_2^2d_1^3}{r^4m^2}; u, r\Big)
\sideset{}{^*}\sum_{x_2\mod \frac{mr}{d_1d_2}}
e\Big(\frac{d_2x_2\overline{x_1}}{\frac{mr}{d_1}}+\frac{n_1\overline{x_2}}{\frac{mr}{d_1d_2}}\Big),
\]
and 
\[
S_2=\sum_{n_2>0}\frac{A_F(d_1,d_2, n_2)}{n_2}\overline{\W_{+, F}\Big(\frac{n_2d_2^2d_1^3}{r^4m^2}; u, r\Big)}
\sideset{}{^*}\sum_{x_2'\mod \frac{mr}{d_1d_2}}
e\Big(-\frac{d_2x_2\overline{x_1+\frac{r}{(r, d_1)}y}}{\frac{mr}{d_1}}-\frac{n_2\overline{x_2'}}{\frac{mr}{d_1d_2}}\Big).
\]
Inside $S_2$, we may use the change of variables $x=\overline{x_1+\frac{r}{(r, d_1)}y}$. The condition on $y$ becomes that
$(\overline{x}-\frac{r}{(r, d_1)}y, \frac{rm}{d_1})=1 $. After this change of variables, we extend the $y$-sum to all residues mod $\frac{(r, d_1)m}{d_1}$. Thus,
\begin{equation*}
\begin{split}
&\sum_{r\sim R}\sum_{d_1\mid mr}\sum_{d_2\mid\frac{mc}{d_1}}\frac{1}{(d_1d_2)^2}\sideset{}{^*}\sum_{x_1\mod \frac{mr}{d_1}}\,
\sideset{}{^\sharp}\sum_{y \mod\frac{(r, d_1)m}{d_1}}|S_2|^2\\
\ll&\sum_{r\sim R}\sum_{d_1\mid mr}\sum_{d_2\mid\frac{mc}{d_1}}\frac{1}{(d_1d_2)^2}\sideset{}{^*}\sum_{x\mod \frac{mr}{d_1}}\sum_{y \mod\frac{(r, d_1)m}{d_1}}|S_2|^2\\
=&\frac{(r, d_1)m}{d_1}\sum_{r\sim R}\sum_{d_1\mid mr}\sum_{d_2\mid\frac{mc}{d_1}}\frac{1}{(d_1d_2)^2}\sideset{}{^*}\sum_{x_1\mod \frac{mr}{d_1}}\sum_{y \mod\frac{(r, d_1)m}{d_1}}|S_1|^2.
\end{split}
\end{equation*}
By a further change of variables from $x_1$ to $\overline{x_1}$, the fact that $S_1$ is independent of $y$ and $\frac{(r, d_1)}{d_1}\leq 1$, the quality in \eqref{eqn:Aftersumming_a} is bounded by
\begin{multline*}
\frac{mT^{1+\varepsilon}R^2}{N}\int_{-T^\varepsilon}^{T^\varepsilon}g(u)
\sum_{r\sim R}\sum_{d_1\mid mr}\sum_{d_2\mid\frac{mc}{d_1}}\frac{1}{(d_1d_2)^2}\sideset{}{^*}\sum_{x_1\mod \frac{mr}{d_1}}
|S_1|^2\dd u\\
\ll \frac{mT^{1+\varepsilon}R^2}{N}\int_{-T^\varepsilon}^{T^\varepsilon}g(u)
\sum_{r\sim R}\sum_{d_1\mid mr}\sum_{d_2\mid\frac{mc}{d_1}}\frac{1}{(d_1d_2)^2}\sideset{}{^*}\sum_{x_1\mod \frac{mr}{d_1}}\\
\times 
\Big|\sum_{n>0}\frac{A_F(d_1,d_2, n)}{n}\W_{+, F}\Big(\frac{nd_2^2d_1^3}{r^4m^2}; u, r\Big)
\sideset{}{^*}\sum_{x_2\mod \frac{mr}{d_1d_2}}
e\Big(\frac{d_2x_2\overline{x_1}}{\frac{mr}{d_1}}+\frac{n\overline{x_2}}{\frac{mr}{d_1d_2}}\Big)\Big|^2\dd u.
\end{multline*}
Now we may extend the sum over $x_1$ to all residues mod $\frac{rm}{d_1}$ by positivity. Opening the square produces two sums $x_2, x_2'\mod \frac{mr}{d_1d_2}$. However, by orthogonality, the sum over $x_1$ gives the condition $d_2x_2\equiv d_2x_2'\mod \frac{rm}{d_1}$, which implies $x_2\equiv x_2'\mod \frac{rm}{d_1d_2}$ due to $d_2\mid \frac{rm}{d_1}$. So the above sum is
\begin{equation*}
\frac{m^2T^{1+\varepsilon}R^3}{N}\int_{-T^\varepsilon}^{T^\varepsilon}g(u)
\sum_{r\sim R}\sum_{d_1\mid mr}\sum_{d_2\mid\frac{mc}{d_1}}\frac{1}{d_1^3d_2^2}\,\sideset{}{^*}\sum_{x\mod \frac{mr}{d_1d_2}}
\Big|\sum_{n>0}\frac{A_F(d_1,d_2, n)}{n}\W_{+, F}\Big(\frac{nd_2^2d_1^3}{r^4m^2}; u, r\Big)
e\Big(\frac{nx}{\frac{mr}{d_1d_2}}\Big)\Big|^2\dd u,
\end{equation*}
where we have used a change of variables $x=\overline{x_2}$. Next we write $r_1=rm$, switch the sums $d_1, d_2$ and $r$ and drop condition $m\mid r_1$. Thus the above expression is bounded by
\begin{equation}\label{eqn:aftersumming_x1}
\begin{split}
&\frac{m^2T^{1+\varepsilon}R^3}{N}\int_{-T^\varepsilon}^{T^\varepsilon}g(u)
\sum_{d_1\ll Rm}\sum_{d_2\ll Rm}\frac{1}{d_1^3d_2^2}\sum_{r_1\sim Rm\atop d_1d_2\mid r_1}\,\sideset{}{^*}\sum_{x\mod \frac{r_1}{d_1d_2}}\\
&\hspace{10em}
\Big|\sum_{n>0}\frac{A_F(d_1,d_2, n)}{n}\W_{+, F}\Big(\frac{nd_2^2d_1^3m^2}{r_1^4}; u, \frac{r_1}{m}\Big)
e\Big(\frac{nx}{\frac{r_1}{d_1d_2}}\Big)\Big|^2\dd u\\
=&\frac{m^2T^{1+\varepsilon}R^3}{N}\int_{-T^\varepsilon}^{T^\varepsilon}g(u)
\sum_{d_1\ll Rm}\sum_{d_2\ll Rm}\frac{1}{d_1^3d_2^2}\sum_{r\sim \frac{Rm}{d_1d_2}}\,\sideset{}{^*}\sum_{x\mod r}\\
&\hspace{10em}
\Big|\sum_{n>0}\frac{A_F(d_1,d_2, n)}{n}\W_{+, F}\Big(\frac{nm^2}{r^4d_1d_2^2}; u, \frac{rd_1d_2}{m}\Big)
e\Big(\frac{nx}{r}\Big)\Big|^2\dd u.
\end{split}
\end{equation}
Now we split $n$-sum into two ranges. We let $\mathcal{I}_{sm}(R, K; m)$ be the expression on the right hand side of \eqref{eqn:aftersumming_x1} with $n\leq \frac{R^4m^2}{Nd_2^2d_1^3}T^{\varepsilon_1}$ and $\mathcal{I}_{big}(R, K; m)$ be the same expression for $n>\frac{R^4m^2}{Nd_2^2d_1^3}T^{\varepsilon_1}$, where $\varepsilon_1$ is a fixed sufficiently small positive constant which will be chosen later.

Since the elementary inequality $|a+b|^2\leq 2(|a|^2+|b|^2)$, it now suffices to prove the following two propositions.
\begin{proposition}\label{Prop:I_sm}
With notations defined as above,
\[
\mathcal{I}_{sm}(R, K; m)\ll_\varepsilon T^{2+\varepsilon}.
\]
\end{proposition}

\begin{proposition}\label{Prop:I_big}
With notations defined as above,
\[
\mathcal{I}_{big}(R, K; m)\ll_\varepsilon T^{2+\varepsilon}.
\]
\end{proposition}
We will finish the proof of these two propositions in the remaining two subsections.
\subsection{Proof of Proposition \ref{Prop:I_sm}}\label{sec:I_sm}
We firstly deal with $\W_{+, F}(\frac{nm^2}{r^4d_1d_2^2}; u, \frac{rd_1d_2}{m})$ which defined in \eqref{eqn:Voronoi_transf} with function $W(x; u, \frac{rd_1d_2}{m})$ in \eqref{eqn:func_W}.
Here we only consider $F=\Phi=\phi\boxplus\phi$. Changing complex variable $s$ to $-2s+1$, we get that for any $\sigma>-2,$
\begin{equation*}
\begin{split}
\W_{+, F}(x; u, \frac{rd_1d_2}{m})
&=\frac{1}{2\pi i}\int_{(-\sigma-1)}\widetilde{W}(s)\pi^{4s-2}\left(\frac{\prod_{\pm}\Gamma(\frac{1-s\pm it_\phi}{2})}
{\prod_{\pm}\Gamma(\frac{s\pm it_\phi}{2})}\right)^2x^s\dd s\\
&=\frac{\pi}{ i}\int_{(2\sigma+3)}\widetilde{W}(-2s+1)\pi^{-8s}\left(\frac{\prod_{\pm}\Gamma(s\pm\frac{ it_\phi}{2})}{\prod_{\pm}\Gamma(\frac{1}{2}-s\pm \frac{it_\phi}{2})}\right)^2x^{1-2s}\dd s.
\end{split}
\end{equation*}
We can shift the contour to $\Re(s)=\sigma_1<\frac{1}{8}$. Note that
\begin{equation}\label{eqn:W_mellinbound}
\widetilde{W}(-2s+1)=\int_{0}^{\infty}w_3\Big(\frac{y}{N}\Big)e\Big(\frac{umy}{rd_1d_2T}\Big)y^{-2s}\dd y\ll N^{1-2\Re(s)}.
\end{equation}
For simplicity, we write
\[
\mathcal{G}(s)=\left(\frac{\prod_{\pm}\Gamma(s\pm\frac{ it_\phi}{2})}{\prod_{\pm}\Gamma(\frac{1}{2}-s\pm \frac{it_\phi}{2})}\right)^2\widetilde{W}(-2s+1)
\]
and note that
\begin{equation}\label{eqn:G_bound}
\mathcal{G}(s)\ll\frac{N^{1-2\Re(s)}}{|s|^{2-8\sigma_1}}\ll \frac{N^{1-2\Re(s)}}{(1+|s+t_\phi/2|)(1+|s-t_\phi/2|))^{\frac{1}{2}+\varepsilon}},
\end{equation}
by Stirling's formula and \eqref{eqn:W_mellinbound}. 
We get by Cauchy--Schwarz that
\begin{multline*}
\Big|\int_{(\sigma_1)}\mathcal{G}(s)
\sum_{0<n\leq \frac{R^4m^2}{Nd_2^2d_1^3}T^{\varepsilon_1}}\frac{A_F(d_1,d_2, n)}{n}e\Big(\frac{nx}{r}\Big)\Big(\frac{nm^2}{r^4d_1d_2^2}\Big)^{1-2s}\dd s\Big|^2\\
\ll \int_{(\sigma_1)}|\mathcal{G}(s)|
\Big|\sum_{0<n\leq \frac{R^4m^2}{Nd_2^2d_1^3}T^{\varepsilon_1}}\frac{A_F(d_1,d_2, n)}{n}e\Big(\frac{nx}{r}\Big)\Big(\frac{nm^2}{r^4d_1d_2^2}\Big)^{1-2s}\Big|^2\dd s
\int_{(\sigma_1)}|\mathcal{G}(s)|\dd s\\
\ll \Big(\frac{N^{\frac{1}{2}}m^2}{r^4d_1d_2^2}\Big)^{2-4\sigma_1}
\int_{-\infty}^{\infty}|\mathcal{G}(\sigma_1+it)|\Big|\sum_{0<n\leq \frac{R^4m^2}{Nd_2^2d_1^3}T^{\varepsilon_1}}
\frac{A_F(d_1,d_2, n)}{n^{2\sigma_1+2it}}e\Big(\frac{nx}{r}\Big)\Big|^2\dd t,
\end{multline*}
since by \eqref{eqn:G_bound}, 
\[
\int_{(\sigma_1)}|\mathcal{G}(s)|\dd s\ll N^{1-2\sigma_1}\int_{-\infty}^{+\infty}\frac{1\dd t}{((1+|t+t_\phi/2|)(1+|t-t_\phi/2|))^{\frac{1}{2}+\varepsilon}}
\ll N^{1-2\sigma_1}\int_{-\infty}^{+\infty}\frac{\dd v}{|1-v^2|^{\frac{1}{2}+\varepsilon}}\ll N^{1-2\sigma_1}. 
\]
We then apply a dyadic subdivision on $n$-sum, so that we examine sums $n\sim N_1$ for $N_1\leq \frac{R^4m^2}{Nd_2^2d_1^3}T^{\varepsilon_1}$.
In order to prove Proposition \ref{Prop:I_sm}, it suffices to prove
\begin{equation}
\begin{split}
&\frac{m^2T^{1+\varepsilon}R^3}{N}\int_{-T^\varepsilon}^{T^\varepsilon}g(u)\int_{-\infty}^{\infty}|\mathcal{G}(\sigma_1+it)|
\sum_{d_1\ll Rm}\sum_{d_2\ll Rm}\frac{1}{d_1^3d_2^2}\Big(\frac{NN_1d_2^2d_1^3}{R^4m^2}\Big)^{2-4\sigma_1}\frac{1}{(N^{\frac{1}{2}}N_1)^{2-4\sigma_1}}\\
&\hspace{10em}\times\sum_{r\sim \frac{Rm}{d_1d_2}}\,\sideset{}{^*}\sum_{x\mod r}
\Big|\sum_{n\sim N_1}\frac{A_F(d_1,d_2, n)}{n^{2\sigma_1+2it}}e\Big(\frac{nx}{r}\Big)\Big|^2\dd t\dd u\ll T^{2+\varepsilon}.
\end{split}
\end{equation}
Since $\frac{NN_1d_2^2d_1^3}{R^4m^2}\ll T^{\varepsilon_1}$ and $2-4\sigma_1>1$, $(\frac{NN_1d_2^2d_1^3}{R^4m^2})^{2-4\sigma_1}\ll T^{\varepsilon_1}\frac{NN_1d_2^2d_1^3}{R^4m^2}.$
Thus the right hand side of the equation above is 
\[
\frac{T^{1+\varepsilon}}{RN^{1-2\sigma_1}}\int_{-T^\varepsilon}^{T^\varepsilon}g(u)\int_{-\infty}^{\infty}|\mathcal{G}(\sigma_1+it)|
\sum_{d_1\ll Rm}\sum_{d_2\ll Rm}N_1^{4\sigma_1-1}\sum_{r\sim \frac{Rm}{d_1d_2}}\,\sideset{}{^*}\sum_{x\mod r}
\Big|\sum_{n\sim N_1}\frac{A_F(d_1,d_2, n)}{n^{2\sigma_1+2it}}e\Big(\frac{nx}{r}\Big)\Big|^2\dd t\dd u.
\]
Using the large sieve inequality, it is bounded by
\[
\frac{T^{1+\varepsilon}}{RN^{1-2\sigma_1}}\int_{-T^\varepsilon}^{T^\varepsilon}g(u)\int_{-\infty}^{\infty}|\mathcal{G}(\sigma_1+it)|
\sum_{d_1\ll Rm}\sum_{d_2\ll Rm}N_1^{4\sigma_1-1}\Big(\Big(\frac{Rm}{d_1d_2}\Big)^2+N_1\Big)
\sum_{n\sim N_1}\frac{|A_F(d_1,d_2, n)|^2}{n^{4\sigma_1}}\dd t\dd u.
\]
By \eqref{eqn:G_bound} and $\int_{-T^\varepsilon}^{T^\varepsilon}g(u)\dd u\ll T^\varepsilon$, it is
\begin{equation*}
\begin{split}
&\ll \frac{T^{1+\varepsilon}}{RN^{1-2\sigma_1}}
\sum_{d_1\ll Rm}\sum_{d_2\ll Rm}N_1^{4\sigma_1-1}\Big(\Big(\frac{Rm}{d_1d_2}\Big)^2+N_1\Big)
\sum_{n\sim N_1}\frac{|A_F(d_1,d_2, n)|^2}{n^{4\sigma_1}}\\
&\ll \frac{T^{1+\varepsilon}}{RN^{1-2\sigma_1}}
\sum_{d_1\ll Rm}\sum_{d_2\ll Rm}\Big(\frac{R^2m^2}{d_1d_2}+\frac{R^4m^2}{Nd_2d_1^2}\Big)\ll \frac{T^{1+\varepsilon}Rm^2}{N^{1-2\sigma_1}}+\frac{T^{1+\varepsilon}R^3m^2}{N^{2-2\sigma_1}}.
\end{split}
\end{equation*}
Since that $\sigma_1<\frac{1}{8}$ , $m\leq T^{1+\varepsilon}$, $N=\frac{T^{2+\varepsilon}}{m^2}$, $R\leq \frac{N}{T}$, we see
\[
\frac{T^{1+\varepsilon}Rm^2}{N^{1-2\sigma_1}}\ll \frac{T^{1+\varepsilon}Rm^2}{N^{\frac{3}{4}}}\leq T^\varepsilon N^{\frac{1}{4}}m^2\ll T^{\frac{1}{2}+\varepsilon}m^{\frac{3}{2}}\ll T^{2+\varepsilon},
\]
\[
\frac{T^{1+\varepsilon}R^3m^2}{N^{2-2\sigma_1}}\ll \frac{T^{1+\varepsilon}R^3m^2}{N^{\frac{7}{4}}}\leq \frac{T^\varepsilon N^{\frac{5}{4}}m^2}{T^2}\ll
\frac{T^{\frac{1}{2}+\varepsilon}}{m^\frac{1}{2}}\ll T^{\frac{1}{2}+\varepsilon}
\]
as desired. This finishes the proof of Proposition \ref{Prop:I_sm}.

\subsection{Proof of Proposition \ref{Prop:I_big}}\label{sec:I_big}
At the beginning of proof of Proposition \ref{Prop:I_big}, we apply a dyadic subdivision to the $n$-sum and the $u$-integral. So we investigate sums $n\sim N_2$ for $N_2\geq \frac{R^4m^2}{Nd_2^2d_1^3}T^{\varepsilon_1}$ and $u\sim U$ where $T^{-100}<|U|\leq T^\varepsilon$. This suffices since that we can truncate the sum to $n\ll T^{2025}$ due to the rapid decay of Mellin inversion of $W$ and there are $\ll \log^2 T$ such subdivisions, the interval $|u|\leq T^{-100}$ is trivially negligible.

From \eqref{eqn:aftersumming_x1}, it suffices to consider
\begin{equation}\label{eqn:J(R,K,N2,U)}
\begin{split}
\mathcal{J}(R, K, N_2, U)&:=\int_{u\sim U}g(u)
\sum_{d_1\ll Rm}\sum_{d_2\ll Rm}\frac{1}{d_1^3d_2^2}\sum_{r\sim \frac{Rm}{d_1d_2}}\,\sideset{}{^*}\sum_{x\mod r}\\
&\hspace{10em}
\Big|\sum_{n\sim N_2}\frac{A_F(d_1,d_2, n)}{n}\W_{+, F}\Big(\frac{nm^2}{r^4d_1d_2^2}; u, \frac{rd_1d_2}{m}\Big)
e\Big(\frac{nx}{r}\Big)\Big|^2\dd u\\
&\leq \int_{-\infty}^{\infty}g_1\Big(\frac{u}{U}\Big)g_2(U)
\sum_{d_1\ll Rm}\sum_{d_2\ll Rm}\frac{1}{d_1^3d_2^2}\sum_{r\sim \frac{Rm}{d_1d_2}}\,\sideset{}{^*}\sum_{x\mod r}\\
&\hspace{10em}
\Big|\sum_{n\sim N_2}\frac{A_F(d_1,d_2, n)}{n}\W_{+, F}\Big(\frac{nm^2}{r^4d_1d_2^2}; u, \frac{rd_1d_2}{m}\Big)
e\Big(\frac{nx}{r}\Big)\Big|^2\dd u,\\
\end{split}
\end{equation}
where $g_1(x)$ is a smooth compactly supported function in  $[\frac{1}{2},\frac{5}{2}]$ and $g_2(U)=\min\{\frac{1}{|U|}, \frac{R}{K}\}$, since we have
\begin{equation}\label{eqn:I_bigBoundbyJ}
  \mathcal{I}_{big}:=\mathcal{I}_{big}(R, K; m)\ll \frac{T^{1+\varepsilon}R^3m^2}{N}\sum_{\text{ dyadic } N_2 :\atop \frac{R^4m^2}{Nd_2^2d_1^3}T^{\varepsilon_1}\leq N_2\leq T^{2025}}
  \sum_{ \text{ dyadic } U :\atop T^{-100}<|U|\leq T^\varepsilon}\mathcal{J}(R, K, N_2, U).
\end{equation}
Opening the square in the right hand side of \eqref{eqn:J(R,K,N2,U)}, we see that
\begin{equation}\label{eqn:J_boundbyJ_0}
\mathcal{J}(R, K, N_2, U)\ll g_2(U)\sum_{d_1\ll Rm}\sum_{d_2\ll Rm}\frac{1}{d_1^3d_2^2}\mathcal{J}_0(d_1, d_2),
\end{equation}
where 
\begin{equation}\label{eqn:J_0(d1,d2)}
\mathcal{J}_0(d_1, d_2)=\sum_{r\sim \frac{Rm}{d_1d_2}}\,\sideset{}{^*}\sum_{x\mod r}
\sum_{n_1\sim N_2}\sum_{n_2\sim N_2}\frac{A_F(d_1,d_2, n_1)\overline{A_F(d_1,d_2, n_2)}}{n_1n_2}e\Big(\frac{(n_1-n_2)x}{r}\Big)
\frak{J}(r, n_1, n_2; d_1, d_2)
\end{equation}
and
\begin{equation}\label{eqn:frakJ}
  \frak{J}=\frak{J}(r, n_1, n_2; d_1, d_2)=\int_{-\infty}^{\infty}g_1\Big(\frac{u}{U}\Big)
  \W_{+, F}\Big(\frac{n_1m^2}{r^4d_1d_2^2}; u, \frac{rd_1d_2}{m}\Big)\overline{\W_{+, F}\Big(\frac{n_2m^2}{r^4d_1d_2^2}; u, \frac{rd_1d_2}{m}\Big)}
  \dd u.
\end{equation}

Now we consider the crucial case $F=\Phi=\phi\boxplus \phi$. Since that $\E$ can be regarded as $F$ with bounded spectral parameter and the approximation in \cite[Lemma 5.2]{C-L20} works for $F=\E$, it can be treated as the same method in \cite[\S 9]{C-L20}.
When $F=\Phi$ and $t_\phi$ varies in $1\ll t_\phi\ll T^\varepsilon$, we have the following lemma.
\begin{lemma}\label{Lemma:Analysis_W+F}
Let $F=\phi\boxplus \phi$ with $t_\phi\ll T^\varepsilon$ where $\varepsilon>0$ is a fixed arbitrarily small constant.
Let $X \gg \frac{T^{\varepsilon_1}}{N}$,  $Y>0$, $u\asymp U$ and we chose $\varepsilon_1\geq 24\varepsilon$.
Then for any $A>1$,
\[
\W_{+, F}(X; u, Y)\ll_{\varepsilon, A}T^{-A}, 
\]
unless $X\asymp \frac{|U|^4N^3}{Y^4T^4}$, in which case, we have
\[
\W_{+, F}(X; u, Y)=(NX)^{\frac{1}{2}}e\Big(-\frac{NU}{YT}\Big(\sum_{j\geq 0}(c_{1,j}\alpha^{-\frac{2j-2}{3}}+c_{2, j}\alpha^{-\frac{2j-1}{3}})\gamma^{2j}\Big)\Big)W_5\Big(\frac{u}{U}\Big)+O_{\varepsilon, A}(T^{-A})
\]
where  $\alpha=\frac{NXU}{u}(\frac{YT}{\pi N |U|})^4\asymp 1$,  $\gamma=\frac{t_\phi YT}{2\pi NU}\ll T^{-2\varepsilon}$, $c_{1,j}, c_{2, j}$ are certain constants
with $c_{1, 0}=0, c_{2, 0}=3, c_{1, 1}=4, c_{2, 1}=2, c_{1, 2}=-4/3, c_{2, 2}=-1$ and so on, $W_5$ is a $T^\varepsilon$-inert function.
\end{lemma}
\begin{proof}
From the definition of $\W_{+, F}$, for $x>0$, we have
\begin{equation}\label{eqn:W_+Fdef}
\W_{+, F}(X; u, Y)=\frac{1}{2\pi i}\int_{(-\sigma-1)}\widetilde{W}(s)\pi^{4s-2}\left(\frac{\prod_{\pm}\Gamma(\frac{1-s\pm it_\phi}{2})}
{\prod_{\pm}\Gamma(\frac{s\pm it_\phi}{2})}\right)^2X^s\dd s,
\end{equation}
with 
\[
\widetilde{W}(s)=\int_{0}^{+\infty}w_3\Big(\frac{x}{N}\Big)e\Big(\frac{ux}{YT}\Big)x^{s-1}\dd x
=N^{s}\int_{0}^{+\infty}w_3(x)e\Big(\frac{uN}{YT}x\Big)x^{s-1}\dd x.
\]
When $\frac{|U|N}{YT}\leq T^{3\varepsilon}$, by repeated partial integrals, we have, for $\Im(s)\gg 1,$
\begin{equation}\label{eqn:bound_MellinW}
\widetilde{W}(s)\ll N^{\Re(s)}\int_{0}^{+\infty}\Big|\frac{\dd^j w_3(x)e(\frac{uNx}{YT})}{\dd x^j}\Big|\Big|\frac{x^{s+j-1}}{s(s+1)\cdots(s+j-1)}\Big|\dd x
\ll_{j} T^{{3j\varepsilon}}\frac{N^{\Re(s)}}{(|\Im(s)|+1)^{j}}.
\end{equation}
Otherwise $\Im(s)\ll 1$, we have the trivial bound $\widetilde{W}(s)\ll N^{\Re(s)}$, hence the above bound still works.
By Stirling's formula, 
\[
\left(\frac{\prod_{\pm}\Gamma(\frac{1-s\pm it_\phi}{2})}
{\prod_{\pm}\Gamma(\frac{s\pm it_\phi}{2})}\right)^2\ll_{\Re(s)} \Big((1+|\Im(s)+t_\phi|)(1+|\Im(s)-t_\phi|)\Big)^{1-2\Re(s)}.
\]
We then shift the contour in \eqref{eqn:W_+Fdef} to $\Re(s)=-A$, by using the bound \eqref{eqn:bound_MellinW} with $j=6A$, get
\begin{equation*}
\begin{split}
\W_{+, F}(X; u, Y)&\ll X^{-A}\int_{-\infty}^{+\infty}|\widetilde{W}(-A+i\tau)|\Big((1+|\tau+t_\phi|)(1+|\tau-t_\phi|)\Big)^{1+2A}\dd \tau\\
&\ll_A T^{18A\varepsilon} (XN)^{-A}\int_{-\infty}^{+\infty}\frac{((1+|\tau+t_\phi|)(1+|\tau-t_\phi|))^{1+2A}}{(|\tau|+1)^{6A}}\dd \tau\\
&\ll_A T^{22A\varepsilon+2\varepsilon}(XN)^{-A}\ll T^{2\varepsilon+(22\varepsilon-\varepsilon_1)A}\ll T^{2\varepsilon-2\varepsilon A},
\end{split}
\end{equation*}
which is small by taking $A$ large.

When $\frac{|U|N}{YT}>T^{3{\varepsilon}}$, by shifting the contour to $\Re(s)=1/2,$ we have
\begin{equation*}
\W_{+, F}(X; u, Y)
=\frac{\pi}{2}\int_{-\infty}^{+\infty}\widetilde{W}(1/2+i\tau)\left(\frac{\prod_{\pm}\Gamma(\frac{1/2-i\tau\pm it_\phi}{2})}
{\prod_{\pm}\Gamma(\frac{1/2+i\tau\pm it_\phi}{2})}\right)^2X^{\frac{1}{2}+i\tau}\dd \tau,
\end{equation*}
with
\[
\widetilde{W}(1/2+i\tau)=\int_{0}^{+\infty}w_3\Big(\frac{x}{N}\Big)e\Big(\frac{ux}{YT}\Big)x^{-\frac{1}{2}+i\tau}\dd x
=N^{\frac{1}{2}+i\tau}\int_{0}^{+\infty}x^{-\frac{1}{2}}w_3(x)e\Big(\frac{uN}{YT}x+\frac{\tau}{2\pi}\log x\Big)\dd x.
\]
Let $h(x)=\frac{uN}{YT}x+\frac{\tau}{2\pi}\log x$, we have $h'(x)=\frac{uN}{YT}+\frac{\tau}{2\pi x}$ and  $h^{(j)}(x)\asymp_j \tau$ for $x\asymp 1$, $j\geq 2$. 
If $\frac{|U|N}{YT}\not\asymp |\tau|$ or $\sign(u)=\sign(\tau)$, $|h'(x)|\gg \max\{\frac{|U|N}{YT}, |\tau|\}\gg T^{3\varepsilon}.$
By using Lemma \ref{lemma:repeated_integration_by_parts}, we have
\[
\widetilde{W}(1/2+i\tau)\ll\begin{cases}
                           N^{\frac{1}{2}}|T|^{-\frac{3\varepsilon}{2} A}, & \mbox{if } |\tau|\ll T^{3\varepsilon}, \\
                           N^{\frac{1}{2}}|\tau|^{-\frac{A}{2}}, & \mbox{if } |\tau|\gg T^{3\varepsilon}.
                         \end{cases}
\]
Hence
\[
\W_{+, F}(X; u, Y)\ll (NX)^{\frac{1}{2}}\Big(\int_{|\tau|\ll T^{3\varepsilon}}|T|^{-\frac{3\varepsilon A}{2}}\dd \tau+
\int_{|\tau|\gg  T^{3\varepsilon}}|\tau|^{-\frac{A}{2}}\dd \tau\Big)\ll (NX)^{\frac{1}{2}}T^{-\frac{\varepsilon A}{3}},
\]
which is negligibly small say $O(T^{-A})$ by taking $A$ large. 

Therefore it suffices to consider the case of $\frac{|U|N}{YT}\asymp |\tau|$ and $\sign(u)=-\sign(\tau)$ (i.e. $\tau \asymp -\frac{UN}{YT}$).
Now the phase function $h(x, x_1, x_2)=\frac{UN}{YT}(x_1x-x_2\log x)$ with $x_1=\frac{u}{U}\asymp 1$ and $x_2=-\frac{YT}{2\pi NU}\tau\asymp 1$ which satisfies 
\[
\frac{\partial}{\partial x}h(x, x_1, x_2)=\frac{UN}{YT}(x_1-\frac{x_2}{x}),
\quad 
\frac{\partial^2}{\partial x^2}h(x, x_1, x_2)=\frac{UN}{YT}\frac{x_2}{x}\gg \frac{|U|N}{YT},
\]
\[
\frac{\partial^{a+ a_1+a_2}}{\partial x^{a}\partial x_1^{a_1}\partial x_2^{a_2}}h(x, x_1, x_2)\ll_{a, a_1, a_2} \frac{|U|N}{YT},
\]
for all $a, a_1, a_2\in \mathbb{Z}_{\geq 0}.$
By using Lemma \ref{lemma:stationary_phase}, with the stationary point $x_0=\frac{x_2}{x_1}=-\frac{\tau YT}{2\pi uN}$, we have
\[
\widetilde{W}(1/2+i\tau)=\Big(\frac{YT}{|U|}\Big)^{\frac{1}{2}}e\Big(\frac{\tau}{2\pi }\log(-\frac{\tau YT}{2\pi e u})\Big)W_4\Big(\frac{u}{U}, -\frac{YT\tau}{2\pi NU}\Big)+O(N^{\frac{1}{2}}T^{-\varepsilon A})
\]
where $W_4$ is a $1$-inert function with compactly supported in $\mathbb{R}_{+}^2$ since we are in the case of $u\asymp U$ and $\tau \asymp -\frac{UN}{YT}$.  Absorbing the multiplicative constants into $W_4$, we have 
\[
\W_{+, F}(X; u, Y)=\Big(\frac{XYT}{|U|}\Big)^{\frac{1}{2}}\int_{-\infty}^{\infty}W_4\Big(\frac{u}{U}, -\frac{YT\tau}{2\pi NU}\Big)
e\Big(\frac{\tau}{2\pi }\log(-\frac{\tau XYT}{2\pi e u})\Big)\left(\frac{\prod_{\pm}\Gamma(\frac{1/2-i\tau\pm it_\phi}{2})}
{\prod_{\pm}\Gamma(\frac{1/2+i\tau\pm it_\phi}{2})}\right)^2\dd \tau+O(T^{-A}).
\]
Since $|\tau|\asymp\frac{|U|N}{YT}\gg T^{3\varepsilon}$, $|\tau\pm t_\phi|\gg T^{3\varepsilon}$. By Stirling's formula \eqref{eqn:Stirling_J} and \eqref{eqn:Stirling_inverse}, we get
\begin{multline*}
\left(\frac{\prod_{\pm}\Gamma(\frac{1/2-i\tau\pm it_\phi}{2})}
{\prod_{\pm}\Gamma(\frac{1/2+i\tau\pm it_\phi}{2})}\right)^2
=\exp\Big(-i(\tau+t_\phi)\log\frac{|\tau+t_\phi|}{2e}
+i(-\tau+t_\phi)\log\frac{|-\tau+t_\phi|}{2e}
-i(\tau+t_\phi)\log\frac{|\tau+t_\phi|}{2e}\\
-i(\tau-t_\phi)\log\frac{|-\tau+t_\phi|}{2e}\Big)
\Big(g_1(-\tau-t_\phi)g_2(-\tau+t_\phi)g_3(\tau+t_\phi)(g_4(\tau-t_\phi)+O(T^{-A})\Big)^2
\end{multline*}
where $g_j (j=1, 2, 3, 4)$ depend on $A$ and satisfy $y^\ell\frac{\dd^{\ell}g_j(y)}{\dd y^\ell}\ll_{\ell, A} 1.$ Now we take 
$$g_{\phi, A}(\tau)=(g_1(-\tau-t_\phi)g_2(-\tau+t_\phi)g_3(\tau+t_\phi)(g_4(\tau-t_\phi))^2$$
which satisfies $\tau^\ell\frac{\dd^{\ell}g_{\phi, A}(\tau)}{\dd \tau^\ell}\ll_{\ell, A} 1$ for $|\tau|\gg T^{3\varepsilon}$. Inserting the above approximation into $\W_{+, F}$ and changing variable $\xi= -\frac{YT\tau}{2\pi NU}\asymp 1$, we get
\begin{equation}\label{eqn:W_+FOsi_Int}
\W_{+, F}(X; u, Y)=N\Big(\frac{X|U|}{YT}\Big)^{\frac{1}{2}}\int_{-\infty}^{\infty}W_4\Big(\frac{u}{U}, \xi\Big)g_{\phi, A}\Big(-\frac{2\pi NU}{YT}\xi\Big)
e\Big(h_1(\xi, \frac{u}{U})\Big)\dd \xi+O(T^{-A})
\end{equation}
by absorbing the multiplicative constant terms into $W_4$,
where the phase function, with $x_1=\frac{u}{U}\asymp 1$,
\begin{multline*}
h_1(\xi, x_1)=-\frac{NU}{YT}\xi\log \Big(\frac{\xi NX}{ex_1}\Big)-\frac{-\frac{2\pi NU}{YT}\xi+t_\phi}{\pi}\log \frac{\frac{2\pi N|U|}{YT}\xi-\sign(U)t_\phi}{2e}
+\frac{\frac{2\pi NU}{YT}\xi+t_\phi}{\pi}\log \frac{\frac{2\pi N|U|}{YT}\xi+\sign(U)t_\phi}{2e}\\
=\frac{NU}{YT}\xi\log \frac{\Big((\frac{2\pi N|U|}{YT}\xi)^2-t_\phi^2\Big)^2x_1}{16e^3NX\xi}+
\frac{t_\phi}{\pi}\log\frac{ \xi+\frac{\sign(U)YTt_\phi}{2\pi N|U|}}{\xi-\frac{\sign(U)YTt_\phi}{2\pi N|U|}}.
\end{multline*}
We have
\begin{equation*}
\begin{split}
\frac{\partial h_1(\xi, x_1)}{\partial \xi}
&=-\frac{NU}{YT}\log \Big(\frac{\xi NX}{x_1}\Big)+\frac{2 NU}{YT}\log \frac{\frac{2\pi N|U|}{YT}\xi-\sign(U)t_\phi}{2e}
+ \frac{2 NU}{YT}\log \frac{\frac{2\pi N|U|}{YT}\xi+\sign(U)t_\phi}{2e}
+\frac{4NU}{YT}\\
&=\frac{NU}{YT}\log\Big(\frac{\Big((\frac{2\pi N|U|}{YT}\xi)^2-t_\phi^2\Big)^2}{16}\frac{x_1}{\xi NX}\Big),
\end{split}
\end{equation*}
and 
\begin{equation*}
\begin{split}
\frac{\partial^{2} h_1(\xi, x_1)}{\partial \xi^2}=-\frac{NU}{YT}\frac{1}{\xi}+\frac{2 NU}{YT}\frac{1}{\xi-\frac{t_\phi YT}{2\pi NU}}
+\frac{2 NU}{YT}\frac{1}{\xi+\frac{t_\phi YT}{2\pi NU}}\asymp \frac{NU}{YT},
\end{split}
\end{equation*}
\[
\frac{\partial^{a+a_1} h_1(\xi, x_1)}{\partial \xi^a\partial x_1^{a_1}}\ll_{a, a_1} \frac{N|U|}{YT}
\]
for all $a, a_1\in \mathbb{Z}_{\geq 0}.$ 
If $X\not\asymp \frac{|U|^4N^3}{Y^4T^4}$, then $\frac{\partial h_1(\xi, x_1)}{\partial \xi}\gg \frac{N|U|}{YT}\gg T^{3\varepsilon}$. 
Since $W_4(x_1, \xi)g_{\phi, A}(-\frac{2\pi NU}{YT}\xi)$ is $T^\varepsilon$-inert in both $x_1$ and $\xi$, by repeated integration by parts, we have
\[
\W_{+, F}(X; u, Y)\ll_A T^{-A}.
\]
Now assume $X\asymp \frac{|U|^4N^3}{Y^4T^4}$. Note that $X\gg \frac{T^{\varepsilon_1}}{N}$, we have  $\frac{|U|N}{YT}>T^{3{\varepsilon}}$ automatically in this case.
Let the stationary point be $\xi_0$ which is real and satisfies the equation
\begin{equation}\label{eqn:stationary_point_xi0}
\frac{\Big((\frac{2\pi N|U|}{YT}\xi_0)^2-t_\phi^2\Big)^2}{16}\frac{x_1}{\xi_0 NX}=1.
\end{equation}
Set $\alpha=\frac{NX}{x_1}\Big(\frac{YT}{\pi N |U|}\Big)^4\asymp 1$ and $\gamma=\frac{t_\phi YT}{2\pi NU}\ll T^{-2\varepsilon},$ then \eqref{eqn:stationary_point_xi0} becomes 
\[
\Big(\Big(\frac{\xi_0}{\alpha^{1/3}}\Big)^2-\Big(\frac{\gamma}{\alpha^{1/3}}\Big)^2\Big)^2=\frac{\xi_0}{\alpha^{1/3}}.
\]
Note that we have the only unique real solution with $\xi_0\asymp 1$. By induction on the asymptotic expansion of the inverse function, we have
\begin{equation}\label{eqn:xi0_expansion}
  \frac{\xi_{0}}{\alpha^{1/3}}=1+\frac{2}{3}\Big(\frac{\gamma}{\alpha^{1/3}}\Big)^2
  -\frac{1}{3}\Big(\frac{\gamma}{\alpha^{1/3}}\Big)^4
  +\frac{28}{81}\Big(\frac{\gamma}{\alpha^{1/3}}\Big)^6
  -\frac{110}{243}\Big(\frac{\gamma}{\alpha^{1/3}}\Big)^8
  +\frac{2}{3}\Big(\frac{\gamma}{\alpha^{1/3}}\Big)^{10}
  +\cdots,
\end{equation}
and 
\begin{equation}\label{eqn:hxi0_expansion}
\begin{split}
h_1(&\xi_0, x_1)=\frac{NU}{YT}\Big(-3\xi_0+2\gamma\log \frac{\xi_0-\gamma}{\xi_0+\gamma}\Big)
=-\frac{NU}{YT}\alpha^{1/3}\Big(3\frac{\xi_0}{\alpha^{1/3}}+2\gamma\sum_{j\geq 0}\frac{2}{2j+1}\Big(\frac{\gamma }{\xi_0}\Big)^{2j+1}\Big)\\
&=-\frac{NU}{YT}\alpha^{1/3}\Big(3+4\gamma \frac{\gamma}{\alpha^{1/3}}+2\Big(\frac{\gamma }{\alpha^{1/3}}\Big)^{2}-\frac{4\gamma}{3}\Big(\frac{\gamma }{\alpha^{1/3}}\Big)^{3}
-\Big(\frac{\gamma }{\alpha^{1/3}}\Big)^{4}+\frac{56\gamma}{45}\Big(\frac{\gamma }{\alpha^{1/3}}\Big)^{5}\\
&\hspace{25em}+\frac{28}{27}\Big(\frac{\gamma }{\alpha^{1/3}}\Big)^{6}-\frac{880\gamma}{567}\Big(\frac{\gamma }{\alpha^{1/3}}\Big)^{7}+\cdots
\Big)\\
&=-\frac{NU}{YT}\Big(3\alpha^{1/3}+(4+2\alpha^{-1/3})\gamma^2-(4\alpha^{-2/3}/3+\alpha^{-1})\gamma^4+(56\alpha^{-4/3}/45+28\alpha^{-5/3}/27)\gamma^6+\cdots\Big)\\
&=-\frac{NU}{YT}\Big(\sum_{j\geq 0}(c_{1,j}\alpha^{-\frac{2j-2}{3}}+c_{2, j}\alpha^{-\frac{2j-1}{3}})\gamma^{2j}\Big).
\end{split}
\end{equation}
Applying Lemma \ref{lemma:stationary_phase} in \eqref{eqn:W_+FOsi_Int}, we have, with $h_1(\xi_0, x_1)$ in \eqref{eqn:hxi0_expansion},
\[
\W_{+, F}(X; u, Y)=(XN)^{\frac{1}{2}}e\Big(h_1(\xi_0, \frac{u}{U})\Big)W_5\Big(\frac{u}{U}\Big)+O(T^{-A})
\]
for some $T^\varepsilon$-inert function $W_5$. This finishes the proof of Lemma \ref{Lemma:Analysis_W+F}.
\end{proof}

Now we study the behaviour of $\frak{J}$. We take $\varepsilon_1=26\varepsilon$ and
let $X_1=\frac{n_1m^2}{r^4d_1d_2^2}\gg \frac{T^{\varepsilon_1}}{N}$, $X_2=\frac{n_1m^2}{r^4d_1d_2^2}\gg \frac{T^{\varepsilon_2}}{N}$ and $Y=\frac{rd_1d_2}{m}\asymp R$, then \eqref{eqn:frakJ} becomes
\begin{equation}\label{eqn:frakJXY}
  \frak{J}=\int_{-\infty}^{\infty}g_1\Big(\frac{u}{U}\Big)
  \W_{+, F}(X_1; u, Y)\overline{\W_{+, F}(X_2; u, Y)}
  \dd u.
\end{equation}
Then we have the following lemma.
\begin{lemma}\label{Lemma:Analysis_frakJ}
Let $\frak{J}$ defined as above with $X_1, X_2\gg \frac{T^{\varepsilon_1}}{N}$ and $Y\asymp R$, then
\[
\frak{J}\ll T^{-A},
\]
unless $X_1, X_2\asymp \frac{|U|^4N^3}{R^4T^4}$ and $|X_1-X_2|\ll \frac{N^2|U|^3}{R^3T^{3-\varepsilon}},$
in which case,
\begin{equation}\label{eqn:frakJ_expansion}
\frak{J}=N(X_1X_2)^{\frac{1}{2}}U \int_{-\infty}^{\infty}g_1(v)W_5(v)\overline{W_5(v)}
e\Big(f_J(n_2, r, v)-f_J(n_1, r, v)\Big) \dd v+O_{\varepsilon, J, A}(T^{-A}),
\end{equation}
where $J$ is a fixed large integer and 
\begin{equation}\label{eqn:f_J}
f_J(n, r, v)=\frac{NU}{YT}\sum_{0\leq j\leq J}\big(
c_{1, j}\alpha^{-\frac{2j-2}{3}}+
c_{2, j}\alpha^{-\frac{2j-1}{3}}\big)\gamma^{2j}
\end{equation} 
with $\alpha=(\frac{YT}{\pi N |U|})^4\frac{NX}{v}$, $\gamma=\frac{t_\phi YT}{2\pi NU}$, $X=\frac{nm^2}{r^4d_1d_2^2}$ and $Y=\frac{rd_1d_2}{m}.$
\end{lemma}
\begin{proof}
Note that $Y\asymp R$, by Lemma \ref{Lemma:Analysis_W+F}, we have $\frak{J}\ll UT^{-2A}\ll T^{-A}$ unless $X_1, X_2\asymp \frac{|U|^4N^3}{R^4T^4}$ in which case we have
\begin{equation}\label{eqn:frakJ_approx_int}
\begin{split}
\frak{J}=N(X_1X_2)^{\frac{1}{2}}U \int_{-\infty}^{\infty}&g_1(v)W_5(v)\overline{W_5(v)}\\
&\times e\Big(-\frac{NU}{YT}\sum_{j\geq 0}\big(
c_{1, j}(\alpha_1^{-\frac{2j-2}{3}}-\alpha_2^{-\frac{2j-2}{3}})+
c_{2, j}(\alpha_1^{-\frac{2j-1}{3}}-\alpha_2^{-\frac{2j-1}{3}})\big)\gamma^{2j}\Big) \dd v+O(T^{-A})
\end{split}
\end{equation}
by changing variable $u$ to $Uv$, where
$\alpha_1=(\frac{YT}{\pi N |U|})^4\frac{NX_1}{v}$, $\alpha_2=(\frac{YT}{\pi N |U|})^4\frac{NX_2}{v}$ and $\gamma=\frac{t_\phi YT}{2\pi NU}.$
Since $\frac{N|U|}{YT}\ll T^{O(1)}$, $\alpha_1, \alpha_2\asymp 1$ and $\gamma\ll T^{-2\varepsilon}$, we can truncate the phase series as
\begin{equation}\label{eqn:phase_f1-f2}
\begin{split}
-\frac{NU}{YT}&\sum_{j\geq 0}\big(
c_{1, j}(\alpha_1^{-\frac{2j-2}{3}}-\alpha_2^{-\frac{2j-2}{3}})+
c_{2, j}(\alpha_1^{-\frac{2j-1}{3}}-\alpha_2^{-\frac{2j-1}{3}})\big)\gamma^{2j}\\
&=\frac{NU}{YT}\sum_{0\leq j\leq J}\big(
c_{1, j}(\alpha_2^{-\frac{2j-2}{3}}-\alpha_1^{-\frac{2j-2}{3}})+
c_{2, j}(\alpha_2^{-\frac{2j-1}{3}}-\alpha_1^{-\frac{2j-1}{3}})\big)\gamma^{2j}
+O_{\varepsilon, J, A}(T^{-A})\\
&=f_J(n_2, r, v)-f_J(n_1, r, v)+O_{\varepsilon, J, A}(T^{-A}),
\end{split}
\end{equation}
by taking $J$ sufficiently large. Here by inserting $X_1=\frac{n_1m^2}{r^4d_1d_2^2}$, $X_2=\frac{n_1m^2}{r^4d_1d_2^2}, Y=\frac{rd_1d_2}{m},$ we get
the function
\[
f_J(n, r, v)=\frac{NU}{YT}\sum_{0\leq j\leq J}\big(
c_{1, j}\alpha^{-\frac{2j-2}{3}}+
c_{2, j}\alpha^{-\frac{2j-1}{3}}\big)\gamma^{2j}.
\]
Putting it back to \eqref{eqn:frakJ_approx_int}, we get
\begin{equation}\label{eqn:frakJ_h2}
\begin{split}
\frak{J}=N(X_1X_2)^{\frac{1}{2}}U \int_{-\infty}^{\infty}&g_1(v)W_5(v)\overline{W_5(v)}e\Big(h_{2, J}(v)\Big) \dd v+O(T^{-A}),
\end{split}
\end{equation}
where
\[
h_{2, J}(v)=\frac{NU}{YT}\sum_{0\leq j\leq J}\big(
c_{1, j}(\alpha_2^{-\frac{2j-2}{3}}-\alpha_1^{-\frac{2j-2}{3}})+
c_{2, j}(\alpha_2^{-\frac{2j-1}{3}}-\alpha_1^{-\frac{2j-1}{3}})\big)\gamma^{2j},
\]
with $\alpha_1=(\frac{YT}{\pi N |U|})^4\frac{NX_1}{v}$, $\alpha_2=(\frac{YT}{\pi N |U|})^4\frac{NX_2}{v}$ and $\gamma=\frac{t_\phi YT}{2\pi NU}.$
This is also \eqref{eqn:frakJ_expansion}. 
For the phase function, we have
\[
h_{2, J}'(v)=\frac{NU}{YT}\Big(-\frac{\alpha_2^{1/3}-\alpha_1^{1/3}}{v}+\sum_{1\leq j\leq J}\big(
c_{1, j}'\frac{\alpha_2^{-\frac{2j-2}{3}}-\alpha_1^{-\frac{2j-2}{3}}}{v}+
c_{2, j}'\frac{\alpha_2^{-\frac{2j-1}{3}}-\alpha_1^{-\frac{2j-1}{3}}}{v}\big)\gamma^{2j}\Big)
\]
and 
\[
h_{2, J}''(v)=\frac{NU}{YT}\Big(\frac{4(\alpha_2^{1/3}-\alpha_1^{1/3})}{3v^2}+\sum_{1\leq j\leq J}\big(
c_{1, j}''\frac{\alpha_2^{-\frac{2j-2}{3}}-\alpha_1^{-\frac{2j-2}{3}}}{v^2}+
c_{2, j}''\frac{\alpha_2^{-\frac{2j-1}{3}}-\alpha_1^{-\frac{2j-1}{3}}}{v^2}\big)\gamma^{2j}\Big)
\]
with certain constants $c_{1, j}', c_{2, j}', c_{1, j}'', c_{2, j}''.$
Here one may focus on the leading term of $j=0$ since $\gamma^{2j}$ save additional power of $T$ if $j\geq 1$.
Note that $\alpha_1, \alpha_2\asymp 1$,  for $\ell\neq 0$, we have 
$|\alpha_2^\ell-\alpha_1^\ell|\asymp_\ell |\alpha_2-\alpha_1|.$
If $|\alpha_2-\alpha_1|\gg T^\varepsilon\frac{YT}{N|U|},$ we have, for $v\asymp 1$, 
$
|h_{2, J}'(v)|\asymp  \frac{N|U|}{YT}|\alpha_2-\alpha_1|\gg T^\varepsilon.
$
Similarly, for any integer $k\geq 2$, $
|h_{2, J}^{(k)}(v)|\asymp_k \frac{N|U|}{YT}|\alpha_2-\alpha_1|.
$
Applying Lemma \ref{lemma:repeated_integration_by_parts} in \eqref{eqn:frakJ_h2}, we have
$\frak{J}\ll T^{-A}$. 
Hence we have the generic case $|\alpha_2-\alpha_1|\ll T^\varepsilon\frac{YT}{N|U|}\asymp  T^\varepsilon\frac{RT}{N|U|},$ which is equivalent to 
$|X_1-X_2|\ll \frac{N^2|U|^3}{R^3T^{3-\varepsilon}}.$
This completes the proof.
\end{proof}

By Lemma \ref{Lemma:Analysis_frakJ}, to bound $\mathcal{J}_0(d_1, d_2)$ \eqref{eqn:J_0(d1,d2)}, it suffices to estimate
\begin{multline}\label{eqn:bound_J0}
NU\int_{-\infty}^{\infty}g_1(v)W_5(v)\overline{W_5(v)}\sum_{r\sim \frac{Rm}{d_1d_2}}\,\sideset{}{^*}\sum_{x\mod r}
\mathop{\sum\sum}_{n_1, n_2\sim N_2\atop |n_1-n_2|\ll \frac{N^2|U|^3Rm^2}{d_1^3d_2^2T^3}T^\varepsilon}
\frac{A_F(d_1,d_2, n_1)\overline{A_F(d_1,d_2, n_2)}}{\sqrt{n_1n_2}}\frac{m^2}{r^4d_1d_2^2}\\
\times
e\Big(\frac{(n_1-n_2)x}{r}\Big)
e\Big(f_J(n_2, r, v)-f_J(n_1, r, v)\Big) \dd v
\end{multline}
in the case of 
\[
 \frac{N_2m^2}{r^4d_1d_2^2}\asymp \frac{|U|^4N^3}{R^4T^4}.
\]
This restriction is equivalent to $|U|\asymp \frac{TN_2^{\frac{1}{4}}d_1^{\frac{3}{4}}d_2^{\frac{1}{2}}}{N^{\frac{3}{4}}m^{\frac{1}{2}}}.$
Thus \eqref{eqn:bound_J0} becomes
\begin{multline}\label{eqn:bound_J0_again}
\frac{TN^{\frac{1}{4}} N_2^{\frac{1}{4}}m^{\frac{3}{2}}}
{d_1^{\frac{1}{4}}d_2^{\frac{3}{2}}}
\int_{-\infty}^{\infty}g_1(v)W_5(v)\overline{W_5(v)}\sum_{r\sim \frac{Rm}{d_1d_2}}\frac{1}{r^4}\,\sideset{}{^*}\sum_{x\mod r}
\mathop{\sum\sum}_{n_1, n_2\sim N_2\atop |n_1-n_2|\ll \frac{Rm^{\frac{1}{2}}N_2^{\frac{3}{4}}}{N^{\frac{1}{4}}d_1^{\frac{3}{4}}d_2^{\frac{1}{2}}}T^\varepsilon}
\frac{A_F(d_1,d_2, n_1)\overline{A_F(d_1,d_2, n_2)}}{\sqrt{n_1n_2}}\\
\times
e\Big(\frac{(n_1-n_2)x}{r}\Big)
e\Big(f_J(n_2, r, v)-f_J(n_1, r, v)\Big) \dd v
\end{multline}
Let $\mathcal{R}=\frac{Rm^{\frac{1}{2}}}{N_2^{\frac{1}{4}}N^{\frac{1}{4}}d_1^{\frac{3}{4}}d_2^{\frac{1}{2}}}$ which is $\asymp \frac{YT}{N|U|}$ with $Y=\frac{rd_1d_2}{m}.$ We have
\begin{equation}\label{eqn:|n_1-n_2|small}
|n_1-n_2|\ll\mathcal{R}N_2T^\varepsilon,
\end{equation}
and 
\[
\mathcal{R}\ll \frac{Rm^{\frac{1}{2}}}{N^{\frac{1}{4}}d_1^{\frac{3}{4}}d_2^{\frac{1}{2}}}\Big(\frac{R^4m^2}{Nd_2^2d_1^3}T^{\varepsilon_1}\Big)^{-\frac{1}{4}}\ll T^{-\frac{\varepsilon_1}{4}}.
\]

Next we divide the range of $n_1, n_2$ into the segment of $\mathcal{C}_{\eta_1}$ and $\mathcal{C}_{\eta_2}$ of length $\mathcal{R}N_2T^{\varepsilon-\varepsilon_1}$
where $\eta_1$ and $\eta_2$ are the left endpoints of intervals $\mathcal{C}_{\eta_1}$ and $\mathcal{C}_{\eta_2}$ respectively. When $n_1\in \mathcal{C}_{\eta_1}$, 
$n_2\in \mathcal{C}_{\eta_2}$ and \eqref{eqn:|n_1-n_2|small}, the restriction of length of the intervals implies $|\eta_1-\eta_2|\ll\mathcal{R}N_2T^\varepsilon$.
Hence for fixed $\mathcal{C}_{\eta_1}$, there are $O(T^{\varepsilon_1})$ choice of $\mathcal{C}_{\eta_2}$. The $n_1$-sum is of length $N_2$, thus there are 
$O(\frac{N_2}{\mathcal{R}N_2T^{\varepsilon-\varepsilon_1}}T^\varepsilon)=O(\frac{T^{\varepsilon_1}}{\mathcal{R}})$ relevant pairs of $(\mathcal{C}_{\eta_1}, \mathcal{C}_{\eta_2})$
with end points satisfying $|\eta_1-\eta_2|\ll\mathcal{R}N_2T^\varepsilon$. We let $\sideset{}{^{\rel}}\sum_{(\mathcal{C}_{\eta_1}, \mathcal{C}_{\eta_2})}$ denote the sum over such pairs. 

From \eqref{eqn:bound_J0_again}, we have
\begin{equation}\label{eqn:J0boundby_S}
\begin{split}
\mathcal{J}_0(d_1, d_2)&\ll\frac{TN^{\frac{1}{4}} N_2^{\frac{1}{4}}m^{\frac{3}{2}}}
{d_1^{\frac{1}{4}}d_2^{\frac{3}{2}}}
\int_{-\infty}^{\infty}\Big|g_1(v)W_5(v)\overline{W_5(v)}\Big|\sum_{r\sim \frac{Rm}{d_1d_2}}\frac{1}{r^4}\,\sideset{}{^*}\sum_{x\mod r}
\quad \sideset{}{^{\rel}}\sum_{(\mathcal{C}_{\eta_1}, \mathcal{C}_{\eta_2})}|\mathcal{S}(\mathcal{C}_{\eta_1})\overline{\mathcal{S}(\mathcal{C}_{\eta_2})}|\dd v\\
&\ll \frac{TN^{\frac{1}{4}} N_2^{\frac{1}{4}}d_1^{\frac{15}{4}}d_2^{\frac{5}{2}}}
{R^4m^{\frac{5}{2}}}
\int_{v\asymp 1}\quad \sideset{}{^{\rel}}\sum_{(\mathcal{C}_{\eta_1}, \mathcal{C}_{\eta_2})}\sum_{r\sim \frac{Rm}{d_1d_2}}\,\sideset{}{^*}\sum_{x\mod r}
\Big(|\mathcal{S}(\mathcal{C}_{\eta_1})|^2+|\mathcal{S}(\mathcal{C}_{\eta_2})|^2\Big)\dd v,\\
\end{split}
\end{equation}
where 
\[
\mathcal{S}(\mathcal{C}_{\eta_j}):=\sum_{n\in \mathcal{C}_{\eta_j}}
\frac{A_F(d_1,d_2, n)}{\sqrt{n}}
e\Big(-\frac{nx}{r}\Big)
e\Big(f_J(n, r, v)\Big)
\]
for  $j=1, 2.$ To bound $\mathcal{J}_0(d_1, d_2)$, we first prove the following lemma.
\begin{lemma}\label{lemma:applyingLS}
Let $\eta=\eta_1$ or $\eta_2$ and $\mathcal{S}(\mathcal{C}_{\eta})$ be defined as above. We have
\[
\sum_{r\sim \frac{Rm}{d_1d_2}}\,\sideset{}{^*}\sum_{x\mod r}
|\mathcal{S}(\mathcal{C}_{\eta})|^2\ll \Big(\frac{1}{N_2}\Big(\frac{Rm}{d_1d_2}\Big)^2+\mathcal{R}\Big)\frak{S}(\eta),
\]
where 
\[
\frak{S}(\eta):=\sum_{n\in \mathcal{C}_{\eta}}|A_F(d_1, d_2, n)|^2.
\]
\end{lemma}

\begin{proof}
Let $e\Big(f_J(n, r, v)\Big)=F_1(n)F_2(n)$ with
\[
F_1(n)=e\Big(\frac{NU}{YT}\sum_{1\leq j\leq J}
c_{1, j}\alpha^{-\frac{2j-2}{3}}\gamma^{2j}\Big)=\prod_{1\leq j\leq J}e\Big(B_{1, j} n^{-\frac{2j-2}{3}}\Big),
\]
\[
F_2(n)=e\Big(\frac{NU}{YT}\sum_{0\leq j\leq J}
c_{2, j}\alpha^{-\frac{2j-1}{3}}\gamma^{2j}\Big)=\prod_{0\leq j\leq J}e\Big(B_{2, j} n^{-\frac{2j-1}{3}}\Big),
\]
where
\[
B_{1, j}=c_{1, j}\frac{NU}{YT}\Big(\big(\frac{YT}{\pi N |U|}\big)^4\frac{Nm^2}{r^4d_1d_2^2}\Big)^{-\frac{2j-2}{3}}\gamma^{2j}, \quad
B_{2, j}=c_{2, j}\frac{NU}{YT}\Big(\big(\frac{YT}{\pi N |U|}\big)^4\frac{Nm^2}{r^4d_1d_2^2}\Big)^{-\frac{2j-1}{3}}\gamma^{2j}
\]
and 
\[
B_{1, j}\ll_{j, v} t_\phi^{2j}\Big(\frac{YT}{N|U|}\Big)^{2j-1}N_2^{\frac{2j-2}{3}}, \quad
B_{2, j}\ll_{j, v} t_\phi^{2j}\Big(\frac{YT}{N|U|}\Big)^{2j-1}N_2^{\frac{2j-1}{3}},
\]
both of them are independent of $n$. Moreover $B_{2, 0}\asymp \frac{N|U|}{YT}N_2^{-\frac{1}{3}}$.
Note that $t_\phi\ll T^\varepsilon$ and $\frac{YT}{N|U|}\ll T^{-2\varepsilon}$, hence for $j\geq 1$
\[
B_{1, j}\ll_{j, v} T^{2(1-j)\varepsilon}N_2^{\frac{2j-2}{3}}, \quad
B_{2, j}\ll_{j, v} T^{2(1-j)\varepsilon}N_2^{\frac{2j-1}{3}}.
\]
We will take a Taylor expansion of $F_1(n)F_2(n)$ at $n=\eta$ to separate the variables $r$ and $n$ before using large sieve.
We write
\[
F_1(n)F_2(n)-F_1(\eta)F_2(\eta)=\sum_{k\geq 1}\frac{1}{k!}(F_1F_2)^{(k)}(\eta)\big(n-\eta\big)^k,
\]
To bound the derivatives of $F_1$ and $F_2$, it suffices to consider the derivatives of each component $e(B_{1, j} n^{-\frac{2j-2}{3}})$ and $e(B_{2, j} n^{-\frac{2j-1}{3}}).$
For any integer $\ell, \ell_1, \ell_2\geq 1$, with $n\sim N_2$,  we have
\[
\frac{\dd^{\ell} e(B_{2, 0} n^{\frac{1}{3}})}{\dd n^{\ell}}\ll_{\ell, v} N_2^{-\ell}\Big(B_{2, 0} N_2^{\frac{1}{3}}+(B_{2, 0} N_2^{\frac{1}{3}})^{\ell}\Big)\ll 
 N_2^{-\ell}\Big(\frac{N|U|}{YT}\Big)^\ell,
\]
and for $j\geq 1$, 
\[
\frac{\dd^{\ell_1} e(B_{1, j} n^{-\frac{2j-2}{3}})}{\dd n^{\ell_1}}\ll_{\ell_1} N_2^{-\ell_1}\Big(B_{1, j} N_2^{-\frac{2j-2}{3}}+(B_{1, j} N_2^{-\frac{2j-2}{3}})^{\ell_1}\Big)
\ll_{\ell_1, v}N_2^{-\ell_1}T^{2(1-j)\varepsilon},
\]
\[
\frac{\dd^{\ell_2} e(B_{2, j} n^{-\frac{2j-1}{3}})}{\dd n^{\ell_2}}\ll_{\ell_2} N_2^{-\ell_2}\Big(B_{2, j} N_2^{-\frac{2j-1}{3}}+(B_{2, j} N_2^{-\frac{2j-1}{3}})^{\ell_2}\Big)
\ll_{\ell_2, v}N_2^{-\ell_2}T^{2(1-j)\varepsilon}.
\]
Therefore, for $k\geq 1$, we have
\begin{equation*}
\begin{split}
(F_1F_2)^{(k)}(\eta)&\ll_k
\max_{\kappa_1+\kappa_2=k\atop \kappa_1,\kappa_2\geq 0}|F_1^{(\kappa_1)}(\eta)F_2^{(\kappa_2)}(\eta)|\\
\ll_{\kappa, J} &\max_{\kappa_1+\kappa_2=k\atop \kappa_1,\kappa_2\geq 0}\max_{\ell_{1, 1}+\cdots+\ell_{1, J}=\kappa_1\atop\ell_{1, 1},\cdots,\ell_{1, J}\geq 0}
\max_{\ell_{2, 0}+\cdots+\ell_{2, J}=\kappa_2\atop\ell_{2, 0},\cdots,\ell_{2, J}\geq 0}
\left(\prod_{1\leq j\leq J}\Big|\frac{\dd^{\ell_{1, j}} e(B_{1, j} n^{-\frac{2j-2}{3}})}{\dd n^{\ell_{1, j}}}\Big|
\prod_{0\leq j\leq J}\Big|\frac{\dd^{\ell_{2, j}} e(B_{2, j} n^{-\frac{2j-1}{3}})}{\dd n^{\ell_{2, j}}}\Big|\right)_{n=\eta}\\
\ll_{\kappa, J} &\, N_2^{-k}\Big(\frac{N|U|}{YT}\Big)^k\asymp (N_2\mathcal{R})^{-k}.
\end{split}
\end{equation*}
The above bound comes from the derivatives of the dominated phase $e(B_{2, 0} n^{\frac{1}{3}}).$
And trivially  $(F_1F_2)^{(0)}(\eta)=(F_1F_2)(\eta)\ll 1.$
Therefore for $n\in \mathcal{C}_\eta,$
\[
\frac{1}{k!}(F_1F_2)^{(k)}(\eta)\big(n-\eta\big)^k\ll_{k}(N_2\mathcal{R})^{-k}(\mathcal{R}N_2T^{\varepsilon-\varepsilon_1})^k\ll T^{(\varepsilon-\varepsilon_1)k}.
\]
We truncate the Taylor expansion of $F_1F_2$ to get
\[
F_1(n)F_2(n)-F_1(\eta)F_2(\eta)=\sum_{1\leq k\leq K}\frac{1}{k!}(F_1F_2)^{(k)}(\eta)\big(n-\eta\big)^k+O_K(T^{-2025}),
\]
by taking $K$ sufficiently large. Putting this in $\mathcal{S}(\mathcal{C}_\eta)$, it suffices to bound
\[
\sum_{r\sim \frac{Rm}{d_1d_2}}\,\sideset{}{^*}\sum_{x\mod r}
\Big|\sum_{n\in \mathcal{C}_{\eta}}
\frac{A_F(d_1,d_2, n)}{\sqrt{n}}
e\Big(-\frac{nx}{r}\Big)(F_1F_2)^{(k)}(\eta)(n-\eta)^k\Big|^2
\]
for all $0\leq k\leq K$. This is
\begin{equation*}
\begin{split}
\ll &\sum_{r\sim \frac{Rm}{d_1d_2}}\,\Big|(F_1F_2)^{(k)}(\eta)\Big|^2\sideset{}{^*}\sum_{x\mod r}
\Big|\sum_{n\in \mathcal{C}_{\eta}}
\frac{A_F(d_1,d_2, n)}{\sqrt{n}}
e\Big(-\frac{nx}{r}\Big)(n-\eta)^k\Big|^2\\
\ll & (N_2\mathcal{R})^{-2k}\sum_{r\sim \frac{Rm}{d_1d_2}}\,\sideset{}{^*}\sum_{x\mod r}
\Big|\sum_{n\in \mathcal{C}_{\eta}}
\frac{A_F(d_1,d_2, n)}{\sqrt{n}}
e\Big(-\frac{nx}{r}\Big)(n-\eta)^k\Big|^2\\
\ll& (N_2\mathcal{R})^{-2k}\Big(\Big(\frac{Rm}{d_1d_2}\Big)^2+\mathcal{R}N_2T^{\varepsilon-\varepsilon_1}\Big)
\sum_{n\in \mathcal{C}_{\eta}}
\frac{|A_F(d_1,d_2, n)|^2}{n}|n-\eta|^{2k},
\end{split}
\end{equation*}
by using large sieve inequality. Since $n\in \mathcal{C}_\eta$, $n\sim N_2$, we have $|n-\eta|\ll \mathcal{R}N_2T^{\varepsilon-\varepsilon_1}.$
This implies the desired bound $(\frac{1}{N_2}(\frac{Rm}{d_1d_2})^2+\mathcal{R})\frak{S}(\eta)$ in Lemma \ref{lemma:applyingLS}.
\end{proof}

By Lemma \ref{lemma:applyingLS} and \eqref{eqn:J0boundby_S}, we have 
\begin{equation*}
\begin{split}
\mathcal{J}_0(d_1, d_2)&\ll \frac{TN^{\frac{1}{4}} N_2^{\frac{1}{4}}d_1^{\frac{15}{4}}d_2^{\frac{5}{2}}}
{R^4m^{\frac{5}{2}}}\Big(\frac{1}{N_2}\Big(\frac{Rm}{d_1d_2}\Big)^2+\mathcal{R}\Big)\quad \sideset{}{^{\rel}}\sum_{(\mathcal{C}_{\eta_1}, \mathcal{C}_{\eta_2})}
(\frak{S}(\eta_1)+\frak{S}(\eta_2))\\
&\ll \frac{TN^{\frac{1}{4}} N_2^{\frac{1}{4}}d_1^{\frac{15}{4}}d_2^{\frac{5}{2}}}
{R^4m^{\frac{5}{2}}}\Big(\frac{1}{N_2}\Big(\frac{Rm}{d_1d_2}\Big)^2+\mathcal{R}\Big)T^{\varepsilon_1}
\sum_{n\sim N_2}|A_F(d_1, d_2, n)|^2,
\end{split}
\end{equation*}
since for each $\mathcal{C}_{\eta_1}$ there are $O(T^{\varepsilon_1})$ choice for $\mathcal{C}_{\eta_2}$ and vice versa.
Then combining the bound in Lemma \ref{lemma:RamanujanOnAverage2} and $N_2, d_1, d_2\ll T^{O(1)}$, $\mathcal{R}=\frac{Rm^{\frac{1}{2}}}{N_2^{\frac{1}{4}}N^{\frac{1}{4}}d_1^{\frac{3}{4}}d_2^{\frac{1}{2}}}$,we get
\begin{equation*}
\begin{split}
\mathcal{J}_0(d_1, d_2)&\ll T^{\varepsilon+\varepsilon_1}\frac{TN^{\frac{1}{4}} N_2^{\frac{1}{4}}d_1^{\frac{15}{4}}d_2^{\frac{5}{2}}}
{R^4m^{\frac{5}{2}}}\Big(\frac{1}{N_2}\Big(\frac{Rm}{d_1d_2}\Big)^2+\mathcal{R}\Big)(d_1d_2N_2)^{1+\varepsilon}T^\varepsilon\\
&\ll T^{\varepsilon}\Big(\frac{TN^{\frac{1}{4}} N_2^{\frac{1}{4}}d_1^{\frac{11}{4}}d_2^{\frac{3}{2}}}
{R^2m^{\frac{1}{2}}}+\frac{TN_2d_1^{4}d_2^{3}}
{R^3m^{2}}\Big).
\end{split}
\end{equation*}
Here we gather all arbitrarily small power of $T$ and replace it by $O(T^\varepsilon)$ with $\varepsilon$ sufficiently small for simplifying the notation.
From \eqref{eqn:I_bigBoundbyJ} and \eqref{eqn:J_boundbyJ_0} with $ \frac{TN_2^{\frac{1}{4}}d_1^{\frac{3}{4}}d_2^{\frac{1}{2}}}{N^{\frac{3}{4}}m^{\frac{1}{2}}}\asymp|U|\leq T^\varepsilon$ 
and $\frac{R^4m^2}{Nd_2^2d_1^3}T^{\varepsilon_1}\leq N_2\leq T^{2025}$ which imply 
$N_2\ll T^{4\varepsilon}\frac{N^3m^2|U|^4}{T^4d_1^3d_2^2}$, we get
\begin{equation*}
\begin{split}
\mathcal{I}_{big}&\ll \sup_{N_2}
\frac{T^{1+\varepsilon}R^3m^2}{N}\frac{1}{|U|}\mathop{\sum\sum}_{d_1, d_2\ll Rm\atop d_1^{3}d_2^4\asymp \frac{|U|^4N^3m^2}{T^4}N_2}\frac{1}{d_1^3d_2^2}
\Big(\frac{TN^{\frac{1}{4}} N_2^{\frac{1}{4}}d_1^{\frac{11}{4}}d_2^{\frac{3}{2}}}
{R^2m^{\frac{1}{2}}}+\frac{TN_2d_1^{4}d_2^{3}}
{R^3m^{2}}\Big)\\
&\ll \sup_{N_2}\frac{T^{1+\varepsilon}R^3m^2}{N}\frac{1}{|U|}
\mathop{\sum\sum}_{d_1, d_2\ll Rm\atop d_1^{3}d_2^4\asymp \frac{|U|^4N^3m^2}{T^4}N_2}
\Big(\frac{TN^{\frac{1}{4}} N_2^{\frac{1}{4}}}
{R^2m^{\frac{1}{2}}d_1^{\frac{1}{4}}d_2^{\frac{1}{2}}}+\frac{TN_2d_1d_2}
{R^3m^{2}}\Big)\\
&\ll \frac{T^{1+\varepsilon}R^3m^2}{N}\frac{1}{|U|}\mathop{\sum\sum}_{d_1, d_2\ll Rm}
\Big(\frac{N|U|}{R^2d_1d_2}
+\frac{N^3|U|^4}{T^3R^3d_1^2d_2}\Big)\\
&\ll T^\varepsilon\Big(Tm^2R+\frac{N^2m^2}{T^2}|U|^3\Big)\ll T^{2+\varepsilon},
\end{split}
\end{equation*}
where we have used that $|U|\ll T^{\varepsilon}$, $R\ll \frac{N}{T}$ and $N\ll \frac{T^{2+\varepsilon}}{m^2}$. This completes  the proof of Proposition \ref{Prop:I_big}.


\end{document}